\documentclass[reqno,a4paper]{amsart}

\usepackage[T1]{fontenc}
\usepackage[english]{babel}

\usepackage[colorinlistoftodos]{todonotes}
\usepackage[left=2.8cm,right=2.8cm,top=2.8cm,bottom=2.8cm]{geometry}
\allowdisplaybreaks

\usepackage{enumitem}

\usepackage{amsmath}
\usepackage{amssymb,amsthm,amsfonts}
\usepackage{mathrsfs,mathtools,braket,esint}
\usepackage{dsfont,mathscinet,xfrac,comment}

\numberwithin{equation}{section}
\theoremstyle{plain}
\newtheorem{theorem}{Theorem}[section]
\newtheorem*{theorem*}{Theorem}
\newtheorem{proposition}[theorem]{Proposition}
\newtheorem{lemma}[theorem]{Lemma}
\newtheorem{example}[theorem]{Example}
\newtheorem{corollary}[theorem]{Corollary}

\theoremstyle{definition}
\newtheorem{definition}[theorem]{Definition}
\newtheorem{remark}[theorem]{Remark}
\newtheorem{assumption}{Assumption}

\newcommand{\call}[1]{\mathscr{#1}}
\newcommand{\A}{\call{A}}
\newcommand{\B}{\call{B}}

\newcommand{\E}{\call{E}}
\newcommand{\F}{\call{F}}
\newcommand{\G}{\call{G}}

\newcommand{\bA}{\mathbb{A}}
\newcommand{\bB}{\mathbb{B}}
\newcommand{\bG}{\mathbb{G}}

\newcommand{\eps}{\varepsilon}
\renewcommand{\epsilon}{\varepsilon}
\newcommand\ep{\varepsilon}
\renewcommand{\theta}{\vartheta}
\newcommand{\sfPi}{\mathsf{\Pi}}
\newcommand{\N}{\mathbb{N}}

\newcommand{\Z}{\mathbb{Z}}
\newcommand{\Zd}{\mathbb{Z}^d}
\newcommand{\R}{\mathbb{R}}
\newcommand{\C}{\mathbb{C}}
\newcommand{\Rd}{\R^d}
\newcommand{\Rdmz}{\R^d\setminus\set{0}}
\newcommand{\RN}{\R^N}
\newcommand{\RM}{\R^M}
\newcommand{\Td}{\mathbb{T}^d}
\newcommand{\wto}{{\rightharpoonup}}

\DeclarePairedDelimiter{\va}{\lvert}{\rvert}

\DeclarePairedDelimiter{\norm}{\lVert}{\Vert}
\newcommand{\dist}{\mathrm{dist}}
\newcommand{\de}{\mathrm{d}}

\newcommand{\pder}[2]{\frac{\partial #1}{\partial #2}}
\DeclareMathOperator{\dvg}{div}
\renewcommand{\div}{\dvg}
\newcommand{\Ld}{\call{L}^d}

\DeclareMathOperator{\supp}{supp}

\newcommand{\weak}{\rightharpoonup}
\newcommand\wk{\rightharpoonup}

\newcommand{\weakts}{\stackrel{2}{\rightharpoonup}}
\newcommand{\strongts}{\stackrel{2}{\to}}
\newcommand{\Cc}{C_\mathrm{c}}
\newcommand{\loc}{\mathrm{loc}}
\newcommand{\per}{\mathrm{per}}
\renewcommand{\hom}{\mathrm{hom}}
\DeclareMathOperator{\im}{im}

\usepackage{cancel}
\usepackage{color}
\usepackage{graphicx}
    \newcommand\BBB{\color{blue}}
    \newcommand\EEE{\color{black}}
    
    \newcommand\RRR{\color{black}}
\usepackage{tikz}
\usepackage{epstopdf}
\usepackage{subfig}
\usepackage{comment}
\usepackage[colorlinks=true, pdfstartview=FitV, linkcolor=blue,citecolor=blue, urlcolor=blue]{hyperref}

\title{Homogenization~of~high-contrast~composites under~differential~constraints}
\author[E. Davoli]{Elisa Davoli}
\address[E. Davoli]{
		Institute of Analysis and Scientific Computing,
		TU Wien,
		Wiedner Hauptstrasse 8-10, 1040 Vienna, Austria.
		Email: \href{mailto:elisa.davoli@tuwien.ac.at}{\tt elisa.davoli@tuwien.ac.at}.
}
\author[M. Kru\v{z}\'ik]{Martin Kru\v{z}\'ik}
\address[M. Kru\v{z}\'ik]{Czech Academy of Sciences, 
Institute of Information Theory and Automation, 
Pod Vod\'arenskou V\v{e}\v{z}\'i 4, 180 00 Prague, Czech Republic.
		Email: \href{mailto: kruzik@utia.cas.cz}{\tt kruzik@utia.cas.cz}.
}
\author[V. Pagliari]{Valerio Pagliari}
\address[V. Pagliari]{
		Institute of Analysis and Scientific Computing,
		TU Wien,
		Wiedner Hauptstrasse 8-10, 1040 Vienna, Austria.
		Email: \href{mailto:valerio.pagliari@tuwien.ac.at}{\tt valerio.pagliari@tuwien.ac.at}.
}
\date{\today}
\keywords{homogenization, high-contrast, $\mathscr{A}$-quasiconvexity, $\Gamma$-convergence, two-scale convergence, periodic unfolding}
\subjclass[2000]{49J45; 35D99; 49K20; 74Q99}

\begin{document}

\maketitle

\begin{abstract}
We derive, by means of variational techniques, a limiting description 
for a class of integral functionals
under linear differential constraints.
The functionals are designed  to encode the energy of a high-contrast composite, that is,
a heterogeneous material which, at a microscopic level, consists of a periodically perforated matrix
whose cavities are occupied by a filling with very different physical properties.
Our main result provides a $\Gamma$-convergence analysis as the periodicity tends to zero, and
 shows that
the variational limit of the functionals at stake is the sum of two contributions,
one resulting from the energy stored in the matrix and
the other from the energy stored in the inclusions.
As a consequence of the underlying high-contrast structure,
the study is faced with a lack of coercivity with respect to the standard topologies in $L^p$,
which we tackle by means of two-scale convergence techniques.
In order to handle the differential constraints, instead,
we establish new results about the existence of potentials and of constraint-preserving extension operators
for linear, $k$-th order, homogeneous differential operators
with constant coefficients and constant rank.
\end{abstract}

{\parskip=0pt \tableofcontents}
\section{Introduction}

The identification of materials capable
to switch, enhance, or tune their properties according to external stimuli
is one of the key catalysts
for the creation of new artificially engineered composites: the metamaterials.
In the present paper, 
we contribute to the mathematical study of metamaterials
by providing a unified limiting description for a class of energies
that are inspired by the physics of high-contrast composites and 
are evaluated on fields subject to suitable partial differential constraints.

High-contrast materials are characterized by the property that
their microscopic physical features may change abruptly from point to point.
In the case of a binary periodic medium of this kind,
we have two relevant microscales:
the periodicity of the microstructure and the high-contrast parameter,
which encodes the strong (high-contrast) difference in the properties of the two components.
Typically, the scale at which periodicity is observed is much smaller than the size of the specimen.
It is hence natural, from a mathematical viewpoint,
to study the asymptotics of the quantities that describe the system
in the limit of vanishing period.
This procedure is called homogenization.
It allows to make predictions on the effective character of heterogeneous media
when no experiments are available,
and eases numerical simulations
by reducing the degrees of freedom of the problem.
Therefore, mathematical homogenization has been a lively field
of investigation for many decades.
We refer e.g.~to the monographs \cite{bakhvalov.panasenko,cioranescu.donato,zhikov.kozlov.olejnik}
for a thorough introduction to the subject. 

\RRR
In this paper, we tackle the case in which
both the periodicity and the high-contrast are quantified by a small number $\ep>0$
(see \eqref{eq:ep-funct} and Remark \ref{stm:degeneracy}).
\EEE
%
We assume that the vector fields that encode the significant quantities of the system
satisfy a linear first-order differential constraint.
The rationale behind this choice becomes transparent
if one focuses on specific theories:
in large strain elasticity
the deformation gradient is a significant strain measure, while
electromagnetism deals with the differential operators encoded by the Maxwell system. 
In other settings,
the elastic behavior of some materials can be influenced
by electric currents or magnetic fields and vice versa
(see e.g.~\cite{silhavy} for a recent contribution). 
In all these cases,
the key quantities lie in the kernel of a suitable differential operator.
We therefore see that
the need of combining homogenization and differential constraints naturally arises and leads to generalized notions of convexity (see \cite{benesova.kruzik} for an overview).
It is well-known that all the differential constraints above (and many more) can be treated in a unified way
in the framework of 
$\A$-quasiconvexity
(see \cite{fonseca.muller} and references therein).
Such a notion allows to extend various results
which were originally available only for Sobolev maps and their gradients
to the setting where admissible fields belong to the kernel of a linear first-order differential operator $\A$ of constant rank. 

The novelty of our contribution consists
in tackling homogenization of high-contrast materials in the setting of $\A$-quasiconvexity.
In order to describe our main results, we need to introduce some basic notation. 

For $d\in \N$, $d\geq 2$,
let $\Omega \subset \Rd$ be a bounded, connected, open set with Lipschitz boundary.
We regard $\Omega$ as a reference configuration
for a composite with a high-contrast microstructure
obtained by periodically inserting inclusions of a specific kind in a surrounding matrix.
For instance, in the context of elasticity,
{\sc M. Cherdantsev \& K. D. Cherednichenko} \cite{CC} considered 
a stiff matrix with embedded soft fillings.
In general, applications require that
the properties of the inclusions and of the matrix are tuned
so that the resulting effective behavior of the composite is the desired one.

To depict the fine texture of the material in mathematical terms,
we consider the periodicity cell $Q\coloneqq  (0,1)^d \subset \Rd$
and an open connected subset with Lipschitz boundary $D_0 \Subset Q$ (see Figure \ref{fig:Q}),
where the symbol $\Subset$ indicates that $D_0$ is compactly contained in $Q$.
The high-contrast behavior of the composite sitting in $\Omega$ is identified
by means of the sets $D_0$ and $D_1 \coloneqq  Q \setminus\bar D_0$
in the following way:
we introduce a small parameter $\eps>0$
representing the periodicity scale
and we define the sets
\begin{gather}\label{eq:om0}
\Omega_{0,\eps} \coloneqq  \bigcup_{z \in Z_\eps} \eps (D_0 + z ),
\quad\text{with }
Z_\eps \coloneqq  \Set{ z \in \Z^d : \eps (\bar D_0 + z ) \subset \Omega}, 
\end{gather}
and
\begin{equation}\label{eq:om1}
\Omega_{1,\eps} \coloneqq  \Omega \setminus \overline\Omega_{0,\eps},
\end{equation}
which correspond respectively
to the inclusions and to the surrounding matrix (see Figure \ref{fig:Omega}).

\begin{figure}
	\caption{The subdivision of the unit cube $Q$
		into the `soft' region  $D_0$ (white) 
		and the `stiff' region $D_1$ (grey).
	}
	\label{fig:Q}		
	\centering
	\begin{tikzpicture}[scale=0.75]
	\filldraw[fill=lightgray] (-2,-2) rectangle (2,2);
	
	\filldraw[fill=white,rotate=120]
	(-0.5,-1) .. controls (-1.5,-1) and (-1.5,0) .. 
	(-1,1) .. controls (0,2) and (0,1) ..
	(1,1) arc(90:-90:0.5)
	(1,0) .. controls (0,0) and (0,0) ..
	(-0.5,-1);
	\path (-0.3,-0.4) node {$D_0$};
	\path (0.8,1) node {$D_1$};
	\end{tikzpicture}
\end{figure}
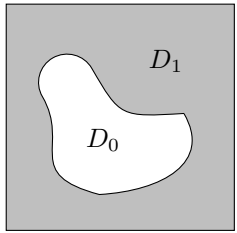
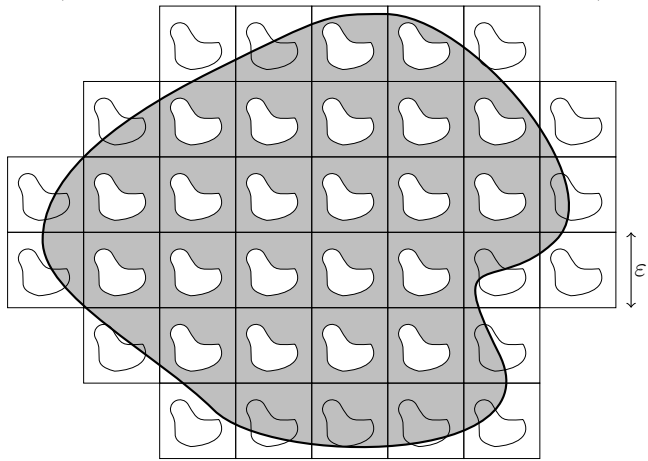
\begin{figure}
	\caption{The reference set $\Omega$ and its microscopic structure.
		The collection of the `soft' inclusions $\Omega_{0,\varepsilon}$ is depicted in white,
		whereas the matrix $\Omega_{1,\varepsilon}$ is in grey.}
	\label{fig:Omega}
	
	\begin{tikzpicture}
	\filldraw[thick,fill=lightgray,scale=1.5]
	[xshift=4,yshift=-2] (-1.5,1)  .. controls (-1.5,1.5) and (-1,2) ..
	(0,2.5) .. controls (1,3) and (1,3) ..
	(1.5,3) .. controls (2.25,3) and (3.5,1.5) .. 
	(3,1) .. controls (2.5,0.5) and (2,1) .. 
	(2.5,0)  .. controls (3,-1) and (0.5,-1) ..
	(0,-0.5) .. controls (-0.5,0) and (-1.5,0.5) ..
	(-1.5,1);
	
	
	\foreach \x in {-4}
	\foreach \y in {4,8}
	{
		\draw[very thin,scale=0.25] (\x-2,\y-2) rectangle (\x+2,\y+2);
		
		\filldraw[very thin,fill=white,scale=0.25,rotate around={120:(\x,\y)},]
		(\x-0.5,\y-1) .. controls (\x-1.5,\y-1) and (\x-1.5,\y) .. 
		(\x-1,\y+1) .. controls (\x,\y+2) and (\x-0,\y+1) ..
		(\x+1,\y+1) arc(90:-90:0.5)
		(\x+1,\y) .. controls (\x,\y) and (\x,\y) ..
		(\x-0.5,\y-1);
	}
	
	\foreach \x in {0,8,12}
	\foreach \y in {0,4,...,16}
	{
		\draw[very thin,scale=0.25] (\x-2,\y-2) rectangle (\x+2,\y+2);
		
		\filldraw[very thin,fill=white,scale=0.25,rotate around={120:(\x,\y)},]
		(\x-0.5,\y-1) .. controls (\x-1.5,\y-1) and (\x-1.5,\y) .. 
		(\x-1,\y+1) .. controls (\x,\y+2) and (\x-0,\y+1) ..
		(\x+1,\y+1) arc(90:-90:0.5)
		(\x+1,\y) .. controls (\x,\y) and (\x,\y) ..
		(\x-0.5,\y-1);
	}
	
	\foreach \x in {4}
	\foreach \y in {0,4,8,12}
	{
		\draw[very thin,scale=0.25] (\x-2,\y-2) rectangle (\x+2,\y+2);
		
		\filldraw[very thin,fill=white,scale=0.25,rotate around={120:(\x,\y)},]
		(\x-0.5,\y-1) .. controls (\x-1.5,\y-1) and (\x-1.5,\y) .. 
		(\x-1,\y+1) .. controls (\x,\y+2) and (\x-0,\y+1) ..
		(\x+1,\y+1) arc(90:-90:0.5)
		(\x+1,\y) .. controls (\x,\y) and (\x,\y) ..
		(\x-0.5,\y-1);
	}
	
%
	
	\foreach \x in {16}
	\foreach \y in {8,12}
	{
		\draw[very thin,scale=0.25] (\x-2,\y-2) rectangle (\x+2,\y+2);
		
		\filldraw[very thin,fill=white,scale=0.25,rotate around={120:(\x,\y)},]
		(\x-0.5,\y-1) .. controls (\x-1.5,\y-1) and (\x-1.5,\y) .. 
		(\x-1,\y+1) .. controls (\x,\y+2) and (\x-0,\y+1) ..
		(\x+1,\y+1) arc(90:-90:0.5)
		(\x+1,\y) .. controls (\x,\y) and (\x,\y) ..
		(\x-0.5,\y-1);
	}
	\draw[<->,scale=0.25] (22.8,2) -- (22.8,6);
	\path[scale=0.25] (23.3,4) node {$\varepsilon$};
	
	\foreach \x in {-8}
	\foreach \y in {4,8}
	{
		\draw[very thin,scale=0.25] (\x-2,\y-2) rectangle (\x+2,\y+2);
		
		\draw[very thin,scale=0.25,rotate around={120:(\x,\y)},]
		(\x-0.5,\y-1) .. controls (\x-1.5,\y-1) and (\x-1.5,\y) .. 
		(\x-1,\y+1) .. controls (\x,\y+2) and (\x-0,\y+1) ..
		(\x+1,\y+1) arc(90:-90:0.5)
		(\x+1,\y) .. controls (\x,\y) and (\x,\y) ..
		(\x-0.5,\y-1);
	}
	
	\foreach \x in {-4}
	\foreach \y in {0,12} 
	{
		\draw[very thin,scale=0.25] (\x-2,\y-2) rectangle (\x+2,\y+2);
		
		\draw[very thin,scale=0.25,rotate around={120:(\x,\y)},]
		(\x-0.5,\y-1) .. controls (\x-1.5,\y-1) and (\x-1.5,\y) .. 
		(\x-1,\y+1) .. controls (\x,\y+2) and (\x-0,\y+1) ..
		(\x+1,\y+1) arc(90:-90:0.5)
		(\x+1,\y) .. controls (\x,\y) and (\x,\y) ..
		(\x-0.5,\y-1);
	}
	
	\foreach \x in {0,4,8,12,16}
	\foreach \y in {-4} 
	{
		\draw[very thin,scale=0.25] (\x-2,\y-2) rectangle (\x+2,\y+2);
		
		\draw[very thin,scale=0.25,rotate around={120:(\x,\y)},]
		(\x-0.5,\y-1) .. controls (\x-1.5,\y-1) and (\x-1.5,\y) .. 
		(\x-1,\y+1) .. controls (\x,\y+2) and (\x-0,\y+1) ..
		(\x+1,\y+1) arc(90:-90:0.5)
		(\x+1,\y) .. controls (\x,\y) and (\x,\y) ..
		(\x-0.5,\y-1);
	}
	
	\foreach \x in {16}
	\foreach \y in {0,4,16} 
	{
		\draw[very thin,scale=0.25] (\x-2,\y-2) rectangle (\x+2,\y+2);
		
		\draw[very thin,scale=0.25,rotate around={120:(\x,\y)},]
		(\x-0.5,\y-1) .. controls (\x-1.5,\y-1) and (\x-1.5,\y) .. 
		(\x-1,\y+1) .. controls (\x,\y+2) and (\x-0,\y+1) ..
		(\x+1,\y+1) arc(90:-90:0.5)
		(\x+1,\y) .. controls (\x,\y) and (\x,\y) ..
		(\x-0.5,\y-1);
	}
	
	\foreach \x in {4}
	\foreach \y in {16} 
	{
		\draw[very thin,scale=0.25] (\x-2,\y-2) rectangle (\x+2,\y+2);
		
		\draw[very thin,scale=0.25,rotate around={120:(\x,\y)},]
		(\x-0.5,\y-1) .. controls (\x-1.5,\y-1) and (\x-1.5,\y) .. 
		(\x-1,\y+1) .. controls (\x,\y+2) and (\x-0,\y+1) ..
		(\x+1,\y+1) arc(90:-90:0.5)
		(\x+1,\y) .. controls (\x,\y) and (\x,\y) ..
		(\x-0.5,\y-1);
	}
	
	\foreach \x in {20}
	\foreach \y in {4,8,12} 
	{
		\draw[very thin,scale=0.25] (\x-2,\y-2) rectangle (\x+2,\y+2);
		
		\draw[very thin,scale=0.25,rotate around={120:(\x,\y)},]
		(\x-0.5,\y-1) .. controls (\x-1.5,\y-1) and (\x-1.5,\y) .. 
		(\x-1,\y+1) .. controls (\x,\y+2) and (\x-0,\y+1) ..
		(\x+1,\y+1) arc(90:-90:0.5)
		(\x+1,\y) .. controls (\x,\y) and (\x,\y) ..
		(\x-0.5,\y-1);
	}
	\end{tikzpicture}
\end{figure}

Hereafter, we will systematically adopt the adjectives `soft' and `stiff'
to refer respectively to $\Omega_{0,\eps}$ and to $\Omega_{1,\eps}$,
as well as to their related quantities.
This use is informal and just meant to convey ideas,
though being justified by the possible interpretation of our model in the framework of elasticity.
The reader should however bear in mind that
our treatment is suited for a wider range of applications.

The main goal of this contribution is to characterize the asymptotic behavior as $\ep\to 0$ in the sense of $\Gamma$-convergence \cite{braides, DalMaso93} of the energy functionals 
\begin{equation}\label{eq:ep-funct}
\int_{\Omega_{0,\eps}} f_{0,\eps}\left(\eps u\right) \de x
+ \int_{\Omega_{1,\eps}} f_1\left(\frac{x}{\ep},u\right) \de x, 
\end{equation}
evaluated on fields $u\in L^p(\Omega;\R^N)$ satisfying
$\A u = 0$,
where $\A$ is a suitable linear first-order partial differential operator.
The precise class of constraints is presented in Subsection \ref{subs:diff} below
and further investigated in Section \ref{sec:adm}. In the expression above,
the family $\set{f_{0,\ep}}$ represents the energy densities stored in the `soft' inclusions,
while $f_1$ encodes the behavior of the `stiff' matrix.  The exact assumptions on $\set{f_{0,\ep}}_{\ep>0}$ and $f_1$ are collected in Subsection \ref{subs:notation}. 
We point out that
considering a sequence of $\ep$-dependent energy densities for the `soft' component is not merely a mathematical mannerism:
on the contrary, it is a modeling necessity
that arises when the admissible maps $u$ are assumed to have the form $u=\nabla v$ for suitable fields $v$, i.e., when $\A=$ curl.
We refer to \cite[Remark 2]{CC} for a discussion on this point.

\RRR
When $f_1$ does not oscillate and the parameter $\ep$ is fixed,
determining an optimal shape $D=\set{x\in \Omega:\chi(x)=1}$ for the functionals
$$(u,\chi)\mapsto \int_{\Omega} \big[ \chi f_0(u)+(1-\chi)f_1(u) \big] \de{x}$$
under a volume constraint amounts to solve a shape-optimization problem.
In this case, the aim is, for instance, finding
the optimal distribution in terms of compliance
for two materials with different stiffness in a given container $\Omega$.
By enforcing in the energy a dependence on $\ep$
like the one we consider,
we attempt to encode the fact that
the material properties in the set $D$ and in its complement may be very different;
specifically, the one sitting in $D$ becomes increasingly soft as $\ep$ vanishes.
In addition, finer and finer microstructures may be taken into account
by assuming as well a dependence on $\ep$ of the shapes,
so that $D$ is replaced by $\Omega_{0,\ep}$.
Problems of this kind are central, e.g.,
in the modeling of lattice materials for additive manufacturing.
We refer to the monograph \cite{allaire-book1} and the paper \cite{allaire-book2}
for extensive discussions about the connections
between homogenization and topology optimization.
\EEE

Even though $\set{f_{0,\ep}}$ and $f_1$ in \eqref{eq:ep-funct} are assumed to satisfy $p$-growth conditions from above and below, the fact that the `soft' energies are evaluated on $\ep u$ rather than $u$ results in a loss of coercivity. As a result, classical weak and strong $L^p$-topologies are not capable to capture the asymptotic behavior of the problem and one needs to resort to an {\RRR {\it ad hoc} notion of  `convergence in the high-contrast sense'}, which instead keeps track of the finest features of the microstructure, cf.~Definition \ref{stm:HC-conv}. 
Denoting by $U_1$ the set of maps in $L^p(\Omega;\R^N)$ such that $\A u=0$ in $\Omega$, our main result consists in showing that the limiting behavior of the energies in \eqref{eq:ep-funct} is encoded by the functional
$\F\colon U_1 \to \R$
defined as

\begin{equation*}
\F(u) \coloneqq
\alpha_0 + \int_{\Omega} f_\hom \big( u (x) \big) \de x.
\end{equation*}
In the formula above,
$\alpha_0$ is a constant determined by $f_0$ (see \eqref{eq:alpha0}),
$f_\hom$ is a suitable energy density related to $f_1$ (see \eqref{eq:def-fhom}),
and the condition $\A u = 0$ in $\Omega$ has to be interpreted in a distributional sense (see Subsection \ref{subs:diff}).

Our first main result is presented in Theorem \ref{stm:main}, and,
loosely speaking, states the following: 
\begin{theorem*}\nonumber
	The functional $\F$ is the $\Gamma$-limit as $\eps\to 0$ of the energies in \eqref{eq:ep-funct}
	restricted to $\A$-free fields with respect to the high-contrast convergence. 
\end{theorem*}
Key techniques to establish the theorem are $p$-equiintegrability arguments, as well as the notions of two-scale convergence \cite{allaire} and of periodic unfolding \cite{cioranescu1, cioranescu2, visintin1, visintin2}.
The proof relies essentially on the possibility of decoupling the behavior of the material in the `soft' and `stiff' contributions.
The $\Gamma$-convergence of the `stiff' portion is a corollary of homogenization results in the setting of $\A$-quasiconvex \cite{fonseca.kromer}. 
The study of the `soft' part is instead more challenging,
for the presence of the high-contrast microstructure and its interplay with the differential constraint result in the emergence of a second, hidden scale.
This makes
both the identification of a lower bound and the proof of its optimality
need a delicate combination
of two-scale results and weak-$L^p$ compactness. 

The decoupling of the system into regions having different material features hinges on careful compactness and splitting arguments. The latter are in turn based on two essential properties of the differential constraint $\A$: the fact that $\A$-free maps in the periodicity cell have null average and the existence of a suitable extension operator on perforated domains (see Assumptions \ref{stm:null-av} and \ref{stm:exist-ext}). In our general setting, these conditions are quite hard to ensure, in contrast to the well-understood case in which the set $\Omega$ is contractible, $\A$ is the curl operator, and the $\A$-free fields are gradients. In particular, to the authors' knowledge, even for standard operators, it is still an open question whether extension operators preserving the $\A$-free constraint from a perforated domain to the `filled' set exist. Similar results are available for gradients both in the Sobolev setting \cite{ACDP} and in the weaker one of antiplane fracture \cite{cagnetti.scardia}. We also mention the recent result in \cite{chambolle.perugini} for the case of symmetric gradients in linearized fracture (in the space $GSBD$ \cite{dalmaso2}).

The analysis of the differential constraints satisfying the aforementioned Assumptions \ref{stm:null-av} and \ref{stm:exist-ext} is the focus of Section \ref{sec:adm}. A special case in which the first requirement holds is that of operators $\A$ admitting `potentials', namely such that $\A u=0$ if and only if $u=\B w$ for a suitable differential operator $\B$ and a field $w$. The second main contribution of this paper consists in showing that for a broad class of operators $\A$ existence of potentials and existence of extension operators are closely related. Roughly speaking, we prove the following.

\begin{theorem*}
Let $\Omega$ be a `nice' set and let $\A$ be a linear homogeneous differential operator with constant coefficients and constant rank. If an extension operator from $\A$-free maps on $\Omega$ to $\A$-free maps on the whole space exists, then $\A$ admits a potential on $\Omega$.
Conversely, if $\A$ admits a potential on $\Omega$ for which a suitable Korn-type inequality holds true, then there exists an extension operator from $\Omega$ to the whole space which preserves the property of being $\A$-free.
\end{theorem*}
We refer to Theorems \ref{stm:raita-bis} and \ref{thm:exist-ext} for the precise formulation of the result,
as well as to Section \ref{sec:adm} for a broader discussion on the topic.
We point out that assuming the existence of potentials imposes some topological constraints on the set $\Omega$.
An example is easily provided by the case $\A=\mathrm{curl}$ for $d=3$:
one needs to require at least $\Omega$ to be simply connected, cf. Remark \ref{stm:simplconn} and Examples \ref{ex:curl}--\ref{ex:div}. Essential tools for the proof of Theorems \ref{stm:raita-bis} and \ref{thm:exist-ext} are the theory of Fourier multipliers and the notion of generalized Moore-Penrose inverse.

Before proceeding with the mathematical details of our analysis,
we conclude the section with some bibliographical notes.

The very first mathematical analysis of high-contrast materials was developed
in the seminal work \cite{auriault} by {\sc G.~L.~Auriault},
who employed formal asymptotic expansions.
In the early 2000, {\sc V.~V.~Zhikov} \cite{zhikov} introduced a novel approach
by extending classical two-scale techniques \cite{allaire}
to the setting of PDEs with rapidly oscillating coefficients.
The study of high-contrast problems has been ever since at the center of an extensive scientific effort,
with applications ranging from elastodynamics \cite{smyshlyaev}
to Maxwell’s equations \cite{bouchitte.schweizer, cherednichenko.cooper}.
A full characterization of the settings of linear elasticity and of conducting materials was provided in \cite{camar-eddine.seppecher},
while the nonlinear elastic setting was considered in \cite{CC,CCN} (see also \cite{braides.garroni}).
The effects of microstructure on free-discontinuity functionals have been studied
when high-contrast behavior appears in bulk terms only \cite{barchiesi, barchiesi.lazzaroni.zeppieri},
when it is just in surface terms \cite{scardia, scardia2},
and when both contributions are affected \cite{braides.solci}.
Mumford-Shah energies
with high-contrast surface contributions
have been recently characterized in \cite{xavier.scardia.zeppieri}.
The analysis of variational models for layered high-contrast materials was undertaken
in \cite{kreisbeck.christowiak, kreisbeck.christowiak2, davoli.ferreira.kreisbeck}.  

As for the notion of $\A$-quasiconvexity,
its introduction goes back to \cite{dacorogna},
while an extensive study was carried out in \cite{fonseca.muller}
for operators $\A$ with constant coefficients and constant rank
(see \eqref{eq:op-A} and \eqref{eq:symbol} below).
In \cite{braides.fonseca.leoni} and \cite{fonseca.kromer},
under $p$-growth assumptions on the energy density,
relaxation and homogenization results were obtained
for integral functionals evaluated on $\A$-free fields (i.e. fields in the kernel of $\A$);
see also \cite{ferreira.fonseca.raghav} for a related analysis on quasicrystals.
Problems featuring simultaneously homogenization and dimension reduction were addressed
in \cite{kreisbeck.kromer, kreisbeck.rindler}, whereas
oscillations and concentrations generated by $\A$-free fields were characterized in \cite{fonseca.kruzik}.
The case of non-positive energy densities has been studied in \cite{kramer.kromer.kruzik.patho}.
A complete theory under linear growth assumptions on the energy density was established
in \cite{arroyo, arroyo.dephilippis.rindler, baia.chermisi.matias.santos, matias.morandotti.santos}.
The characterization of $\A$-quasiconvexity was extended to operators with variable coefficients in \cite{santos}.
We refer to \cite{davoli.fonseca, davoli.fonseca2} for homogenization results in this purview,
and to \cite{davoli.fonseca3} for a corresponding relaxation formula.
Applications to the theory of compressible Euler systems,
as well as to adaptive image processing and to data-driven finite elasticity were the subject of \cite{chiodaroli.feireisl.kreml.wiedemann}, \cite{davoli.fonseca.liu}, and \cite{conti.mueller.ortiz}, respectively.
To complete our review on $\A$-quasiconvexity,
we mention the works \cite{raita, raita.skorobogatova} on $BV_{\A}$-spaces for elliptic and cancelling operators, \cite{kristensen.raita} for a characterization of associated Young measures, as well as \cite{GR, GRV, raita2, raita3} for the corresponding Sobolev analysis, and \cite{guerra.raita.schrecker} for a compensated-compactness result.


\subsection{Structure of the paper}
The plan of the paper is as follows.
In Section~\ref{sec:setting},
we describe the setting of the problem and the related assumptions, and we formulate the precise statements of our main results, i.e., Theorems~\ref{stm:main}, \ref{stm:raita-bis}, and \ref{thm:exist-ext}.
To lay the ground for the proofs,
we collect auxiliary assertions and properties concerning measurable selection criteria, $\A$-free sequences, two-scale convergence, and Fourier analysis in Section~\ref{sec:preliminaries}.
In Section \ref{sec:comp+split} we study compactness properties of sequences of fields with equibounded energies, we detail the splitting argument, and present the proof of Theorem~\ref{stm:main}. Sections \ref{sec:soft} and \ref{sec:stiff} contain the $\Gamma$-convergence analysis for the `soft' and `stiff' energy-contributions. Section~\ref{sec:adm} is focused on the description of the class of the differential operators
that are admissible for our analysis, and
deals with the proofs of Theorems \ref{stm:raita-bis} and \ref{thm:exist-ext}.


\section{Setting and main results}\label{sec:setting}
This section is devoted to the set-up of our analysis and
to the presentation of the main achievements of this contribution.
We will first fix the notation and the hypotheses used throughout the paper.
The major results will be summarized in Subsection \ref{subs:results}.


\subsection{Energy functionals} \label{subs:notation}
Let $\Omega\subset \R^d$, $d\in \mathbb{N}$, $d\geq 2$, be a bounded, connected, open set with Lipschitz boundary.
We denote by $\chi_{0,\eps}$ and $\chi_{1,\eps}$
the characteristic functions respectively of $\Omega_{0,\eps}$ and $\Omega_{1,\eps}$  in $\Omega$ (see \eqref{eq:om0} and \eqref{eq:om1}), i.e. for $i=0,1$,
\begin{equation}\label{eq:chi-i-eps}
   \chi_{i,\eps}(x) \coloneqq
    \begin{cases}
        1 & \text{if } x\in \Omega_{i,\eps}, \\
        0 & \text{otherwise.}
    \end{cases} 
\end{equation}
Similarly, we let
\[
    \chi_i(y) \coloneqq
     \begin{cases}
        1 & \text{if } y\in D_i, \\
        0 & \text{otherwise},
    \end{cases}
\]
denote the characteristic function $\chi_i$ of $D_i$ in $Q$.

We assume that $\set{f_{0,\ep}}_{\ep>0}$ and $f_1$ in \eqref{eq:ep-funct} fulfill the following set of hypotheses:
	\begin{description}
		\item[H1] each $f_{0,\eps}\colon \RN \to \R$ is continuous
		and $f_1\colon \Rd \times \RN \to \R$ is a Carath\'eodory function;
		\item[H2] $f_1(\,\cdot\,,\xi)$ is $Q$-periodic for every $\xi\in \RN$;
		\item[H3] there exist $a,\lambda,\Lambda > 0$ such that
		for a.e. $x\in \R^d$, for all $\xi\in \RN$, and all  $\eps>0$, 
				\begin{gather*}
				\lambda(-a + \va{\xi}^p)\leq f_{0,\eps}(\xi) \leq \Lambda (1+\va{\xi}^p), \\
				\lambda(-a + \va{\xi}^p)\leq f_1(x,\xi) \leq \Lambda (1+\va{\xi}^p),
				\end{gather*}
			for a fixed $p\in (1,+\infty)$;
		\item[H4] there exists $\mu> 0$ such that for a.e. $x\in\Omega$,
				for all $\xi,\eta\in \RN$,  and $\eps>0$, 
				\begin{gather*}
				\va{f_{0,\eps}(\xi) - f_{0,\eps}(\eta)}
					\leq \mu \left(1+\va{\xi}^{p-1}+\va{\eta}^{p-1}\right)\va{\xi - \eta}, \\
				\va{f_1(x,\xi) - f_1(x,\eta)}
				\leq \mu \left(1+\va{\xi}^{p-1}+\va{\eta}^{p-1}\right)\va{\xi-\eta},
				\end{gather*}
		where $p$ is the same as in {\bf H3};
		\item[H5] there exists $f_0\colon \RN \to \R$ such that
		    for all $\xi\in \RN$
		        \[
		          \lim_{\eps\to 0} f_{0,\eps}(\xi) = f_0(\xi).
		        \]
	\end{description}


\begin{remark}[Degeneracy of the `soft' component]
\label{stm:degeneracy}
\RRR
    We draw again attention to the fact that,
    in spite of the standard coercivity and growth of $f_{0,\eps}$,
    the problem we address features some degeneracy,
    which is expressed by the factor $\eps$ multiplying the argument of the `soft' integrand (see \eqref{eq:ep-funct}).
    From a modeling perspective,
    this $\eps$ accounts for the high-contrast between the two components,
    in that it makes the `soft' component less and less sensitive to external stimuli.
    This can be easily seen in the simple instance $f_{0,\eps}(\xi)=f_1(\xi)=\va{\xi}^p$:
    formula \eqref{eq:ep-funct} becomes
    \[
        \eps^p\int_{\Omega_{0,\eps}} f_{0,\eps}(u) \de x
        + \int_{\Omega_{1,\eps}} f_1\left(\frac{x}{\ep},u\right) \de x, 
    \]
    and the ratio $1/\eps^p$ between the `stiffness coefficients' of the two components blows up as $\eps$ vanishes.
\end{remark}

Owing to the growth conditions prescribed by hypothesis {\bf H3},
the natural environment for our problem is the space $L^p$.
We will often work with $Q$-periodic functions, i.e.,
those $u\colon \Rd \to \RN$ such that
$u(x+z) = u(x)$ for all $x\in\Rd$ and all $z\in\Z^d$.
We will use the subscript `$\per$' to denote subspaces of $Q$-periodic maps;
for instance,
\begin{gather*}
L^p_\per(\Rd;\RN) \coloneqq  \Set{ u\in L^p_\loc(\Rd;\RN)
	: u \text{ is $Q$-periodic}}.
\end{gather*}
We endow the previous space with the norm of $L^p(Q;\RN)$.
With a slight abuse of notation,
we will tacitly identify $L^p(Q;\RN)$ with $L^p_\per(\Rd;\RN)$
by mapping the unit cube into the unit torus in $\R^d$ and extending the corresponding fields by periodicity. Analogously, we will denote the subspace of $L^p_{\rm loc}(\Omega \times \Rd;\RN)$
containing all the functions
that are $Q$-periodic w.r.t. their second argument by 
 $L^p(\Omega; L^p_{\rm per}(\Rd;\RN))$ and implicitly identify this space with $L^p(\Omega\times Q;\RN)$.


\subsection{The differential constraint}
\label{subs:diff}
Let $p\in (1,+\infty)$ be as in {\bf H3} and {\bf H4} above.
We suppose that
the constraint that we couple with the energy in \eqref{eq:ep-funct}
falls within the framework of $\A$-quasiconvexity
as described by {\sc I. Fonseca \& S. M\"uller} \cite{fonseca.muller}.
Namely, for $N,M\in \N\setminus\set{0}$,
we let $\A$ be a partial differential operator on $\Rd$ from $\RN$ to $\RM$
whose action on a function $u\colon\R^d\to\RN$ is given by
\begin{equation} \label{eq:op-A}
	\A u \coloneqq \sum_{i=1}^{d} A^{(i)} \pder{u}{x_i},
\end{equation}
where $A^{(i)}\colon \RN \to \RM$ are linear maps for all $i=1,\dots,d$.
In other words,
$\A$ is a linear first-order differential operator with constant coefficients.

Our fundamental requirement on $\A$ is the \emph{constant rank} assumption
introduced by {\sc F. Murat} \cite{murat}, that is,
we suppose that there exists $r\in\N$ such that
the \emph{symbol} $\bA$ of $\A$ satisfies the following:
for all $\omega=(\omega_1,\dots,\omega_d)\in \Rd\setminus\set{0}$, the operator
	\begin{equation}
	\label{eq:symbol}
		\bA[\omega]\coloneqq  \sum_{i=1}^{d} \omega_i A^{(i)} 
	\end{equation}
has rank equal to $r$.

We observe that for every open set $O\subset \R^d$,
$\A$ can be regarded as an operator
from $L^{p}(O;\RN)$ to $W^{-1,p}(O;\RM)$, where
$W^{-1,p}(O;\RM)$ denotes the dual
of the Sobolev space $W^{1,p/(p-1)}_0(O;\RM)$.
To see this,
we introduce the formal adjoint of $\A$, denoted by $\A^\ast$. For $v\colon O\to \RM$, we set
\[
    \A^\ast v \coloneqq  -  \sum_{i=1}^{d} A^{(i),\ast} \pder{v}{x_i},
\]
where $A^{(i),\ast}$ is the adjoint (i.e. the transpose) of $A^{(i)}$.
In this way,
\[
	\int_{O} \A \phi \cdot \psi\,\de{x} = \int_{O} \phi \cdot \A^\ast \psi\,\de{x}
	\quad\text{for all }
	\phi\in \Cc^1(O;\RN)
	\text{ and }
	\psi\in\Cc^1(O;\RM),
\]
and for $u\in L^{p}(O;\RN)$ we define the pairing
\[
    \langle \A u, v \rangle \coloneqq \int_O u\cdot \A^\ast v\, \de x
    \quad\text{for all } v \in W^{1,p'}_0(O;\RM),
\]
where $p'$ is the conjugate exponent to $p$.
In what follows, if $u\in L^p(O;\R^N)$,
equalities of the form $\A u = 0$ in $O$ are always tacitly understood
in the sense of $W^{-1,p}$, or, in other words,
in the sense that
\begin{equation}\label{eq:Afree-openset}
    \int_{O} u \cdot \A^\ast v\,\de{x} = 0
	\quad\text{for all }
	v\in W^{1,p'}_0(O;\RM).
\end{equation}
When such a relation holds,
we say that $u$ is {\it $\A$-free} on $O$.
Similarly, if $u\in L^p_\per(\Rd;\RN)$,
we say that it is $\A$-free on the unit torus $\Td$
when the equality in \eqref{eq:Afree-openset} is satisfied for $O=Q$
and for all $v\in W^{1,p'}_\per(\Rd;\RM)$;
in particular, if $u\in L^p_\per(\Rd;\RN)$ is $\A$-free on $\Td$,
then it is $\A$-free on $Q$ as well.

In addition to the constant rank hypothesis,
we need to consider further restrictions on the class of operators $\A$
for which our analysis will be performed.
We formulate two {\it ad hoc} assumptions:

\begin{assumption}[Null-average] \label{stm:null-av}
    For all $u\in L^p_{\rm per}(\R^d;\RN)$ such that
    $u=0$ on $D_1$ and 
    $\A u=0$ on $\Td$,
    it holds $\int_Q u(y)\,\de{y} = 0$.
\end{assumption}

\begin{assumption}[$\A$-free extension] \label{stm:exist-ext}
	There exists an $\eps$-independent constant $c> 0$ and a sequence of operators $\set{\mathsf{E}_\A^\ep}$, with  $\mathsf{E}_\A^\ep \colon L^p(\Omega;\RN) \to L^p(\Omega;\RN)$
	such that the following holds:
	for all $\A$-free $u\in L^p(\Omega;\RN)$
	\begin{enumerate}
    \item ${\mathsf E}_{\A}^\ep u = u$ a.e. in $\Omega_{1,\ep}$,
    \item $\norm{ \mathsf E_{\A}^\ep u}_{L^p(\Omega;\RN)} \leq c \norm{u}_{L^p(\Omega_{1,\eps};\RN)}$,
	\item $\A (\mathsf E_{\A}^\ep u) = 0$ on $\Omega$, and
	\item if $\set{u_\ep}\subset L^p(\Omega;\R^N)$ is $p$-equiintegrable, then $\set{\mathsf{E}_{\A}^\ep u_\ep}$ is also $p$-equiintegrable in $L^p(\Omega;\R^N)$.
	\end{enumerate}
\end{assumption}
We recall that 
$\set{u_\ep}\subset L^p(\Omega;\R^N)$ is $p$-equiintegrable
if for every $\eta > 0$ there exists $m>0$ such that
\[
    \int_E \va{u_\eps}^p\,\de{x} < \eta
    \quad\text{for all } \eps>0
\]
whenever $E\subset \Omega$ satisfies $\Ld(E)<m$.
Thanks to the Dunford-Pettis Theorem,
this is equivalent to the fact that
$\set{\va{u_\ep}^p}$ admits a subsequence
that is weakly convergent in $L^1(\Omega)$.
We point out that in Assumption \ref{stm:exist-ext}
it is essential to start with maps $u$ which are $\A$-free in the whole set $\Omega$,
for this is related to the existence of `potentials' for the operator $\A$.
We refer to the considerations after Theorem \ref{thm:exist-ext} in Subsection \ref{subs:results} 
for a discussion on this point,
while we collect some comments on Assumption \ref{stm:exist-ext} in the next remark .

\RRR
\begin{remark}[On Assumption \ref{stm:exist-ext}]\label{stm:rmkA2}
	In problems involving perforated media
	the use of extension techniques is fairly common, and
	Assumption \ref{stm:exist-ext} is akin to others in the literature.
	For comparison, we mention \cite[Definition 1.4]{fonseca.kruzik}.
	The main differences consist in: 1) the fact that the extension operator here is assumed to exist on a periodically perforated domain, with constants independent of the periodicity parameter $\ep$; 2) the fact that the existence of an extended $p$-equiintegrable sequence here is only required when starting from a sequence that also had $p$-equiintegrability properties. We refer to \cite{fonseca.kruzik} for a general discussion on the connections between the existence of extension operators from an open set to a bigger surrounding domain, and the theory of DiPerna-Majda measures. 
	
	A notion of $\A$-free extension domain was also considered in \cite[Definition 4.4]{kramer.kromer.kruzik.patho}.
	There, however, no $p$-equiintegrability conditions are prescribed. As pointed out in \cite[Section 4.2]{kramer.kromer.kruzik.patho}, such extension results are, in general, non-trivial. Additionally, they depend both on the topology of the set and on the operator $\A$ under consideration. In particular, no such $\A$-free extension is possible when $\A$ is the operator associated to the Cauchy-Riemann system, not even in the case of very regular sets $\Omega$.
	
	As far as our analysis is concerned,
	items (1)--(3) in Assumption \ref{stm:exist-ext} are among the main ingredients
	in the proof of Proposition \ref{stm:HC-limits}, and
	without them we would not be able
	to establish high-contrast compactness for sequences with equibounded energy
	(see Definition \ref{stm:HC-conv}).
	On the contrary, the preservation of $p$-equiintegrability is not required at this stage,
	whereas it is used in the proof of our main $\Gamma$-convergence result
	(see the proofs of Propositions \ref{prop:stiff-part} and \ref{stm:split2}).
	Nonetheless, point (4) in Assumption \ref{stm:exist-ext} could be dispensed with
	in some simplified versions of our problem,
	e.g. when $f_{0,\eps}=f_0$ for every $\eps>0$ and $\Omega$ is a rectangle.
\end{remark}
\EEE

For the sake of brevity,
it is convenient to give a name
to the class of operators addressed in our study:

\begin{definition}
\label{def:adm}
Let $\A$ be a linear first-order differential operator as in \eqref{eq:op-A}.
We say that $\A$ is \emph{admissible}
if and only if it is of constant rank
and it satisfies Assumptions \ref{stm:null-av} and \ref{stm:exist-ext}.
\end{definition}

It is well-known that
the family of constant rank operators is quite large \cite{fonseca.muller};
as for the subclass of the admissible ones,
we prove that it is non-empty in Section \ref{sec:adm}
where, in particular, we analyze the cases
of the curl, of the $\mathrm {curl\,curl}$ operator, of the divergence, 
as well as the setting associated with higher-order gradients.


\subsection{Main results}\label{subs:results}
The first main result of this paper, Theorem \ref{stm:main} below, is the characterization
of the asymptotic behavior as $\ep\to 0$ of the functionals in \eqref{eq:ep-funct}
when the family is restricted to $\A$-free fields.
In order to deduce also some information on the related energy minimizers,
we resort to a variational kind of convergence,
namely $\Gamma$-convergence (see e.g. \cite{braides,DalMaso93}).
Here, we consider a quite abstract version of its definition:
indeed, we consider functionals defined on a certain set
for which it is only known that some of its sequences are converging,
the limit point being declared too;
no topological structure is provided, instead.
We summarize this situation by saying that
such a set is endowed with a notion of convergence.

\begin{definition}
	Let $X$ be a set endowed with a notion of convergence.
	We say that the family $\set{F_\eps}_{\eps>0}$
	of functions on $X$ with values in $[-\infty,+\infty]$
	\emph{$\Gamma$-converges} as $\epsilon\to 0$
	to the function \mbox{$F\colon X\to [-\infty,+\infty]$}
	if for any $x\in X$ and
	for any sequence $\set{\epsilon_k}_{k\in\N}$ such that $\eps_k\to 0$
	the following holds:
	\begin{enumerate}
		\item for any sequence $\set{x_{\eps_k}}_{k\in\N}\subset X$
			such that $x_{\eps_k}\to x$, we have
			\begin{equation*}
				F(x)\leq\liminf_{k\to +\infty} F_{\eps_k}(x_{\eps_k});
			\end{equation*}
		\item there exists a sequence $\set{x_{\eps_k}}_{k\in\N}\subset X$
			such that $x_{\eps_k}\to x$ and
			\begin{equation*}
				\limsup_{k \to +\infty} F_{\eps_k}(x_{\eps_k}) \leq F(x).
			\end{equation*}
	\end{enumerate}
\end{definition}

We have already remarked that
the natural domain of the functionals \eqref{eq:energy} is contained in $L^p(\Omega;\RN)$.
However, in our setting, the two kinds of convergence in this space
that are most frequently considered,
the strong and the weak one,
are not well-suited to perform a $\Gamma$-convergence analysis.
Indeed, they do not match the high-contrast nature of the problem and
would therefore not yield useful compactness results
for sequences with equibounded energy.
This leads us to introduce a peculiar notion of convergence. 

\begin{definition}\label{stm:HC-conv}
	We say that	a family $\set{u_\eps} \subset L^p(\Omega;\RN)$
	\emph{converges in the high-contrast sense} to $u\in L^p(\Omega;\RN)$
	{\RRR (relatively to the sequence of sets $\set{\Omega_{1,\eps}}$)}
	if $\eps u_\eps \weak 0$ weakly in $L^p(\Omega;\R^N)$
	and if there exists a second family $\set{\tilde u_\eps}\subset L^p(\Omega;\RN)$
	such that  $\A\tilde{u}_\ep=0$ in $W^{-1,p}(\Omega;\R^M)$,  $u_\eps = \tilde u_\eps$ in $\Omega_{1,\ep}$ and
	$\tilde u_\eps \weak u$ weakly in $L^p(\Omega;\RN)$.
\end{definition}

{\RRR
Some comments are in order.
The reason why
the definition above is made up of two requirements 
is that
there is no standard convergence which is able to capture alone
the behaviors of both components.
On one hand, the convergence of extensions is related to the coercivity of the `stiff' part,
and it is a somehow customary requirement in homogenization 
(see Subsection \ref{sec:comparison});
on the other, the convergence $\eps u \weak 0$ is needed
to keep track of the `soft' component.
Here, in the light of Assumption \ref{stm:null-av},
the $0$ vector turns out to be the average of some two-scale limit.

If we select $\A=\mathrm{curl}$ and we assume that $\Omega$ is simply connected,
the problem may be recast in $W^{1,p}(\Omega)$.
In this setting, as far as the asymptotics of the `stiff' part is concerned,
Rellich-Kondrachov's Theorem allows to switch from weak convergence in $W^{1,p}(\Omega)$ to strong convergence in $L^p(\Omega)$.
Therefore, the use of high-contrast convergence on the matrix leads to the standard homogenization results for perforated media.
At the same time, it allows to quantify the change in the effective energy of the perforated material
that occurs when `soft' inclusions are added
(cf. $\alpha_0$ in \eqref{eq:limitfunct}).
Indeed, it can be readily shown
that $\eps \nabla w_\eps \weak 0$
if $\F_\eps(\nabla w_\eps)\leq c$
by well-known two-scale properties of gradients.

As a last remark on Definition \ref{stm:HC-conv}, note that
it is obvious that
the weak convergence in $L^p(\Omega;\RN)$ of a sequence $\set{u_\eps}$ of $\A$-free maps entails high-contrast convergence
(it suffices to set $\tilde u_\eps \coloneqq u_\eps$).
In particular, bounded sequences in $L^p$ are pre-compact in the high-contrast sense.
As Proposition \ref{stm:HC-limits} proves, however,
a much weaker control is enough (see \eqref{eq:apriori-compact}),
at least for admissible differential operators.
}

It will be convenient to have a special notation
for spaces of functions satisfying the differential constraint encoded by $\A$.
For any open $\Omega' \subset \Omega$, we define
\begin{multline}\label{eq:U0}
	U_0(\Omega') \coloneqq
		\big\{
			u\in L^p(\Omega;L^p_\per(\Rd;\RN)) : u = 0 \text{ if } y\in D_1 \\
			\text{ and }
			\A_y u= 0 \text{ in } W^{-1,p}(\Td;\RM) \text{ for a.e. } x\in  \Omega'
		\big\}
\end{multline}
and
\begin{gather} \label{eq:U1}
    U_1 \coloneqq \set{ u \in L^p(\Omega;\RN) : \A u = 0\text{ in } W^{-1,p}(\Omega;\RM)}.
\end{gather}
Note that limit points of high-contrast convergent sequences are automatically in $U_1$.  Hereafter, we use the subscripts $x$ and $y$ to denote that
the operator $\A$ acts on the variables $x$ and $y$, respectively.
We will prove that
	the limiting behavior of the $\eps$-dependent energies constrained to $\A$-free fields is described
	by the functional $\F\colon U_1 \to \R $
	defined as
	\begin{equation}\label{eq:limitfunct}
	\F(u) \coloneqq
		\alpha_0 + \int_{\Omega} f_\hom \big( u (x) \big) \de x
\end{equation} 
	where
\begin{equation}\label{eq:alpha0}
	\alpha_0 \coloneqq \sup_{\Omega'\Subset\Omega} \inf_{u \in U_0(\Omega')} \int_{\Omega'}\int_{D_0} f_{0} \big( u(x,y) \big) \de y\,\de x
\end{equation}
and
\begin{multline}\label{eq:def-fhom}
	f_\hom (\xi) \coloneqq \liminf_{k\to +\infty}
		\inf\Bigg\{
			\sum_{z \in \Zd} \int_{Q \cap k^{-1}(D_1 + z)} f_1 \big( ky, \xi + v(y) \big) \de y
			: v\in L^p_\per(\Rd;\RN) \\
			\text{ with } \int_Q v(z) \de z = 0 \text{ and } \A v=0 \text{ in } W^{-1,p}(\Td;\RM)
	\Bigg\}.
\end{multline}

We are now ready to state our main convergence result.

\begin{theorem}\label{stm:main}
	
    Let $\F\colon L^p(\Omega; L^p_{\rm per} (\R^d; \RN)) \to \R \cup \set{+\infty}$ be the functional in \eqref{eq:limitfunct} and, for $\ep>0$,
    let $\F_\ep\colon L^p(\Omega;\RN) \to \R\cup \set{+\infty}$ be given by
    \begin{equation}\label{eq:energy}
    \F_\ep(u) \coloneqq \begin{cases}
    \displaystyle{
    	\int_{\Omega_{0,\eps}} f_{0,\eps}\left(\eps u\right) \de x
    	+ \int_{\Omega_{1,\eps}} f_1\left(\frac{x}{\ep},u\right) \de x
    }   & \text{if } u \in U_1,\\
    +\infty     & \text{otherwise in } L^p(\Omega;\RN).
    \end{cases}
    \end{equation}
    Let us assume that $\set{f_{0,\ep}}$ and $f_1$ fulfill {\bf H1}--{\bf H5},
    and that $\A$ is an admissible partial differential operator
    in the sense of Definition \ref{def:adm}.
    Then, the following properties hold:
    \begin{enumerate}
    	\item If a family $\set{u_\ep} \subset L^p(\Omega;\R^N)$ satisfies
    	$\F_\ep(u_\ep)\leq c$ for all $\ep>0$ and  for some \RRR $c\geq 0$, \EEE
    	then it admits a subsequence
    	that is convergent in the high-contrast sense.
    	\item For any $u\in U_1$ and for any sequence $\set{u_\eps}\subset L^p(\Omega;\RN)$
    	that converges to $u$ in the high-contrast sense, it holds that
    	\[
    	\F(u) \leq \liminf_{\eps \to 0} \F_\eps(u_\eps).
    	\]
    	\item For any $u\in U_1$  there exists a sequence $\set{u_\eps}\subset U_1$
    	that converges to $u$ in the high-contrast sense and satisfies
    	\[
    	\limsup_{\eps \to 0} \F_\eps(u_\eps) \leq \F(u).
    	\]
    \end{enumerate}
\end{theorem}

By standard $\Gamma$-convergence techniques,
we find the following corollary of Theorem \ref{stm:main}:
\begin{corollary}
\label{stm:minimizers}
Under the same assumptions and notation of Theorem \ref{stm:main}, 
if $\set{u_\ep} \subset L^p(\Omega;\RN)$ is a sequence of almost-minimizers for $\set{\F_\ep}$, i.e., if
\[
    \lim_{\ep\to 0}\left( \F_\ep(u_\ep) - \inf_{u\in  U_1} \F_\ep(u) \right)=0,
\]
then there exists a subsequence of $\set{u_\ep}$
that is convergent in the high-contrast sense to a minimum point of $\F$.
Moreover, 
$$\inf_{u\in L^p(\Omega;\RN)}\F_\ep(u)\to \min_{u\in  U_1}\F(u).$$
\end{corollary}

Our proof of Theorem \ref{stm:main} relies on four main intermediate results:
1) a compactness analysis combined with the characterization of high-contrast limits of $\A$-free fields,
cf. Proposition \ref{stm:HC-limits};
2) a splitting argument that allows
to reduce the study of the $\Gamma$-limit of $\set{\F_\eps}$ to the study of two independent problems,
concerning the asymptotic behavior of the `soft' and of the `stiff' energy contributions, respectively,
cf. Lemma \ref{stm:split} and Proposition \ref{stm:split2};
3) the identification of an optimal lower bound for the `soft' part of the energy, cf. Proposition \ref{stm:Glim-soft};
4) the identification of the limiting description for the `stiff' part, cf. Proposition \ref{prop:stiff-part}.
The proof of Theorem \ref{stm:main} is exposed in Section \ref{sec:comp+split}.

Let us now present shortly the results that
we obtain for the `soft' and `stiff' parts of the energy.
The splitting procedure in Section \ref{sec:comp+split} yields the functionals
\begin{gather}
\label{eq:F0eps}
    \F_{0,\eps}(u) \coloneqq 
    \begin{cases}
    \displaystyle{ \int_{\Omega_{0,\eps}} f_{0,\eps}\left(\eps u \right) \de x} &\text{if } u \in U_1, \\
    +\infty&\text{otherwise in } L^p(\Omega;\R^N),
    \end{cases}
\\
\label{eq:F1eps}
    \F_{1,\ep}(u)\coloneqq 
    \begin{cases}
    \displaystyle{\int_{\Omega_{1,\eps}} f_1\left(\frac{x}{\eps}, u \right) \de x}      & \text{if } u \in U_1,\\
    +\infty     &\text{otherwise in }L^p(\Omega;\R^N),
    \end{cases}
\end{gather}
where the densities $\set{f_{0,\ep}}$ and $f_1$ satisfy {\bf H1}--{\bf H5}.
We will show, respectively in Sections \ref{sec:soft} and \ref{sec:stiff},
that $\set{\F_{0,\eps}}$ and $\set{\F_{1,\ep}}$ $\Gamma$-converge to
\begin{gather}
\label{eq:F0}
    \F_{0}(u) \coloneqq 
    \begin{cases}
    	\alpha_0    &\text{if } u=0, \\
        +\infty     &\text{otherwise in }L^p(\Omega;\RN),
    \end{cases}
\\
\label{eq:F1}
    \F_{1}(u) \coloneqq 
    \begin{cases}
        \displaystyle{\int_{\Omega} f_\hom \big( u(x) \big) \de x} & \text{if } u \in U_1, \\
        +\infty     &\text{otherwise in }L^p(\Omega;\R^N),
    \end{cases}
\end{gather}
Precisely, we establish the following.

\begin{proposition}[$\Gamma$-limit of the `soft' component]
\label{stm:Glim-soft}
    Let $\F_{0,\eps},\F_0\colon L^p(\Omega;\RN) \to \R\cup\set{+\infty}$ be as above.
    If the hypotheses {\bf H1} and  {\bf H3}--{\bf H5} are satisfied, and
    if $\A$ is an admissible differential operator as in Definition \ref{def:adm},
    then for all $u\in L^p(\Omega;\RN)$ the following hold:
	\begin{enumerate} 
	    \item for every  $p$-equiintegrable sequence $\set{u_\ep}\subset L^p(\Omega;\RN)$
	    such that $u_\ep=0$ on $\Omega_{1,\ep}$
	    and that $\ep u_\ep\wk u$ weakly in $L^p(\Omega;\R^N)$, we have
	        \[
		        \F_0(u) \leq \liminf_{\eps \to 0} \F_{0,\eps}(u_\eps);
	        \]
	    \item there exists a sequence $\set{u_\epsilon} \subset U_1$
	    with the properties that $u_\ep=0$ in $\Omega_{1,\ep}$ for every $\ep>0$,
	    $\ep u_\ep\wk u$ weakly in $L^p(\Omega;\R^N)$, and 
	        \[
		        \limsup_{\eps \to 0} \F_{0,\eps}(u_\eps) \leq \F_0(u).
	        \]
	\end{enumerate}
\end{proposition}

We point out that the $p$-equiintegrability condition in the first part of the previous statement does not affect the generality of the result,
as Proposition \ref{stm:split2} will prove. 
For more details on this point, the reader is referred to the splitting argument in Section \ref{sec:comp+split} below.

For what concerns the functionals $\F_{1,\eps}$,
the analysis carried out by {\sc I. Fonseca \& S. Kr\"omer} in \cite{fonseca.kromer} yields the following:

\begin{proposition}[$\Gamma$-limit of the `stiff' component]
\label{prop:stiff-part}
    Let $\F_{1,\eps}, \F_1\colon L^p(\Omega;\RN) \to \R\cup\set{+\infty}$ be as above.
    If the hypotheses {\bf H1}--{\bf H4} are satisfied and
    if $\A$ is an admissible differential operator as in Definition \ref{def:adm},
    then the $\Gamma$-limit of $\set{\F_{1,\eps}}$ with respect to the weak $L^p(\Omega;\RN)$-convergence is $\F_1$.
\end{proposition}

The last section is focused on the study of admissible differential operators.
Differently from the other parts of the paper, the analysis of Section \ref{sec:adm} encompasses all linear, $k$-th order, homogeneous differential operators with constant coefficients and constant rank.
It will be useful to consider those operators $\A$ for which there exists a second operator $\B$ with the property that $u=\B w$ for some Sobolev function $w$ whenever $\A u = 0$;
in this case we say that $\B$ is a potential for $\A$.
Our second main result proves that, for $\A$ to admit a potential on a certain open set $\Omega$, it is sufficient that $\A$-free maps on $\Omega$ can be extended to $\A$-free maps on the whole space in such a way that a control on the $L^p$ norm is ensured.
We dub a set $\Omega$ with this property an $\A$-extension domain, see Definition \ref{stm:ext-dom}. As Remark \ref{stm:simplconn} shows, an $\A$-extension domain has to comply with some topological requirements; in particular, it cannot have holes, in general.

\begin{theorem}[Existence of potentials for $\A$-free maps]
\label{stm:raita-bis}
    Let $\A$ be a linear, $k$-th order, homogeneous differential operator
    with constant coefficients and constant rank.
    If $\Omega\subset \Rd$ is a bounded, connected, open set with Lipschitz boundary
    which is also an $\A$-extension domain,
    then there exists $\ell\in \N$, and a differential operator $\B$ of order $\ell$ satisfying the following:
    for all $\A$-free maps $u\in L^p(\Omega;\RN)$
    there is a function $w\in W^{\ell,p}(\Omega;\RM)$
    such that $u=\B w$ almost everywhere in $\Omega$.
\end{theorem}
An immediate consequence of this result is that
Assumption \ref{stm:null-av} holds
whenever the unit cube $Q$ is an $\A$-extension domain
(cf. Corollary \ref{cor:u0}).

Our third and last main result shows that for operators admitting a potential
satisfying a suitable Korn-type inequality it is possible to define an extension operator preserving $\A$-free maps. 

\begin{theorem}
\label{thm:exist-ext}
Let $\A$ be a linear, $k$-th order, homogeneous differential operator
    with constant coefficients and constant rank. 
    Let also $\B$ be a linear, $\ell$-th order,  homogeneous differential operator 
    with constant coefficients such that  
    \begin{equation*}
            \ker \bA[\omega] = \im \bB[\omega]
            \quad\text{for all }\omega\in \Rdmz,
        \end{equation*}
        where $\bA$ and $\bB$ are
        the symbols of $\A$ and $\B$, respectively.
    We further assume that
    \begin{itemize}
        \item for all $\A$-free $u\in L^p(\Omega;\RN)$ there exists $w\in W^{\ell,p}(\Omega;\RM)$ satisfying $u=\B w$;
        \item {\RRR for any bounded, connected and open set $D\subset \Rd$ with Lipschitz boundary
            there exist a projection operator on the subspace of $\B$-free maps
            $\sfPi_\B\colon W^{\ell,p}(D;\RM) \to W^{\ell,p}(D;\RM)$,
            as well as a constant $c>0$ such that}
        \begin{equation}\label{eq:genKorn-O}
            \norm{\nabla^\ell(w - \sfPi_\B w)}_{L^p(D;\R^{N\times d^\ell})}
            \leq c \norm{\B w}_{L^p(D;\RN)}
            \quad\text{for all } w\in W^{\ell,p}(D;\RM).    
            \end{equation}
    \end{itemize}
    Then, there exist a constant $c\coloneqq c(d,p,D_1)>0$ independent of $\ep$ and and $\Omega$, as well as a sequence of maps $\set{{\mathsf E}_{\A}^\ep}$, with
     ${\mathsf E}_{\A}^\ep\colon L^p(\Omega;\R^N)\to L^p(\Omega;\R^N)$,
    {\RRR satisfying Assumption \ref{stm:exist-ext}.}
\end{theorem}

It is fundamental to observe that the maps under consideration in the theorem above are already {\it a priori} in the kernel of the operator $\A$ on the whole set $\Omega$. For this reason, it is meaningful to assume the existence of potentials $\B$. In the case in which, instead, our fields were in the kernel of the operator $\A$ only in the perforated domain, its lack of contractibility would prevent the existence of a potential to hold even in the simple setting $d=3$ and $\A={\rm curl}$. We refer once again to Remark \ref{stm:simplconn} for a counterexample.

\subsection{Comparison with other works}\label{sec:comparison}
Before getting to the heart of the matter, we take the chance to compare our main convergence result, Theorem \ref{stm:main},
with similar ones in the literature.

When $\A$ is the curl operator, our analysis is akin to the one performed by {\sc M. Cherdantsev \& K.\ D. Cherednichenko} in \cite{CC}.
However, our conclusions differ from theirs, also in the simple situation of a contractible $\Omega$. Indeed, even though under this topological assumption curl-free maps coincide with gradients, the notion of convergence that we use is much weaker than the strong two-scale convergence considered in \cite{CC}. 
Another key difference with \cite{CC} is that
in our setting convergence of minimizers follows directly from Theorem \ref{stm:main} (see Corollary \ref{stm:minimizers}), whereas in \cite{CC} it had to be shown \emph{a posteriori} (see \cite[Section 10]{CC}).
Actually, the introduction of high-contrast convergence is motivated exactly by the fact that sequences with uniformly bounded energy $\F_\ep$ are sequentially precompact in that sense.

This notion of convergence is inspired by the one considered by {\sc X. Pellet, L. Scardia, \& C. I. Zeppieri} in  \cite{xavier.scardia.zeppieri} to deal with high-contrast Mumford-Shah energies. Unlike that contribution, aside from the role of the differential constraint $\A$, we require additionally that $\set{\ep u_\ep}$ converges to zero, so as to comply with the delicate two-scale characterization in Proposition \ref{stm:HC-limits}.

	We also stress that, in the absence of growth conditions from below on the `soft' part of the energy, the weak $L^p$-topology alone would not ensure convergence of minimizers. We refer to \cite[Section 4]{CC} and \cite[Section 4]{braides.garroni} for a discussion on this topic. 
	
	For a general differential constraint $\A$, our approach does not correspond to the ones in \cite{braides.garroni,CC}.
	If $\A$ does not admit a of suitable `potential' $\B$ satisfying a Korn-type inequality (cf. Theorem \ref{stm:raita-bis} below), then $\A$-free maps and fields $u$ of the form $u=\B w$ are not interchangeable.   
    Another {\it caveat} involves the choice of the `soft' effective energy. Looking at the results in \cite{CC}, one could expect that in our case the $\A$-quasiconvex envelope of $f_0$ is to be retrieved in the $\Gamma$-limit.
    We recall that for a continuous function $g\colon \RN \to \R$ with $p$-growth the $\A$-quasiconvex envelope of $g$ is defined as
    \begin{multline*}
        Q_{\A}g(\xi) \coloneqq
    \inf \left\{
        \int_Q g(\xi+v(z)) \de z
        : v\in L^p_\per(\Rd;\RN) \right. \\
		\left.	\text{ with } \int_Q v(z) \de z = 0 \text{ and } \A v=0 \text{ in } W^{-1,p}(\Td;\RM)
        \right\}.
    \end{multline*}
The example below shows that $Q_\A f_0$ in not the correct limiting energy density.

\begin{example}
\RRR
	Let $\set{{f}_{0,\ep}}$ satisfy {\bf H1}--{\bf H5},
	let $\Omega'\Subset \Omega$, and let $D \subset D_0$ be open.
	In general, if a family $\set{w_\eps}\subset L^p(\Omega'; L^p_\per(\Rd;\RN))$ satisfies
	\begin{align*}
		&
		\set{w_\eps}\text{ is }p\text{-equiintegrable},\\
		&
		w_\eps \weak 0\text{ weakly in }L^p(\Omega'\times Q;\RN), \\
		&
		\A_y w_\eps=0 \text{ in } W^{-1,p}(\Td;\RM) \text{ for a.e. } x\in\Omega'
	\end{align*}
	we cannot conclude that for all $\xi_0\in \RN$ it holds
	\begin{equation}\label{eq:lsc-ex}
		\int_{\Omega'}\int_{D}  Q_{\A}f_0(\xi_0)\,\de y\,\de x
		\leq 
		\liminf_{\eps \to 0}
		\int_{\Omega'}\int_{D} {f}_{0,\ep} \big( \xi_0 + w_\eps(x,y) \big)\, \de y\,\de x,
	\end{equation}
	as the following counterexample proves.
	
	For $d=2$ and $N=2\times 2$, up to translations and dilations,
	we may assume $\Omega' \coloneqq (0,1)\times(0,1)=Q$.
	We focus on the case $\A=\mathrm{curl}$, in which
	the $\A$-quasiconvex envelope $Q_\A f$ of the function $f$ coincides
	with the well-know quasiconvex envelope $Qf$
	\cite[Section 6.3]{dacorogna2}.
	Given $\lambda>0$, as energy densities
	we set $f_{0,\ep}(\xi)=f_0(\xi)=\psi_\lambda(\xi)-\det(\xi)$ for every $\ep>0$,
	where $\psi_\lambda \colon \R^{2\times 2} \to \R$ is a convex function such that
	\[
	\psi_\lambda(\xi) =
		\begin{cases}
		0												& \text{if } \left|\xi\right| < 2, \\
		(2+\lambda)\left|\xi\right|^2 	& \text{if } \left|\xi\right| \geq 3.
		\end{cases}	
	\]	
	In this way, requirements {\bf H1} and {\bf H3}--{\bf H5} are met,
	and $Qf_{0,\ep}=Qf_0=f_0$, because $f_{0,\ep}=f_0$ is quasiconvex.
	
	In order to disprove \eqref{eq:lsc-ex}, we pick $\xi_0 = 0$ and,
	for $x=(x_1,x_2) \in (0,1)\times(0,1)$ and $y\in D$,
	we define
	\[
		a(x)\coloneqq
		\begin{cases}
		- \mathbb{I} & \text{if } x_1 \in \left(0,\dfrac{1}{2}\right]\\[11pt]
		\mathbb{I} & \text{if } x_1 \in \left(\dfrac{1}{2},1\right)
		\end{cases}
		\quad \text{and} \quad
		w_\eps(x,y) \coloneqq a\left(\frac{x}{\eps}\right)+\eps \mathbb{I},
	\]
	where $\mathbb{I}$ is the identity matrix in $\R^{2\times 2}$ and
	$a$ is extended by $Q$-periodicity.
	Then, $\set{w_\eps}$ is $p$-equiintegrable and
	$w_\eps\wto 0$ weakly in $L^p(\Omega';\R^{2\times 2})$.
	Besides, $\A_y w_\eps=\mathrm{curl}_y w_\eps=0$,
	because the sequence does not depend on $y$.
	
	For such particular choices, we see that the left-hand side of \eqref{eq:lsc-ex} equals $0$,
	while for the right-hand side we find
	\begin{multline*}
		\lim_{\eps \to 0}
		\int_{\Omega'}\int_{D}
			\Big[\psi_\lambda \big(w_\eps(x,y) \big) - \det\big( w_\eps(x,y)\big) \Big] 
	 	\de y\,\de x
	 	\\ = - \lim_{\eps \to 0}
	 	\int_{\Omega'}\int_{D} \det\big( w_\eps(x,y)\big) \de y\,\de x
	 	= - \mathscr{L}^2(D) < 0.
	\end{multline*}
\EEE
\end{example}

\section{Preliminaries}\label{sec:preliminaries}
We gather in this section
some preliminary definitions and results to be used later in the paper.

\subsection{Measurable selection arguments}
We recall here a measurable selection criterion that we will invoke in the proof of Corollary \ref{cor:u0}.
For a thorough discussion on the topic
we refer to the book by {\sc C. Castaing \& M. Valadier} \cite{castaing.valadier}.

\begin{proposition}[Theorem III.6 in \cite{castaing.valadier}]
    \label{stm:castaing-valadier}
    Let $S$ be a multifunction
    defined on the measurable space $\mathscr{O}$
    and taking values in the collection of nonempty complete subsets
    of the separable metric space $X$.
    If for all open $O\subset X$
    the set $\Set{ \omega\in \mathscr{O} : S(\omega) \cap O \neq \emptyset}$ is measurable,
    then $S$ admits a measurable selection,
    that is,
    there exists a measurable function
    $s\colon \mathscr{O} \to X$ such that $s(x) \in S(x)$ for all $x\in \mathscr{O}$.
\end{proposition}


\subsection{$\A$-free fields and two-scale limits}
\label{subs:prel-A-free}
In this subsection we include some useful tools
for the variational analysis of problems
under differential constraints.

We recall two results about
(asymptotically) $\A$-free fields.
The operator $\A$ is always assumed to be of the form \eqref{eq:op-A} and of constant rank.

\begin{lemma}[$\A$-free decomposition, Lemma 2.15 in \cite{fonseca.muller}]
\label{stm:Adecomp}
	Let $\Omega \subset \Rd$ be an open, bounded domain, and let $p\in (1,+\infty)$. Let $\set{u_k}\subset L^p(\Omega;\RN)$ be a bounded sequence such that $\A u_k\to 0$ strongly in $W^{-1,p}(\Omega;\R^M)$.
	Then, there exist a subsequence $\set{u_{j_k}}$ and
	a sequence $\set{v_k} \subset L^p(\Omega;\RN)$
	such that the following holds:
	\begin{enumerate}
		\item $\set{v_k}$ is bounded, $\A$-free, and $p$-equiintegrable;
		\item $u_{j_k} - v_k\to 0$  strongly in $L^q(\Omega;\RN)$ for every $q\in[1,p)$.
	\end{enumerate}
\end{lemma}

\begin{lemma}[$\A$-free periodic extension, Lemma 2.8 in \cite{fonseca.kromer}]
\label{stm:Afree-ext}
    Let $D\subset Q$ be open, and let 
 $p\in(1,+\infty)$.
    Let $\set{u_k}\subset L^p(D;\RN)$ be a $p$-equiintegrable sequence such that
    $u_k \weak 0$ weakly in $L^p(D;\RN)$ and $\A u_k \to 0$ in $W^{-1,p}(D;\RM)$. Then,
    there exists an $\A$-free sequence $\set{v_k}\subset L^p_\per(\Rd;\RN)$ 
    which is $p$-equiintegrable in $Q$ and satisfies
    \begin{align*}
		&v_k - u_k \to 0	\quad \text{strongly in } L^p(D;\RN),
		\quad
		v_k \to 0 \quad \text{strongly in } L^p(Q\setminus D;\RN), \\
		&\int_Q v_k(x,y) \de y = 0,
		\quad
		\norm{v_k}_{L^p(Q;\RN)} \leq c(\A,D) \norm{u_k}_{L^p(D;\RN)}\quad\text{for all }k\in \N.
		\end{align*}
\end{lemma}

To perform our analysis,
we need to track
how the differential constraint behaves along sequences
that converge in the sense of Definition \ref{stm:HC-conv} above.
In this respect, the notion of two-scale convergence \cite{nguetseng,allaire} will be crucial:

\begin{definition}
\label{stm:twoscaleconv}
We say that
$\set{u_\eps}\subset L^p(\Omega;\RN)$ \emph{weakly two-scale converges in $L^p$}
to a function $u\in L^p(\Omega;L^p_\per(\Rd;\RN))$
if for all $v\in L^{p'}(\Omega;C_\per(\Rd;\RN))$ it holds
\[
    \lim_{\eps\to 0} \int_{\Omega} u_\eps(x)\cdot v \left(x,\frac{x}{\eps}\right) \de x = \int_\Omega \int_Q u(x,y) \cdot v(x,y)\,\de y \de x.
\]
In this case we write $u_\eps \weakts u$.

We say that $\set{u_\eps}\subset L^p(\Omega;\RN)$ \emph{strongly two-scale converges in $L^p$}
to $u\in L^p(\Omega;L^p_\per(\Rd;\RN))$ if $u_\ep\weakts u$ in $L^p$ and 
$\|u_\ep\|_{L^p(\Omega;\R^N)}\to \|u\|_{L^p(\Omega;L^p(Q;\R^N))}$.
In this case we write $u_\eps \stackrel{2}{\to} u$ strongly in $L^p$.
\end{definition}

For the sake of completeness,
we record in the following lines the properties of two-scale convergence
to be exploited in this work;
we refer to \cite{allaire,visintin1} for further reading.

\begin{lemma}
\label{lemma:2-scale}
    Let $\set{u_\eps}\subset L^p(\Omega;\RN)$ be a sequence.
    \begin{enumerate}
        \item If $\set{u_\eps}$ is weakly two-scale convergent,
            then it is bounded in $L^p(\Omega;\RN)$; conversely,
            if $\set{u_\eps}$ is bounded in $L^p(\Omega;\R^N)$,
            then, it possesses a subsequence
            which is weakly two-scale convergent.
        \item If $u_\eps \weakts u$ weakly two-scale in $L^p$,
            then $u_\eps \weak \int_Q u(\,\cdot\,,y)\de y$ weakly in $L^p(\Omega;\RN)$.
        \item If $u_\eps \weakts u$ weakly two-scale in $L^p$
            and $\set{v_\eps}\subset L^p(\Omega;\RN)$ is another sequence
            with the property that
            $v_\eps \stackrel{2}{\to} v$ strongly two-scale in $L^p$,
            then $u_\eps v_\eps \weakts uv$ weakly two scale in $L^p$.
        \item If $u\in L^p(\Omega;C_{\rm per}(\R^d;\R^N))$ or $u\in C(\bar{\Omega};L^p_{\rm per}(\R^d;\R^N))$, then the sequence $\set{u_\ep}$ defined as
        $$u_\ep(x):=u\left(x,\frac{x}{\ep}\right)\quad\text{ for a.e. }x\in \Omega$$
        is $p$-equiintegrable and $u_\eps \stackrel{2}{\to} u$ strongly in $L^p$. 
    \end{enumerate}
\end{lemma}

It is well-known that
two-scale convergence in $L^p$ can be related to $L^p$ convergence
by means of the unfolding operator.

\begin{lemma}
\label{stm:unfolding}
    For $\eps>0$, let us define the \emph{unfolding operator}
    $\mathsf{S}_\eps\colon L^p(\Omega)\to L^p(\Rd;L^p_\per(\Rd;\RN))$ as
	\[
		\mathsf{S}_\eps v(x,y) \coloneqq  \tilde v \left( \eps \left\lfloor \frac{x}{\eps}\right\rfloor + \eps y \right),
	\]
	where $\tilde v$ denotes the extension of $v$ by $0$ outside $\Omega$.
	Then, if $\set{u_\eps}\subset L^p(\Omega;\RN)$ is a bounded sequence,
	the following holds:
	\begin{enumerate}
	    \item $u_\eps \weakts u$ weakly two-scale in $L^p$ if and only if
	        $\mathsf S_\eps u_\eps \weak u$ weakly in $L^p(\R^d\times Q;\RN)$;
	   \item $u_\eps \stackrel{2}{\to} u$ strongly two-scale in $L^p$ if and only if
	        $\mathsf S_\eps u_\eps \to u$ strongly in $L^p(\R^d\times Q;\RN)$.
	\end{enumerate}
	Moreover, if $\set{u_\eps}$ is $p$-equiintegrable,
	the sequence of unfoldings $\set{\mathsf{S}_\eps u_\eps}$ is $p$-equiintegrable too in $L^p(\Rd\times Q;\RN)$.
\end{lemma}
We refer to \cite{cioranescu1, cioranescu2, visintin1, visintin2}
for further properties of the unfolding operator.
	
A characterization of
weak two-scale limits of $\A$-free sequences was
established by {\sc I. Fonseca \& S. Kr\"omer}.

\begin{proposition}[Theorem 1.2 in \cite{fonseca.kromer}] \label{stm:Afree-limits}
    A function $u\in L^p(\Omega;L^p_\per(\Rd;\RN))$ is the weak two-scale limit
    of an $\A$-free sequence $\set{u_\eps}\subset L^p(\Omega;\RN)$ if and only if
    \begin{gather*}
        \A_x \left(\int_Q u(x,y)\de y\right)= 0
        \quad\text{in }W^{-1,p}(\Omega;\R^M),\\
        \A_y u(x,y) = 0
        \quad\text{in }W^{-1,p}(\Td;\R^M)\text{ for a.e. }x\in \Omega.
        \end{gather*}
\end{proposition}

As a last technical tool, we show that
the unfolding of an $\mathscr{A}$-free map is in turn $\A$-free w.r.t. the periodicity variable. 
We preliminarily introduce the set
    \begin{equation}
    \label{eq:def-new-set}
        \hat{\Omega}_\ep\coloneqq \bigcup_{z\in \hat Z_\eps} \eps(Q+z),
        \quad\text{with}\quad
        \hat Z_\ep \coloneqq \set{ z \in \Z^d : \eps (Q + z ) \subset \Omega}.
    \end{equation}
Note that $\hat{Z}_\ep\subset Z_\ep$, where $Z_\ep$ is the collection of indices in \eqref{eq:om0}. 

\begin{lemma}
\label{stm:Afree-unfold}
    Let $v\in L^p(\Omega;\R^N)$ be such that $\mathscr{A}v=0$ in $W^{-1,p}(\Omega;\R^M)$.
    Then, for every $\ep>0$ there holds
    \begin{equation}\label{eq:Ay-unfolded}
        \A_y (\mathsf{S}_\ep v)=0
        \quad\text{in $W^{-1,p}(Q;\R^M)$ for a.e. } x\in \hat{\Omega}_\ep.
    \end{equation}
    Moreover, if also $v=0$ in $\Omega_{1,\eps}$, then
    \begin{equation*}
        \A_y (\mathsf{S}_\ep v)=0
        \quad\text{in $W^{-1,p}(\Td;\R^M)$ for a.e. } x\in \hat{\Omega}_\ep.
    \end{equation*}
\end{lemma}
\begin{proof}
    Let $\eta \in C^\infty_c(\hat\Omega_\ep)$ and $\psi \in W^{1,p'}_0(Q;\R^M)$. A change of variables yields
    \begin{align}
    \int_{\hat{\Omega}^\ep} \int_Q (\mathsf{S}_\ep v)(x,y)\cdot \eta(x)\A^*\psi (y)\, \de{y}\de{x}
        & =\int_{\hat\Omega_\eps } \int_Q
            v\left(\ep \left\lfloor\frac{x}{\ep}\right\rfloor+\ep y\right)\cdot \eta(x)\A^*\psi (y)\,
            \de{y}\de{x} \nonumber \\
        & =\frac{1}{\ep^{d-1}} \int_{\hat\Omega_\eps } \eta(x)
            \int_{\ep \left(Q + \left\lfloor \ep^{-1}x \right\rfloor \right)}
            v(z)\cdot \A^*_z \psi_\ep^x(z)\, \de z\de x, \label{eq:psi-x-eps}
    \end{align}
    where $\psi_\ep^x(z)\coloneqq \psi( \ep^{-1} z -\left\lfloor \ep^{-1} x \right\rfloor)$
    for a.e. $ z \in \ep \left(Q + \left\lfloor \ep^{-1}x \right\rfloor \right)$.
    Since $\psi \in W^{1,p'}_0(Q;\R^M)$,
    $\psi_\ep^x$ belongs to $W^{1,p'}_0 \left(\ep \left(Q + \left\lfloor \ep^{-1}x \right\rfloor \right);\R^M\right)$ 
    and it can be regarded as an element of $W^{1,p'}_0(\Omega;\R^M)$
    by extending it to $0$ outside $\ep \left(Q + \left\lfloor \ep^{-1}x \right\rfloor \right)$.
    In this step \eqref{eq:def-new-set} needs to be used.
    The conclusion follows now from \eqref{eq:psi-x-eps},
    because $v$ is $\A$-free in $\Omega$.
    
    Assume further that $v=0$ in $\Omega_{1,\eps}$.
    By the definition of the unfolding operator,
    $\mathsf{S}_\ep v(x,y)=0$ for all $(x,y) \in \hat \Omega_\eps \times D_1$, so that
    for any $\psi \in W^{1,p'}_\per(\Rd;\RM)$
    \[
        \int_Q \mathsf{S}_\ep v(x,y) \cdot \A^\ast \psi(y)\, \de y
        = \int_{D_0} \mathsf{S}_\ep v(x,y) \cdot \A^\ast \psi(y)\, \de y
        \quad\text{for all } x\in\hat{\Omega}_\eps.
    \]
    Let now $\eta \in C^\infty_c(Q;[0,1])$ be a cut-off function which is constantly $1$ on $D_0$.
    From \eqref{eq:Ay-unfolded}
    we conclude that
    \begin{align*}
        \int_Q \mathsf{S}_\ep v(x,y) \cdot \A^\ast \psi(y)\, \de y
        & = \int_{D_0} \mathsf{S}_\ep v(x,y) \cdot \A^\ast \big( \eta(y) \psi(y) \big)\, \de y \\
        & = \int_Q \mathsf{S}_\ep v(x,y) \cdot \A^\ast \big( \eta(y) \psi(y) \big)\, \de y = 0
    \end{align*}
    for almost every $x\in \hat \Omega_\eps$.
\end{proof}

\subsection{Fourier analysis}\label{sec:fourier}
The study of the class of admissible operators in Section \ref{sec:adm} is grounded on the theory of Fourier multipliers.
For a comprehensive treatment of the matter,
we refer to the monographs \cite{stein,grafakos};
here we limit ourselves to a short recollection of useful properties.

We let $\mathscr{S}$ denote the Schwartz space of rapidly decreasing functions and
for $u\in\mathscr{S}(\Rd;\RN)$, we let
\begin{equation}\label{eq:fourier}
    \mathcal{F} u(\omega)\coloneqq \frac{1}{(2\pi)^{d/2}} \int_{\Rd} e^{-\mathrm{i} \omega \cdot x} u(x)\, \de x
\end{equation}
be its Fourier transform.
We also denote by $\mathcal{F}^{-1}$ the inverse transform and
by $\mathrm{Lin}(\RN;\RN)$ the space of linear maps from $\RN$ to $\RN$.
We recall that a measurable function $m\colon \Rd \to \mathrm{Lin}(\RN;\RN)$ is said to be an \emph {$L^p$-multiplier}
if the linear operator $\mathsf{T}_m$ defined as
\[
    \mathsf{T}_m u \coloneqq \mathcal{F}^{-1}\big( m (\mathcal{F}u) \big)
    \quad\text{for } u\in\mathscr{S}(\Rd;\RN)
\]
can be extended to a bounded operator from $L^p$ to $L^p$.
Given a sufficiently smooth $m$,
{\sc S.G.~ Mikhlin}'s multipliers theorem provides a condition for it to be a multiplier
in terms of the decay of its derivatives.
We firstly introduce some notation:
when $i$ is a $d$-dimensional multi-index,
i.e. $i\coloneqq (i_1,\dots,i_d)\in \N^d$,
we set $\va{i}\coloneqq\sum_j i_j$ and
\[
    \partial_i u(x) \coloneqq \frac{ \partial^{\va{i}} u}{\partial^{i_1} x_1 \cdots \partial^{i_d} x_d}(x).
\]
In the scalar case, the criterion reads:

\begin{proposition}[Theorem 2 in the Appendix of \cite{mikhlin})]
    Let $m\colon \Rdmz\to \R$ be a function of class $C^k$ with $k>d/2$.
    If there exists $c\geq 0$ such that
    \[
        \va{ \partial_i m(x) } \leq \frac{c}{ \va{x}^{\va{i}} }
        \quad\text{whenever } \va{i}\leq k,
    \]
    then $m$ is an $L^p$-multiplier.
\end{proposition}
A helpful consequence of the previous result is:

\begin{corollary}
\label{stm:mikhlin}
    Let $m\colon \Rdmz\to \R$ be a function of class $C^k$ with $k>d/2$.
    If $m$ is $0$-homogeneous,
    then it is an $L^p$-multiplier for all $p\in(1,+\infty)$.
\end{corollary}

Multipliers stand as basics examples of pseudo-differential operators.
Their theory, in turn, can be used to characterize Sobolev spaces as follows.
Given $k\in\N$,
we introduce for $\omega \in \Rd$ the symbol $\mathbb{S}[\omega] \coloneqq (1 + \va{\omega}^2)^{k/2}$ and
the associated pseudo-differential operator $(I-\Delta)^{k/2}$:
\[
    (I-\Delta)^{k/2} u \coloneqq \mathcal{F}^{-1}\big( \mathbb{S} (\mathcal{F}u) \big).
\]
We then say that a distribution $u$ belongs to $W^{k,p}(\Rd)$ if
\[
    (I-\Delta)^{k/2} u \in L^p(\Rd)
\]
and we endow the space with the norm
$\norm{u}_{W^{k,p}(\Rd)} \coloneqq \norm{(I-\Delta)^{k/2} u}_{L^p(\Rd)}$. 
It turns out that
this definition of $W^{k,p}$ is equivalent
to the one given in terms of weak derivatives and that
the norms are comparable too.
The same approach shows that
the dual space $W^{-k,p'}(\Rd)$ coincides with the space of distributions for which it holds
\[
    (I-\Delta)^{-k/2} u \in L^{p'}(\Rd)
\]
and that the norms $\norm{u}_{W^{-k,p'}(\Rd)}$ and $\norm{(I-\Delta)^{-k/2} u}_{L^{p'}(\Rd)}$ coincide.
Here, we recall that $p'\coloneqq p/(p-1)$ and
$(I-\Delta)^{-k/2}$ is the pseudo-differential operator defined by the symbol $(1 + \va{\omega}^2)^{-k/2}$.


\section{Proof of Theorem \ref{stm:main}}
\label{sec:comp+split}
This section is devoted to the proof of Theorem \ref{stm:main},
which stands as our principal result
about the asymptotics of the energy functionals $\F_\eps$ in \eqref{eq:energy}.
For the moment being, we assume that
the limiting behaviors of the `soft' and `stiff' contributions are known, i.e.,
we suppose that Propositions \ref{stm:Glim-soft} and \ref{prop:stiff-part} hold true.
Their proofs are dealt with in Sections \ref{sec:soft} and \ref{sec:stiff} below.
Our main task in the current section is then to show that
sequences with equibounded energies are precompact in a suitable sense and that
the asymptotics of the global energy $\F_\eps$ is determined by that of the functionals accounting for the `soft' and `stiff' parts.
In this respect,
it is useful to regard the total energy $\F_\eps(u)$ of $u\in L^p(\Omega;\RN)$
as the sum of ${\F}_{0,\eps}(u)$ and $\F_{1,\eps}(u)$,
which were defined in \eqref{eq:F0eps} and \eqref{eq:F1eps}.  

The section is organized as follows.
We first address the compactness result in Proposition \ref{stm:HC-limits},
second we show how to reduce our problem to the sub-problems regarding the `soft' and the `stiff' components,
then we prove the $\Gamma$-convergence statement in Theorem \ref{stm:main}. 


\subsection{Compactness}
The next proposition shows that
sequences which are equibounded in energy
have high-contrast converging subsequences.
This kind of convergence was introduced in Definition \ref{stm:HC-conv}.
 We remark that, owing to position \eqref{eq:energy},
only $\A$-free fields give rise to finite energy configurations.
We prove here that
high-contrast limits of $\A$-free sequences inherit differential constraints.

\begin{proposition}[High-contrast limits of  $\mathscr{A}$-free sequences]
\label{stm:HC-limits}
    Let $\A$ be a constant rank differential operator of the form \eqref{eq:op-A}.
    Let also $\set{u_\eps}\subset L^p(\Omega;\RN)$ be such that 
    there exists $c\geq -a \lambda\Ld(\Omega)$ for which $\sup_{\ep>0}\F_\eps(u_\eps)\leq c$. Then,
    \begin{enumerate}
    \item There exist a (non-relabeled) subsequence,
        as well as two maps $u_0 \in L^p(\Omega;L^p_\per(\Rd;\RN))$ and $u_1 \in L^p(\Omega;\RN)$,
        such that  $\set{\eps \chi_{0,\eps} u_\eps}$ converges to $u_0$ weakly two-scale and
        that $\set{\chi_{1,\eps}u_\eps}$ converges to $u_1$ weakly in $L^p(\Omega;\RN)$.
        Moreover, up to subsequences,
        $\set{\eps u_\eps}$ and $\set{\chi_{1,\ep}u_\ep}$ weakly two-scale converge in $L^p$, respectively, 
        to $u_0$ and
        to a map $v_1\in L^p(\Omega;L^p_{\per}(\Rd;\RN))$
        that satisfy
        \begin{equation}\label{eq:cluster1}
            u_0(x,y) = 0 \text{ if }y\in D_1
            \quad\text{and}
            \quad
            \int_Q v_1(\cdot,y)\,\de y=u_1.   
        \end{equation}
    \item  The following differential constraints are fulfilled:
        \begin{align}
        &\A_x \left(\int_{D_0} u_0(x,y)\de y\right)= 0
        \quad\text{in }W^{-1,p}(\Omega;\R^M), \label{eq:cluster2} \\
        &\A_y u_0(x,y) = 0
        \quad \text{in } W^{-1,p}(\Td;\R^M) \text{ for a.e. } x\in\Omega,  \label{eq:cluster2bis}\\
        &\mathscr{A}_y v_1(x,y)=0
        \quad \text{in } W^{-1,p}(D_1;\R^M) \text{ for a.e. } x\in\Omega. \label{eq:cluster3-bis}
    \end{align}
    \item If 
        $\mathscr{A}$ satisfies Assumption \ref{stm:exist-ext},
        setting $\tilde{u}_\ep\coloneqq \mathsf{E}_\A^\ep u_\ep$,
        there exists $\tilde{u}\in L^p(\Omega;L^p_\per(\Rd;\RN))$ such that,
        up to subsequences, $\set{\tilde{u}_\eps}$ weakly two-scale converges to $\tilde{u}$ in $L^p$,
        and there holds
        \begin{align}
            &v_1(x,y) = \chi_{D_1}(y)\tilde{u}(x,y)\quad\text{almost everywhere in }\Omega\times Q, \label{eq:cluster4} \\
            &\A u_1 = - \A_x \left( \int_{D_0} \tilde{u}(x,y)\,\de{y} \right)\text{ in }W^{-1,p}(\Omega;\R^M), \label{eq:cluster5} \\
            &\A_y (\chi_0 \tilde{u} )=0
            \quad \text{in } W^{-1,p}(Q;\R^M) \text{ for a.e. } x\in\Omega. \label{eq:cluster5bis}
        \end{align}
        \item If additionally $\A$ satisfies Assumption \ref{stm:null-av}, then 
        \begin{equation}
            \A u_1=0 \text{ in } W^{-1,p}(\Omega;\RM) \label{eq:cluster6}
        \end{equation}
        and $\tilde u_\eps \weak u_1$ weakly in $L^p(\Omega;\RN)$.
     In particular, up to subsequences, $\set{u_\eps}$ converges to $u_1$ in the high-contrast sense.
    \end{enumerate}
\end{proposition}
\begin{proof}
    If ${\F}_\eps(u_\eps)\leq c$,
    then ${\F}_{0,\eps}(u_\eps)\leq c$ and $\F_{1,\eps}(u_\eps)\leq c$
    (cf. \eqref{eq:F0eps}--\eqref{eq:F1eps}).
    The growth assumptions in {\bf H3} yield
    \begin{equation}
    \label{eq:bds-norms}
        \norm{\eps  \chi_{0,\eps} u_\eps}_{L^p(\Omega;\RN)}\leq c,
        \quad
        \norm{\chi_{1,\eps} u_\eps}_{L^p(\Omega;\RN)}\leq c.
    \end{equation}
    
    \noindent (1) By Lemma \ref{lemma:2-scale}, since bounded sequences in $L^p$ are weakly two-scale precompact,
    there exist $u,u_0,v_1 \in L^p(\Omega;L^p_\per(\Rd;\RN))$
    such that, up to a (non-relabeled) subsequence,
    \begin{equation}\label{eq:apriori-compact}
        \eps u_\eps \weakts u,
        \quad
        \eps \chi_{0,\eps} u_\eps \weakts u_0,
        \quad
        \chi_{1,\eps} u_\eps \weakts v_1
        \quad
        \text{ weakly two-scale in } L^p.   
    \end{equation}
    For what concerns \eqref{eq:cluster1},
    we observe that,
    thanks to the relation between weak two-scale convergence and weak $L^p$-convergence,
    we have
    \begin{equation}
    \label{eq:add1-wklp}
        \chi_{1,\eps} u_\eps \weak u_1 \coloneqq \int_Q v_1(\,\cdot\,,y)\,\de y
        \quad
        \text{weakly in }L^p(\Omega;\R^N).
    \end{equation}
    Further, it holds that
    \begin{equation}\label{eq:chi-eps}
        \chi_{i,\eps} \strongts \chi_i\text{ strongly two-scale in }L^p,
    \end{equation}
    whence, by the first two convergences in \eqref{eq:apriori-compact}, we find
    \[
        u_0(x,y) = \chi_{0}(y) u(x,y)
        \quad\text{almost everywhere in }\Omega\times Q.
    \]
    In particular, also the first equality in \eqref{eq:cluster1} is satisfied.
    
    By linearity,
    the previous considerations also entail that
    $\eps \chi_{1,\eps} u_\eps \weakts (1-\chi_0) u$ weakly two-scale in $L^p$.
    On the other hand, from \eqref{eq:bds-norms} we infer that 
    $\set{\eps \chi_{1,\eps} u_\eps }$ must actually converge to $0$ strongly in $L^p(\Omega;\R^N)$
    and, {\it a fortiori}, in the strong two-scale sense.
    Thus, $(1-\chi_0) u = 0$,
    which implies that $u(x,y)=0$ if $y\in D_1$.
    We therefore conclude that $u=u_0$.
    
    \noindent (2) Since $\A u_\eps = 0$ in $\Omega$ for all $\eps>0$,
    Proposition \ref{stm:Afree-limits} yields immediately \eqref{eq:cluster2} and \eqref{eq:cluster2bis}.
    As for \eqref{eq:cluster3-bis},
   	given any $\psi \in W^{1,p'}_0(D_1;\RM)$,
   	we extend it to the whole $\Rd$
   	by setting $\psi=0$ in $D_0$ and
   	$\psi(x+z) = \psi(x)$ for all $x\in Q$ and $z\in \Zd$.
   	Fix a function $\eta \in C^1_c(\Omega)$.
   	Since $\A u_\ep=0$ in $\Omega$ for every $\ep>0$,
   	denoting by $\bA^\ast$ the symbol of $\A^\ast$ (see Subsection \ref{subs:diff}),
   	we have
   	\[
   	\begin{split}
   	0 & = \left\langle
   	        \A u_\ep, \eps \eta(\,\cdot\,) \psi\left(\, \frac{\cdot}{\eps} \, \right) 
    	\right\rangle_{W^{-1,p},W^{1,p'}_0} \\
   	& = \int_\Omega
   	        u_\ep(x)\cdot 
   	        \left( \ep \mathbb{A}^*[\nabla \eta(x)] \psi \left( \frac{x}{\ep} \right)
   	        + \eta(x) \A^*\psi \left( \frac{x}{\ep} \right)
   	        \right)
   	    \de{x} \\
   	& = \int_\Omega
   	        \chi_{1,\ep}(x) u_\ep(x) \cdot
   	        \left( \ep \mathbb{A}^*[\nabla \eta(x)] \psi \left( \frac{x}{\ep} \right)
   	        + \eta(x) \A^*\psi \left( \frac{x}{\ep} \right)
   	        \right)
   	    \de{x},
   	\end{split}
    \]
    where the latter equality is due to the choice of the support of $\psi$. Therefore, the third convergence in \eqref{eq:apriori-compact} yields
    \[
        \int_{\Omega} \eta(x)
            \left( \int_Q v_1(x,y)\cdot \A^*\psi (y)\de{y}\right) \de{x} = 0,
    \]
    which in turn implies \eqref{eq:cluster3-bis}.
    
    \noindent (3) For every $\ep>0$ let  $\tilde{u}_\ep$ denote the $\mathscr{A}$-free extension of $u_\ep$
    provided by Assumption \ref{stm:exist-ext}.
    Owing to \eqref{eq:bds-norms}, the sequence
    $\set{\tilde{u}_\ep}$ is bounded in $L^p(\Omega;\R^N)$ and
    there exists $\tilde{u}\in L^p(\Omega;L^p_\per(\Rd;\RN))$ such that (up to subsequences)
    \begin{equation}
    \label{eq:add-tildeu}
    \tilde{u}_\ep\weakts\tilde{u}\quad\text{ weakly two-scale in }L^p
    \quad\text{and}\quad
    \A_x \left( \int_Q\tilde{u}(x,y)\de{y} \right) = 0\text{ in }W^{-1,p}(\Omega;\R^M).
    \end{equation}
    Since Assumption \ref{stm:exist-ext} grants that
    $\chi_{1,\ep}\tilde u_\ep = \chi_{1,\ep}u_\ep$ almost everywhere in $\Omega$,
    we infer \eqref{eq:cluster4}
    from \eqref{eq:apriori-compact} and \eqref{eq:chi-eps}.
    Relationship \eqref{eq:cluster4} in turn rewrites as
    \[
        \A u_1
        = \A_x \left(\int_Q v_1(x,y) \de y \right)
        = \A_x \left(\int_{D_1} \tilde u (x,y) \de y \right)
        = - \A_x \left(\int_{D_0} \tilde u (x,y) \de y \right),
    \]
    i.e., \eqref{eq:cluster5} is proved.
    Finally, to obtain \eqref{eq:cluster5bis},
    we combine the fact that
    $\tilde{u}_\ep = \chi_{1,\ep}u_\ep + \chi_{0,\ep} \tilde{u}_\ep$
    almost everywhere in $\Omega$
    with the assumption that
    $\A \tilde{u}_\ep=0$ for every $\ep>0$.
    Using again Proposition \ref{stm:Afree-limits}, as well as \eqref{eq:cluster3-bis} and \eqref{eq:apriori-compact} -- \eqref{eq:add-tildeu}, we complete the proof of the third statement.
    
    \noindent (4) By \eqref{eq:cluster5bis} and Assumption \ref{stm:null-av}, we deduce that $\int_{D_0} \tilde{u}(x,y)\,\de{y}=0$ for almost every $x\in \Omega$. Then, by \eqref{eq:cluster4},
    $\A u_1=0$ in $\Omega$.
    Besides, in view of its weak two-scale convergence,
    $\set{\tilde u_\eps}$ converges weakly in $L^p(\Omega;\RN)$ to
    \[
        \int_Q \tilde u(\,\cdot\,,y) \de y = \int_Q \chi_{D_1}(y)\tilde u(\,\cdot\,,y) \de y = \int_Q v_1(\,\cdot\,,y) \de y = u_1(\,\cdot\,),
    \]
    where we used \eqref{eq:cluster4} and \eqref{eq:cluster1}.
     This also shows that $\set{u_\eps}$ admits an extension on the `soft' part
    that converges to $u_1$ weakly in $L^p$.
    Therefore, to grant that there is a subsequence of $\set{u_\eps}$
    that converges to $u_1$ in the high-contrast sense,
    we are only left to observe that $\eps u_\eps \weak \int_Q u_0(\,\cdot\,,y)\,\de y$
    and that such average vanishes because of \eqref{eq:cluster1}, \eqref{eq:cluster2bis}, and Assumption \ref{stm:null-av}.
    The proof is now concluded.
    \end{proof}

As already anticipated in Remark \ref{stm:rmkA2},
we note that the preservation of $p$-equiintegrability granted by Assumption \ref{stm:exist-ext}
is not needed to prove item (3) of the previous proposition.
We will instead resort to it to establish Proposition \ref{stm:split2} below.

\subsection{Splitting}
	In the light of the of previous subsection,
	we know that those sequences which are equibounded in energy
	are precompact with respect to the high-contrast convergence, and that
	their cluster points fulfill suitable differential constraints.
	Following the approach of {\sc M. Cherdantsev \& K.\ D. Cherednichenko} \cite{CC},
	the next step is to show that
	the asymptotic behavior of the energy
	along such sequences
	coincides with the sum of those of the energies
	of two decoupled systems,
	one sitting on the `soft' inclusions,
	the other on the `stiff' matrix.
	
	To favor intuition, 
	let us consider a family $\set{u_\eps}\subset L^p(\Omega;\RN)$
	such that $\sup_{\ep>0}\F_\eps(u_\eps) \leq c$ for some $c\geq -a \lambda \Ld(\Omega)$,
	and assume that the operator $\A$ is admissible in the sense of Definition \ref{def:adm}.
	 As Proposition \ref{stm:HC-limits} proves, the growth condition {\bf H3} on the energy densities entails
	$\norm{\chi_{1,\eps} u_\eps}_{L^p(\Omega;\RN)} \leq c$.
	Hence,
	by exploiting Assumption \ref{stm:exist-ext},
	we retrieve the functions $\tilde u_\eps \coloneqq \mathsf E_\A^\ep u_\eps \in L^p(\Omega;\RN)$
	such that for all $\eps$
	\begin{enumerate}
		\item $\chi_{1,\eps}\tilde u_\eps = \chi_{1,\eps} u_\eps$ almost everywhere in $\Omega$,
		\item $\norm{\tilde u_\eps}_{L^p(\Omega;\RN)} \leq c$,
		\item $\A \tilde u_\eps = 0$ in $\Omega$,
	\end{enumerate}
	and also that
	\begin{enumerate}[resume]
		\item $\set{\tilde u_\eps}$ is $p$-equiintegrable if so is $\set{u_\eps}$.
	\end{enumerate}
	Setting $v_\eps \coloneqq  u_\eps - \tilde u_\eps$,
	we rewrite
		\[
		{\F}_\eps(u_\eps) = {\F}_{0,\eps}(v_\eps) + \F_{1,\eps}(\tilde u_\eps) + {\E}_\eps(u_\eps),
		\]
	where ${\F}_{0,\eps}$ and $\F_{1,\eps}$ are as in \eqref{eq:F0eps}--\eqref{eq:F1eps} and
		\begin{equation}\label{eq:error}
			{\E}_\eps(u_\eps) \coloneqq 
					\int_{\Omega_{0,\eps}}
						\left[ {f}_{0,\ep}\left(\eps u_\eps \right) - {f}_{0,\ep}\left(\eps v_\eps \right) \right]\de x.
		\end{equation}
	(notice that $v_\ep(x)=0$ when $x\in \Omega_{1,\ep}$).
	In the lemma below
	we show that
	the error made by substituting ${\F}_\eps(u_\eps)$ by the sum ${\F}_{0,\eps}(v_\eps) + \F_{1,\eps}(\tilde u_\eps)$, i.e. ${\E}_\eps(u_\eps)$,
	is asymptotically negligible
	whenever high-contrast convergence holds.

\begin{lemma}[Splitting]
\label{stm:split}
     Let $\A$ be a constant rank differential operator of the form \eqref{eq:op-A}.
    Let also $u\in L^p(\Omega;\RN)$ be the high-contrast limit of a family $\set{u_\eps} \subset L^p(\Omega;\R^N)$.
    Explicitly, assume that $\eps u_\eps \weak 0$ weakly in $L^p(\Omega;\R^N)$ and that
    there is a family $\set{\tilde u_\eps}\subset L^p(\Omega;\R^N)$ with the properties that  $\A\tilde{u}_\ep=0$ in $W^{-1,p}(\Omega;\R^M)$, 
    $\tilde u_\eps \weak u$ weakly in $L^p(\Omega;\R^N)$ and
    $u_\eps = \tilde u_\eps$ in $\Omega_{1,\ep}$.
    If $\sup_{\ep>0}\F_\eps (u_\eps)\leq c$ for some $c\geq -a \lambda_0 \Ld(\Omega)$,
    and if $v_\eps \coloneqq u_\eps - \tilde u_\eps$, then
    the following hold:
    \begin{align}
    & \nonumber \A u = 0 \text{ in } W^{-1,p}(\Omega;\RM), \\
    & \nonumber v_\eps=0 \text{ in } \Omega_{1,\ep}, \\
    & \nonumber \A v_\ep = 0 \text{ in } W^{-1,p}(\Omega;\RM),\\
    & \nonumber \eps v_\eps \weak 0\text{ weakly in }L^p(\Omega;\RN), \\  
    & \liminf_{\eps \to 0} \F_{0,\eps}(v_\eps) + \liminf_{\eps \to 0} \F_{1,\eps}(\tilde u_\eps) \leq \liminf_{\eps \to 0} \F_\eps(u_\ep), \label{eq:split5} \\
    & \nonumber \limsup_{\eps\to 0} \F_\eps(u_\ep) \leq
    \limsup_{\eps \to 0} \F_{0,\eps}(v_\eps) + \limsup_{\eps \to 0} \F_{1,\eps}(\tilde u_\eps).
    \end{align}
\end{lemma}
\begin{proof}
	 As a consequence of the fact that $\A\tilde{u}_\ep=0$ in $\Omega$ for all $\ep$, also the weak limit of $\set{\tilde u_\eps}$, i.e. $u$, must be $\A$-free.
    As immediate consequences of the definition, we also find that
    $v_\ep=0$ in $\Omega_{1,\ep}$ and that $\A v_\eps = 0$ in $\Omega$.
    Being $\set{\tilde{u}_\eps}$ bounded in $L^p(\Omega;\RN)$, we also deduce
    \begin{equation*}
        \eps v_\eps = \eps ( u_\eps - \tilde u_\eps ) \weak 0 \quad\text{weakly in }L^p(\Omega;\RN).
    \end{equation*}
    
    Only the estimates involving the semilimits are now left to prove.
    It suffices to show that
    $\lim_{\eps\to 0} {\E}_\eps(u_\eps) = 0$, with $\E_\eps$ as in \eqref{eq:error}.
    To this aim, we observe that in view of {\bf H4}, for almost every $x$ we have
	\begin{gather*}
		\left\lvert
		{f}_{0,\eps} \left(\eps u_\eps \right)
		- {f}_{0,\eps} \left(\eps v_\eps \right)
		\right\rvert
		\leq \mu \left(
		    1 + \va{\eps v_\eps}^{p-1} + \va{\eps u_\eps}^{p-1}
	    \right)\va{\eps \tilde u_\eps}.
	\end{gather*}
	H\"older's Inequality yields
	$
		\va{{\E}_\eps(u_\eps)}
		\leq c 
		    \eps \norm{\tilde u_\eps}_{L^p(\Omega;\RN)},
	$
	and the conclusion is achieved
	by exploiting again the boundedness of $\set{\tilde u_\eps}$ in $L^p(\Omega;\RN)$.
\end{proof}

We conclude this section
by showing that,
thanks to the $\A$-free decomposition procedure in Lemma \ref{stm:Adecomp},
we can always reduce to the case in which
the sequence $\set{v_\ep}$ on the left-hand side of \eqref{eq:split5} is $p$-equiintegrable.
This motivates the $\Gamma$-liminif inequality
contained in Proposition \ref{prop:liminf-gen}.
We premise a lemma.
\begin{lemma}\label{stm:substitution}
	Let $\set{f_{0,\ep}}$ and $f_1$ satisfy hypotheses {\bf H1}, {\bf H3}, and {\bf H4}.
	Let also $\set{u_\epsilon}, \set{v_\epsilon} \subset L^p(\Omega;\RN)$ be bounded sequences such that
	$u_\epsilon - v_\epsilon\to 0$ in measure, and
	that $\set{v_\epsilon}$ is $p$-equiintegrable.
	Then,
	\begin{gather}
	    \liminf_{\epsilon\to 0} \int_{\Omega_{0,\eps}} \left[
			    {f}_{0,\eps} \left( u_\eps \right)
			    - {f}_{0,\eps} \left(v_\eps \right)
			    \right] \de x
		    \geq 0, \label{eq:substitution} \\
	    \liminf_{\epsilon\to 0} \int_{\Omega_{1,\eps}} \left[
			    {f}_{1} \left(\frac{x}{\ep}, u_\eps \right)
			    - {f}_{1} \left(\frac{x}{\ep},v_\eps \right)
			    \right] \de x
		    \geq 0. \nonumber
    \end{gather}
\end{lemma}
\begin{proof}
Despite the dependence of the energy density $f_1$
on the oscillating variable,
up to a different notational realization,
the ‘stiff’ case is completely analogous to the ‘soft’ one.
For this reason, we
detail the proof of \eqref{eq:substitution} only

We first note that
\eqref{eq:substitution} is left unchanged
if we replace ${f}_{0,\eps}$ with ${f}_{0,\eps}-{f}_{0,\eps}(0)$, and hence
we may assume that ${f}_{0,\eps}(0)=0$. Upon extraction of a (non-relabeled) subsequence, we may also suppose that 
the left-hand side in \eqref{eq:substitution} is a limit.
In view of \cite[Lemma~8.13]{fonseca.leoni},
for every $\eps>0$,
we can decompose the elements of such subsequence as $u_\eps=u^o_\eps+u^c_\eps$,
where $\set{u^o_\eps}$ (the `oscillating' part) is $p$-equiintegrable and $\set{u_\eps^c}$ (the `concentrating' part) converges to zero in measure.
If we let $R_\eps \coloneqq \set{x\in\Omega: u^o_\eps\ne u_\eps}$,
we have that $\Ld(R_\eps)\to 0$, whence, by the current assumptions,
$\set{u^o_\epsilon-v_\eps}$ is $p$-equiintegrable and converges to zero in measure.
Thanks to Vitali's convergence theorem (see e.g.~\cite[Theorem~2.24]{fonseca.leoni}),
it follows that $u^o_\epsilon-v_\eps \to 0$
strongly $L^p(\Omega;\RN)$, and,
in view of {\bf H4}, we deduce 
\begin{align}\label{strong}
\lim_{\eps\to 0}\int_{\Omega_{0,\eps}}\left|{f}_{0,\eps}(u^o_\eps)-{f}_{0,\eps}(v_\eps)\right| \de x=0 .
\end{align}

By employing the definition of $R_\eps$, {\bf H3}, and {\bf H4},
we find the estimate
\begin{multline*}
    \left|
    \int_{\Omega_{0,\eps}}{f}_{0,\eps}(u_\eps)\,\de x - \int_{\Omega_{0,\eps}} \left[{f}_{0,\eps}(u^o_\eps)+{f}_{0,\eps}(u^c_\eps)\right] \de x
    \right| \\
    = \left|
    \int_{\Omega_{0,\eps}\cap R_\eps}{f}_{0,\eps}(u^o_\eps+u^c_\eps)\,\de x - \int_{\Omega_{0,\eps}\cap R_\eps} \left[{f}_{0,\eps}(u^o_\eps)+{f}_{0,\eps}(u^c_\eps)\right] \de x
    \right| \\
    \leq 
    \int_{\Omega_{0,\eps}\cap R_\eps} \left|
        {f}_{0,\eps}(u^o_\eps+u^c_\eps) - {f}_{0,\eps}(u^c_\eps)
    \right| \de x
    + \int_{\Omega_{0,\eps}\cap R_\eps}\left |{f}_{0,\eps}(u^o_\eps)\right| \de x
    \\ \le 
    \mu \int_{\Omega_{0,\eps}\cap R_\eps}\big(1+|u_\eps|^{p-1}+|u^c_\eps|^{p-1}\big)|u^o_\eps| \,\de x
    +
    \Lambda\int_{\Omega_{0,\eps}\cap R_\eps}\big(1+|u^o_\eps|^p\big)\, \de x.
\end{multline*}
Since $\set{u^o_\eps}$ and $\set{u^c_\eps}$ are bounded in $L^p(\Omega;\RN)$ and
$\set{u^o_\eps}$ is $p$-equiintegrable,
we infer that the last term tends to zero as $\eps\to 0$.
Therefore, we see that 
\begin{align*}
    \lim_{\epsilon\to 0} \int_{\Omega_{0,\eps}}
        \left[
            f_{0,\eps} ( u_\eps ) - f_{0,\eps} ( v_\eps )
		\right] \de x
	& = \lim_{\epsilon\to 0} \int_{\Omega_{0,\eps}}
	    \left[
	        f_{0,\eps} ( u^o_\eps ) + f_{0,\eps} ( u^c_\eps ) - f_{0,\eps} ( v_\eps )
	    \right] \de x \\
    & = \lim_{\epsilon\to 0} \int_{\Omega_{0,\eps} \cap R_\eps}
	        f_{0,\eps} ( u^c_\eps ) \de x. \\
	& \geq \liminf_{\epsilon\to 0} \int_{\Omega_{0,\eps}\cap R_\eps}\lambda\big(-a+|u^c_\eps|^p\big)\de x\\
	& \ge 0,
\end{align*}
where the second equality is a consequence of \eqref{strong},
the second to last inequality is due to {\bf H3},
and the last one follows from $\Ld(R_\eps)\to 0$.
\end{proof}

The improved variant of Lemma \ref{stm:split} for the liminf inequality in \eqref{eq:split5} reads as follows.

\begin{proposition}[Refined splitting for the liminf]
\label{stm:split2}
     Let $\A$ be an admissible differential operator.
    Let $u\in L^p(\Omega;\RN)$ and consider two sequences $\set{u_\eps},\set{\tilde u_\eps}\subset L^p(\Omega;\RN)$ with the property that 
    $\eps u_\eps \weak 0$ weakly in $L^p(\Omega;\R^N)$, $\A\tilde{u}_\ep=0$ in $W^{-1,p}(\Omega;\R^M)$, 
    $\tilde u_\eps \weak u$ weakly in $L^p(\Omega;\R^N)$, and
    $u_\eps = \tilde u_\eps$ in $\Omega_{1,\ep}$.
    If $\sup_{\ep>0}\F_\eps (u_\eps)\leq c$ for some $c\geq -a \lambda_0 \Ld(\Omega)$,
    then there exists $\set{\tilde v_\eps} \subset L^p(\Omega;\RN)$ such that the following hold:
    \begin{align}
    & \nonumber \tilde v_\eps=0 \text{ in } \Omega_{1,\ep}, \\
    & \nonumber \A \tilde v_\ep = 0 \text{ in } W^{-1,p}(\Omega;\RM),\\
    & \set{\ep \tilde{v}_\ep} \text{ is $p$-equiintegrable,} \nonumber \\
    & \eps \tilde v_\eps \weak 0\text{ weakly in }L^p(\Omega;\RN), \label{eq:epvep}\\  
    & \liminf_{\eps \to 0} \F_{0,\eps}(\tilde v_\eps) + \liminf_{\eps \to 0} \F_{1,\eps}(\tilde u_\eps) \leq \liminf_{\eps \to 0} \F_\eps(u_\ep), \label{eq:split5-2}.
    \end{align}
\end{proposition}
\begin{proof}
    The equiboundedness in energy entails that
    \begin{equation}
    \label{eq:bd-from-en}    
    \norm{\eps u_\eps}_{L^p(\Omega;\RN)}\leq c
    \quad\text{and}\quad
    \norm{\chi_{1,\ep}u_\ep}_{L^p(\Omega;\RN)}\leq c.
    \end{equation}
    By Lemma \ref{stm:Adecomp},
    there exist a (not relabeled) subsequence of $\set{u_\ep}$ and a sequence $\set{w_\ep}\subset L^p(\Omega;\RN)$ such that 
	\begin{enumerate}
		\item $\set{\ep w_\ep}$ is bounded, $\A$-free, and $p$-equiintegrable;
		\item $\ep( w_{\eps} - u_\ep ) \to 0$ strongly in $L^q(\Omega;\R^N)$ for every $q\in[1,p)$;
		\item $\ep \chi_{1,\ep}(w_\ep - u_\ep)\to 0$ strongly in $L^p(\Omega;\R^N)$.
	\end{enumerate}
	Note that the last property is not mentioned explicitly in the statement of Lemma \ref{stm:Adecomp},
	but it is a byproduct of the construction in its proof (see \cite[Lemma 2.15]{fonseca.muller}).
	
	For every $\ep>0$, we define $\tilde w_\ep \coloneqq \mathsf{E}^\ep_{\A} w_\ep$,
	with $\mathsf{E}^\ep_{\A}$ as in Assumption \ref{stm:exist-ext}.
	We then set $\tilde v_\ep \coloneqq w_\ep - \tilde w_\ep$.
	In view of Assumption \ref{stm:exist-ext},
	we immediately obtain that
	$\tilde v_\ep$ vanishes on $\Omega_{1,\ep}$, that it is $\A$-free in $\Omega$, and that $\set{\ep \tilde{v}_\ep}$ is $p$-equiintegrable.
	To prove \eqref{eq:epvep}, we consider the identity
	\begin{equation}\label{eq:rewrite-vep}
	\ep \tilde v_\ep= \ep (w_\ep - u_\ep) - \ep (\tilde w_\ep - u_\ep).
	\end{equation}
	Thanks to (2) above, $\ep( w_{\eps} - u_\ep ) \to 0$ strongly in $L^q(\Omega;\R^N)$.
	Additionally, by Assumption \ref{stm:exist-ext},
	$$
	\ep \norm{\tilde w_\ep}_{L^p(\Omega;\R^N)}
	\leq c \ep \norm{w_\ep}_{L^p(\Omega_{1,\ep};\R^N)}
	\leq c\ep \norm{w_\ep-u_\ep}_{L^p(\Omega_{1,\ep};\R^N)}+c\ep\norm{u_\ep}_{L^p(\Omega_{1,\ep};\RN)}
	\to 0,$$
	the convergence to $0$ following from (3) above and \eqref{eq:bd-from-en}.
	In view of the assumptions on $\set{u_\eps}$ and of (2), \eqref{eq:epvep} is inferred.
	
	To complete the proof of the corollary, it suffices now show that for $q\in [1,p)$,
	\begin{align}
	    \ep (\tilde{v}_\ep-v_\ep)\to 0 \quad\text{ strongly in }L^q(\Omega;\R^N),
	    \label{eq:change-small}
	\end{align}
	where $v_\eps \coloneqq u_\eps - \tilde u_\eps$.
	Indeed, once this is proven, \eqref{eq:split5-2} is deduced from Lemma \ref{stm:split} and Lemma \ref{stm:substitution}.
	To prove \eqref{eq:change-small},
	we notice that, by the definitions of $v_\ep$ and $\tilde v_\ep$, as well as by H\"older's inequality,
	for every $q\in [1,p)$ it holds
	\begin{equation*}
	\norm{\ep (\tilde{v}_\ep - v_\ep)}_{L^q(\Omega;\R^N)}
	\leq \ep \norm{w_\ep - u_\ep}_{L^q(\Omega;\R^N)} + c\ep \norm{\mathsf{E}^\ep_{\A}(w_\ep - u_\ep)}_{L^p(\Omega;\R^N)}.
	\end{equation*}
	Then, in view of (2), of Assumption \ref{stm:exist-ext}, and of (3), we deduce \eqref{eq:change-small}.
\end{proof}

\subsection{$\Gamma$-convergence}
\label{sec:main-result}

In this short subsection
we tackle the proofs of Theorem \ref{stm:main} and Corollary \ref{stm:minimizers}.

\begin{proof}[Proof of Theorem \ref{stm:main}]
	Statement (1),
	concerned with the high-contrast compactness of families with equibounded energy,
	follows from Proposition \ref{stm:HC-limits}.
	
	Let now $u\in U_1 $ be fixed and assume that
	there are two sequences $\set{u_\eps},\set{\tilde u_\eps}\subset L^p(\Omega;\RN)$ with the property that 
	$\eps u_\eps \weak 0$ weakly in $L^p(\Omega;\R^N)$,  $\A\tilde{u}_\ep=0$ in $\Omega$,
	$\tilde u_\eps \weak u$ weakly in $L^p(\Omega;\R^N)$, and
	$u_\eps = \tilde u_\eps$ in $\Omega_{1,\ep}$.
	If the lower limit of $\set{\F_\ep(u_\eps) }$ is not finite, the estimate of point (2) holds trivially.
	Otherwise, we define $v_\eps \coloneqq u_\eps - \tilde u_\eps$ and,
	owing to Proposition \ref{stm:split2},
	by applying Proposition \ref{stm:Glim-soft} to $\set{v_\eps}$,
	as well as Proposition \ref{prop:stiff-part} to $\set{\tilde u_\eps}$,
	we deduce
	\begin{align*}
	\liminf_{\ep\to 0} \F_\ep(u_\eps)
	& \geq \liminf_{\eps \to 0} \F_{0,\eps}(u_\eps) + \liminf_{\eps \to 0} \F_{1,\eps}(u_\eps) \\
	& = \liminf_{\eps \to 0} \F_{0,\eps}(v_\eps) + \liminf_{\eps \to 0} \F_{1,\eps}(\tilde u_\eps) \\
	& \geq \alpha_0 + \F_1(u) \\
	& \eqqcolon \F(u).
	\end{align*}	
	
	Finally, we turn to (3).
	For $u\in U_1$, again in the light of Propositions \ref{stm:Glim-soft} and \ref{prop:stiff-part},
	we can find two sequences $\set{v_\eps},\set{\tilde u_\eps} \subset U_1$
	such that
	\begin{align*}
		& v_\eps = 0 \quad\text{in } \Omega_{1,\ep}, \\
		& \eps v_\eps \weak 0 \quad\text{weakly in } L^p(\Omega;\RN), \\
		& \tilde u_\eps \weak u \quad\text{weakly in } L^p(\Omega;\RN), \\
		& \limsup_{\eps \to 0} \F_{0,\eps}(u_{0,\eps}) \leq \alpha_0, \\
		& \limsup_{\eps \to 0} \F_{1,\eps}(\tilde u_\eps) \leq \F_1(u).
	\end{align*}
	Then, setting $u_\eps \coloneqq v_\eps+\tilde u_\eps$,
	by Lemma \ref{stm:split} we deduce the desired limsup inequality 
	from the ones satisfied by $\set{\F_{0,\eps}(v_\eps)}$ and $\set{\F_{1,\eps}(\tilde u_\eps)}$:
	\begin{align*}
		\limsup_{\eps \to 0} \F_\eps(u_\eps)
			& \leq \limsup_{\eps \to 0} \F_{0,\eps}(u_\eps) + \limsup_{\eps \to 0} \F_{1,\eps}(u_\eps) \\
			& = \limsup_{\eps \to 0} \F_{0,\eps}(v_\eps) + \limsup_{\eps \to 0} \F_{1,\eps}(\tilde u_\eps) \\
			& \leq \alpha_0 + \F_1(u) \\
			& \eqqcolon \F(u).
	\end{align*} 
\end{proof}

At this stage, the convergence of infima and of minimizers easily follows.

\begin{proof}[Proof of Corollary \ref{stm:minimizers}]
The equicoercivity of the energies with respect to the high-contrast convergence and the convergence result in Theorem \ref{stm:main} yield the conclusion by standard $\Gamma$-convergence arguments.
\end{proof}

\section{Asymptotics for the soft component}\label{sec:soft}
This section is devoted to the proof of Proposition \ref{stm:Glim-soft}. 
We address the asymptotic analysis of the energy stored in the `soft' component of the system, that is, $\set{\F_{0,\eps}}$ in \eqref{eq:F0eps},
and we prove that the limiting behavior is encoded by the functional $\F_0$ in \eqref{eq:F0}.
We recall that $\F_0$ is finite only when $u=0$
and in that case the value of the functional equals the constant $\alpha_0$
given by
\begin{equation*}
	\alpha_0 \coloneqq \sup_{\Omega'\Subset\Omega} \inf_{u \in U_0(\Omega')} \int_{\Omega'}\int_{D_0} f_{0} \big( u(x,y) \big) \de y\,\de x,
\end{equation*}
where the supremum is meant to run over all open sets that are compactly contained in $\Omega$,
and, for any open $\Omega' \subset \Omega$, $U_0(\Omega')$ is as in \eqref{eq:U0}.
	
\begin{remark}[On the constant $\alpha_0$]
	\label{stm:rmkG0}
	By the definition of $U_0(\Omega')$ in \eqref{eq:U0}, we see that
	\begin{equation}\label{eq:U0-monotone}
	U_0(\Omega) \subset U_0(\Omega'_2) \subset U_0(\Omega'_1)
	\quad\text{whenever } \Omega'_1\subset\Omega'_2 \subset \Omega.
	\end{equation}
	Since $\lambda(-a + \va{\xi}^p)\leq f_0(\xi)$ for every $\xi\in \R^N$,
	from \eqref{eq:U0-monotone} we deduce that
	\[
		\alpha_0 \leq \inf_{u \in U_0}\int_{\Omega}\int_{D_0} f_{0} \big( u(x,y) \big) \de y\,\de x,
	\]
	and, more generally, that
	\begin{equation}\label{eq:inf-monoton}
		\inf_{u\in U_0\in (\Omega'_1)}\int_{\Omega'_1}\int_{D_0} f_0 \big( u(x,y) \big) \de y\,\de x
		\leq \inf_{u\in U_0(\Omega'_2)}\int_{\Omega'_2}\int_{D_0} f_0 \big( u(x,y) \big) \de y\,\de x
	\end{equation}
	when $\Omega'_1 \subset \Omega'_2 \subset \Omega$.
	In particular, the supremum in the definition of $\alpha_0$ is not a maximum, in general.
	Indeed, suppose that $\Omega'_1\Subset \Omega$ achieves such supremum.
	Then, we would get a contradiction whenever there exists $\Omega'_2 \Subset \Omega$ containing $\Omega'_1$
	such that the inequality in \eqref{eq:inf-monoton} is strict.
\end{remark}

We separate the discussion of the $\Gamma$-liminf and of the $\Gamma$-limsup inequalities.

\subsection{$\Gamma$-liminf inequality}
The proof of the $\Gamma$-liminf inequality relies
on the compactness and splitting results of Section \ref{sec:comp+split},
on unfolding techniques, as well as on the following intermediate statement.

\begin{lemma}
	\label{stm:liminfineq-Aquasiconv}
	Let $\set{{f}_{0,\ep}}$ satisfy {\bf H1} and {\bf H3}--{\bf H5}.
	Suppose that $\set{w_\eps}\subset L^p(\Omega; L^p_\per(\Rd;\RN))$ is a $p$-equiintegrable family satisfying the following properties:
	\begin{align}\label{eq:0inD1}
	w_\eps = 0 \text{ if } y \in D_1,
	\end{align}
	and for all open $\Omega'\Subset \Omega$ there exists $\eps'\coloneqq \eps'(\Omega')$ such that
	\begin{align}\label{eq:Ayfree}
	\A_y w_\eps=0 \text{ in } W^{-1,p}(\Td;\RM)
		\text{ for a.e. } x\in\Omega' \text{ if } \eps < \eps'.
	\end{align}
	Then, it holds
	\begin{align}\label{lsc}
	\alpha_0	\leq \liminf_{\eps \to 0} \int_{\Omega}\int_{D_0} {f}_{0,\ep} \big( w_\eps(x,y) \big)\de y\,\de x.
	\end{align}
\end{lemma}
\begin{proof}
	The proof is divided in two steps.
	
	\noindent
	{\sc Step 1: a simplified setting.} We start by establishing an inequality for the case in which
	the densities on the right-hand side do not depend on $\eps$ (cf. \cite[Lemma 20]{CC}), that is, we first prove that
	\begin{equation}\label{eq:noeps}
	\alpha_0
	\leq 
	\liminf_{\eps \to 0}
	\int_{\Omega}\int_{D_0} f_0 \big( w_\eps(x,y) \big) \de y\,\de x.
	\end{equation}
	
	According to the definition of $\alpha_0$ and to Remark \ref{stm:rmkG0},
	for an arbitrary $\delta>0$,  there is an open set $\Omega'_\delta\Subset \Omega$ such that
	$\dist(\partial \Omega,\partial \Omega'_\delta) < \delta$ and
	\[
		\alpha_0 \leq \inf_{u\in U_0(\Omega'_\delta)}\int_{\Omega'_\delta}\int_{D_0} f_0 \big( u(x,y) \big) \de y\,\de x + \delta.
	\] 
	In view of \eqref{eq:0inD1} and \eqref{eq:Ayfree},
	$w_\ep \in U_0(\Omega'_\delta)$ for $\ep>0$ sufficiently small, whence
	\begin{align*}
		\alpha_0
		& \leq
		\liminf_{\eps \to 0} \int_{\Omega'_\delta}\int_{D_0} f_0 \big( w_\eps(x,y) \big) \de y\,\de x + \delta \\
		& =
		\liminf_{\eps \to 0} \int_{\Omega}\int_{D_0} f_0 \big( w_\eps(x,y) \big) \de y\,\de x
			- \int_{\Omega \setminus \Omega'_\delta}\int_{D_0} f_0 \big( w_\eps(x,y) \big) \de y\,\de x + \delta.
	\end{align*}
	Thanks to the $p$-equiintegrability of $\set{w_\ep}$,
	the conclusion \eqref{eq:noeps} now follows	by letting $\delta$ vanish.
	
	\noindent{\sc Step 2: the general case.}
	\noindent We adapt the argument in the proof of \cite[Thm.~5.14]{DalMaso93},
	relying on the almost everywhere convergence of $\set{{f}_{0,\ep}}$ to ${f}_0$, as well as on the $p$-Lipschitz continuity of these energy densities (see {\bf H4} and {\bf H5}). 
	
	Fix $\delta>0$. By the $p$-equiintegrability of $\set{w_\ep}$ and by the $p$-growth assumptions on $f_0$, there exists $r>0$ such that 
	\begin{equation}
	\label{eq:first-part}
	\sup_{\ep>0}\int_{\set{(x,y)\in \Omega\times D_0: \va{w_\ep(x,y)}>r}}f_0 \big( w_\ep(x,y) \big)\de{x}\,\de{y}\leq \delta.
	\end{equation}
	Besides, since $f_{0,\ep}$ and $f_0$ are $p$-Lipschitz, we find $\rho>0$ such that
	\begin{equation}
	\label{eq:second-part}
	|f_0(\xi)-f_0(\eta)|+\sup_{\ep>0}|f_{0,\ep}(\xi)-f_{0,\ep}(\eta)|\leq \delta\quad\text{for every } \xi,\eta \in B(0,\rho).
	\end{equation}
	Let now $\xi_1,\dots,\xi_K\in B(0,r)$ be such that
	\begin{equation}\label{eq:cover}
	\overline{B(0,r)} \subset \bigcup_{k=1}^K B\left(\xi_k,\rho\right).
	\end{equation} 
	In view of {\bf H5}, for any such $\xi_k$
	there exist $\bar \eps_k>0$
	such that 
	$\va{ {f}_{0,\ep}(\xi_k) - {f}_0(\xi_k)}\le \delta$ 
	if $\eps<\bar \eps_k$.
	Letting 
	$\bar \eps \coloneqq \min\set{\bar \eps_1,\dots,\bar \eps_K}$,
	it follows that 
	for any $k=1,\dots,K$
	\begin{equation}\label{eq:unifconv}
	\va{{f}_{0,\ep}(\xi_k) - {f}_0(\xi_k)}\le \delta
	\quad
	\text{if $\eps<\bar \eps$.}
	\end{equation}
	By \eqref{eq:cover},
	for every $\eta \in \overline{B(0,r)}$
	there exists $k\in\{1,\dots,K\}$ such that $\eta \in B ( \xi_k , \rho )$.
	For this particular $k$,
	the combination of the triangle inequality, \eqref{eq:second-part}, and \eqref{eq:unifconv} yields
	\begin{align}
	\label{eq:lipschitz-est}
	\va{{f}_{0,\ep} ( \eta ) - {f}_0 ( \eta )} 
	\le \va{{f}_{0,\ep} ( \eta ) - {f}_{0,\ep}( \xi_k)}
	+ \va{{f}_{0,\ep}( \xi_k) - {f}_0( \xi_k)} 
	+ \va{{f}_0 (\eta ) - {f}_0( \xi_k)}\leq 3\delta,
	\end{align}
	for every $\eta \in \overline{B(0,r)}$ and every $\ep<\bar\ep$.
	
	In view of \eqref{eq:noeps} we deduce	
	\begin{align*}
	\alpha_0 & \leq
		\liminf_{\eps \to 0}
		\int_{\Omega}\int_{D_0}
			f_0 \big(w_\eps(x,y) \big)
		\de y\,\de x \\
	&\leq
		\liminf_{\ep\to 0}\int_{\set{(x,y)\in \Omega\times D_0: \va{w_\ep(x,y)} \leq r}}
			f_0\big(w_\eps(x,y) \big)
		\de{x}\,\de{y}
		+ \delta \\
	&\leq
		\liminf_{\eps \to 0}\int_{\Omega}\int_{D_0}
			f_{0,\ep} \big(w_\eps(x,y) \big)
		\de y\,\de x
		+3\delta \mathscr{L}^{2d}(\Omega\times D_0)+\delta,
	\end{align*}
	where the first inequality is due to \eqref{eq:first-part}, and the second one to \eqref{eq:lipschitz-est}.
	The arbitrariness of $\delta>0$ yields the conclusion.
\end{proof}

We are now ready to tackle the $\Gamma$-liminf inequality for the `soft' contribution;
the argument is comparable to the one in \cite[Lemma~21]{CC}.

\begin{proposition}\label{prop:liminf-gen}
Let $\set{{f}_{0,\ep}}_\ep$ satisfy assumptions {\bf H1} and {\bf H3}--{\bf H5}, and
let $u\in L^p(\Omega;\R^N)$.
Then, for every $p$-equiintegrable sequence $\set{u_\ep}\subset L^p(\Omega;\R^N)$
such that $u_\ep=0$ on $\Omega_{1,\ep}$,
$\A u_\ep=0$ in $W^{-1,p}(\Omega;\R^M)$ for every $\ep>0$,
and $\ep u_\ep\wk u$ weakly in $L^p(\Omega;\R^N)$,
we have
$$\F_0(u)\leq \liminf_{\ep\to 0}\F_{0,\ep}(u_\ep).$$
\end{proposition}

\begin{proof}
If the lower limit of $\F_{0,\ep}(u_\eps)$ is not finite,
then there is nothing to prove.
Otherwise, without loss of generality,
we focus on the case in which the lower limit of $\F_{0,\ep}(u_\eps)$ is finite, and
we assume that it is a limit.
Then, it follows by Proposition \ref{stm:HC-limits} and Assumption \ref{stm:null-av} that necessarily $u=0$.
Therefore, we are left to prove that
\[
	\alpha_0 \leq \lim_{\eps\to0} \int_{\Omega_{0,\ep}} f_{0,\ep} \big( \eps u_\ep(x) \big) \de{x},
\]
where $\alpha_0$ is given by \eqref{eq:alpha0} and $\eps u_\ep\wk 0$ weakly in $L^p(\Omega;\R^N)$.

We exhibit a formula for $\F_{0,\ep}(u_\ep)$
that involves the unfolding operator introduced in Lemma \ref{stm:unfolding}.
We let $\hat u_\epsilon$ be the unfolded map of $\ep u_\epsilon$:
		\[
		\hat u_\eps(x,y) \coloneqq  \ep \mathsf{S}_\eps u_\eps(x,y),
		\]
		whence
		\begin{equation}\label{eq:vanish}
		 \hat{u}_\ep(x,y)=0\quad\text{for almost every }x\in \Omega\text{ and }y\in D_1.
		\end{equation}
		Also, since the unfolding procedure preserves $p$-equiintegrability, 
		$\set{\hat{u}_\ep} \subset L^p(\Omega\times Q;\RN)$ is $p$-equiintegrable.
		Finally, Lemma \ref{stm:Afree-unfold} grants that
		\begin{equation*}		
		\A_y \hat{u}_\ep=0\quad\text{in }W^{-1,p}(\mathbb{T}^d;\R^M)\quad\text{for a.e. }x\in \hat{\Omega}_\ep,
		\end{equation*}
		where $\hat \Omega_\eps$ is as in \eqref{eq:def-new-set}.
		In particular, if $\Omega' \subset \Omega$ is an open set such that
		the distance $\delta \coloneqq \dist(\partial \Omega',\partial \Omega)$	is strictly positive,
		it is clear that $\Omega'\subset \hat \Omega_\epsilon$ if $\sqrt{d}\eps < \delta$.
		Thus, for all open $\Omega'\Subset \Omega$ 
		\begin{equation}\label{eq:A-ker}
		\A_y \hat{u}_\eps=0 \text{ in } W^{-1,p}(\Td;\RM)
		\text{ for a.e. } x\in\Omega' \text{ if } \eps < \eps',
		\end{equation} 
		where $\eps'>0$ is a suitable threshold depending on $\Omega'$.
		
		By the definition of $\Omega_{0,\eps}$ (see \eqref{eq:om0}),
		we have
		\begin{equation*}
		\F_{0,\eps}(u_\ep) 
		= \eps^d \sum_{z\in Z_\eps} \int_{D_0} {f}_{0,\eps} \Big(\ep u_\eps \big(\epsilon(y+z) \big) \Big) \de y.
		\end{equation*}
		Being $z$ an integer-valued vector,
		it holds $\hat u_\eps (\eps z, y) = \ep u_\eps \big(\epsilon(y+z) \big)$ for every $z\in Z_\ep$ and $y\in D_0$. Hence, we obtain
		\begin{equation*}
		\begin{split}
		\F_{0,\ep}(u_\ep) & =
		\eps^d \sum_{z\in Z_\eps} \int_{D_0} {f}_{0,\eps} \big(\hat u_\eps (\eps z, y) \big)\, \de y \\
		& = \sum_{z\in Z_\eps} \int_{\eps (Q+z)}
		\int_{D_0} {f}_{0,\eps} \left( \hat u_\eps \left( \eps  \left\lfloor \frac{x}{\eps}\right\rfloor, y \right) \right)  \de y\,\de x\\
		& \geq \int_{\hat\Omega_\eps}
		\int_{D_0} {f}_{0,\eps} \left( \hat u_\eps ( x, y ) \right)  \de y\,\de x
		\end{split}
		\end{equation*}
		because for all $x\in \eps (Q+z)$, $\lfloor x / \eps \rfloor = z$.
		We can rewrite the last estimate as follows:
		\begin{equation*}
			\F_{0,\eps}(u_\ep) \geq
			\int_{\Omega}\int_{D_0} {f}_{0,\eps} \left( \hat u_\eps ( x, y ) \right)\,  \de y\,\de x 
			- \int_{\Omega\setminus \hat\Omega_\eps}\int_{D_0} {f}_{0,\eps} \big( \hat u_\eps ( x, y ) \big) \de y\,\de x.
		\end{equation*}
		From this, owing to \eqref{eq:vanish} and \eqref{eq:A-ker},
		we can invoke Proposition \ref{stm:liminfineq-Aquasiconv} to infer
		\begin{equation*}
		\F_{0,\eps}(u_\ep) \geq \alpha_0	- \int_{\Omega\setminus \hat\Omega_\eps}\int_{D_0} {f}_{0,\eps} \big( \hat u_\eps ( x, y ) \big) \de y\,\de x.
		\end{equation*}
		In view of the $p$-equiintegrability of $\set{\hat u_\eps}$,
		the desired liminf inequality is now obtained by taking the limit as $\eps\to 0$.
\end{proof}

\subsection{$\Gamma$-limsup inequality}
We now turn to the $\Gamma$-limsup inequality for the energy functional associated to the `soft' portion of the material. 
The optimality of the lower bound identified in Proposition \ref{prop:liminf-gen} hinges upon the next proposition.

\begin{proposition}	
\label{prop:limsup-soft}
	Let $\set{f_{0,\ep}}$ satisfy {\bf H1} and {\bf H3}--{\bf H5}, and
	let $\Omega'\Subset\Omega$ be open.
	Then, for all $w\in U_0(\Omega')$,
	there exists a family $\set{u_\eps}\subset L^p(\Omega;\R^N)$
	that satisfies the following:
	\begin{enumerate}
		\item $u_\ep=0$ on $\Omega_{1,\ep}$ and
		$u_\eps$ is $\A$-free in $\Omega$ for all $\eps>0$;
		\item $\ep u_\eps \strongts w'$ strongly two-scale in $L^p$, where $w'(x,y) \coloneqq \chi_{\Omega'}(x)w(x,y)$; 
		\item it holds that
			\[
				\limsup_{\eps \to 0} \F_{0,\ep}(u_\eps) \leq \int_{\Omega'}\int_{D_0}f_0(w(x,y))\,\de{y}\,\de{x}.
			\]
	\end{enumerate}
	\end{proposition}

	We postpone the proof of Proposition \ref{prop:limsup-soft} to the end of the section, and
	we discuss next how the $\Gamma$-limsup inequality is deduced from it. 	
	
	\begin{corollary}\label{cor:limsup-soft}
	Let $\set{f_{0,\ep}}$ satisfy {\bf H1} and {\bf H3}--{\bf H5}, and let $u\in L^p(\Omega;\R^N)$. Then, 
	there exists a family $\set{u_\eps}\subset L^p(\Omega;\R^N)$ such that
	the following holds:
	\begin{enumerate}
		\item $u_\ep=0$ on $\Omega_{1,\ep}$ and
		$u_\eps$ is $\A$-free in $\Omega$ for all $\eps>0$;
		\item $\ep u_\eps \wk u$ weakly in $L^p(\Omega;\R^N)$;
		\item the upper limit inequality
			\[
				\limsup_{\eps \to 0} \F_{0,\ep}(u_\eps) \leq \F_0(u)
			\]
			is satisfied.
	\end{enumerate}
	\end{corollary}
	\begin{proof}
	We first remark that if $u\neq 0$, there is nothing to prove.
	Thus, we assume in what follows that $u=0$ and $\F_0(u)=\alpha_0$.
	
	Fix now $\delta>0$ and  $\Omega'\Subset\Omega$. 
	By the definition of $\alpha_0$, there exists a map $w_\delta\in U_0(\Omega')$ such that
	\[
		\int_{\Omega'}\int_{D_0} f_0 \big( w_\delta(x,y) \big)\de{y}\,\de{x}\leq \alpha_0+\delta.
	\]
	Proposition \ref{prop:limsup-soft} yields a sequence $\set{u_\ep}$ such that 
	\begin{enumerate}
		\item $u_\ep=0$ on $\Omega_{1,\ep}$ and
		$u_\eps$ is $\A$-free in $\Omega$ for all $\eps>0$;
		\item $\set{\ep u_\eps}$ converges strongly two-scale in $L^p$ to $\chi_{\Omega'}(x)w_\delta(x,y)$;
		\item the inequality
			\[
				\limsup_{\eps \to 0} \F_{0,\ep}(u_\eps)
				\leq \int_{\Omega'}\int_{D_0}f_0\big( w_\delta(x,y) \big)\de{y}\,\de{x}
			\]
			is satisfied.
	\end{enumerate}
	
	Lemma \ref{lemma:2-scale} and Assumption \ref{stm:null-av} entail that
	$\ep u_\ep \wk 0$ weakly in $L^p(\Omega;\R^N)$.
	Besides, the last estimate and the definition of $w_\delta$ get
	\[
	\limsup_{\eps \to 0} \F_{0,\ep}(u_\eps)	\leq \alpha_0+\delta.
	\]
	The conclusion hence follows then by the arbitrariness of $\delta$
	and by standard properties of $\Gamma$-convergence (see e.g. \cite[Subsection 1.2]{braides}).
	\end{proof}
	We now address the proof of Proposition \ref{prop:limsup-soft}.
	
	\begin{proof}[Proof of Proposition \ref{prop:limsup-soft}]
	We construct a family $\set{\tilde u_\eps}$ that strongly two-scale converges to $w'(x,y) \coloneqq \chi_{\Omega'}(x)w(x,y)$ in $L^p$,
	so that $\set{u_\eps}$ will be given by $u_\eps \coloneqq \tilde u_\eps / \eps$.
	However, because of measurability issues, we cannot deal directly with a general $w\in U_0(\Omega')$.
	The argument is thus subdivided in three parts:
	firstly, we define $\set{\tilde u_\eps}$ under some extra regularity assumption on $w$; 
	then, we prove the limsup inequality when $w$ is regular;
	lastly, we recover the general statement by means of an approximation argument.

	\noindent {\sc Step 1: construction of $\set{\tilde u_\ep}$ when $w$ is regular}.
	In this step we assume that $w$ possesses some extra regularity in the $x$ variable,
	namely we take $w\in U_0(\Omega') \cap L^p_\per(\Rd;C(\bar{\Omega};\RN))$.
	We construct a sequence $\set{\tilde u_\eps}$ of $\A$-free maps such that
	$\tilde u_\ep=0$ on $\Omega_{1,\ep}$ for all $\eps>0$ and that
	$\tilde u_\eps \strongts w'$ strongly two scale in $L^p$.
	Below, when we write $w(x,y)$, we consider the second entry $y$ to be the periodic variable.
	
    Recalling \eqref{eq:om0},
	we define
		\begin{equation*}
		\hat\Omega'_\eps \coloneqq \bigcup_{z\in \hat Z'_\ep} \eps(Q+z), \quad 
		\hat Z'_\ep \coloneqq \set{ z \in \Z^d : \eps (Q + z ) \subset \Omega'} \subset Z_\eps,
		\end{equation*}		
	and, for $\epsilon>0$ and $(\bar x,\bar y) \in \Omega \times \Rd$,
	we consider the averages of $w(\,\cdot\,,\bar y)$
	on the cubes that compound $\hat \Omega'_\epsilon$:
	\begin{equation}
		\label{eq:def-wep}
		w_\eps(\bar x,\bar y) \coloneqq
		\begin{cases}
			\displaystyle \fint_{\eps (Q+z)} w( x , \bar y )\, \de x
			& \text{if } \bar x \in \epsilon(Q+z) \text{ for some } z \in \hat Z'_\eps,\\
			0 & \text{for any other } \bar x\in \Omega.
		\end{cases}
	\end{equation}
	In this way, $w_\ep(\,\cdot\,,y)$ is piecewise constant for all $y\in Q$ and $w_\ep(x,\,\cdot\,)$ is $Q$-periodic for almost every $x\in\Omega$.
	The position
	\begin{equation}\label{eq:def-uep-tilde}
		\tilde u_\eps (x) \coloneqq w_\eps \left( x , \frac{x}{\eps}\right)\quad\text{for every }x\in \Omega,
	\end{equation}
	defines a measurable function which vanishes on $\Omega_{1,\ep}$,
	because $w=0$ on $\Omega\times D_1$.
	
	We firstly check that $\set{\tilde u_\eps}$ converges strongly two-scale in $L^p$ to $w'(x,y) \coloneqq \chi_{\Omega'}(x)w(x,y)$.
	To prove the claim,
	we start by showing that $\set{\tilde u_\ep}$ is bounded in $L^p(\Omega;\R^N)$.
	The properties of $\set{w_\ep}$ and a change of variables grant that the following identities hold:
	\begin{align*}
		\int_{\Omega} \va{\tilde u_\eps(x)}^p \de x 
		& = \int_{\Omega_{0,\eps}} \left| w_\ep\left(x,\frac{x}{\ep}\right) \right|^p\,\de{x}
		=\sum_{z\in Z_\ep}\int_{\ep(D_0+z)}\left| w_\ep\left(x,\frac{x}{\ep}\right) \right|^p\,\de{x} \\
		& =\sum_{z\in Z_\ep} \ep^d \int_{D_0} \left| w_\ep\big( \ep (z+y),y \big) \right|^p\,\de{y}
		=\sum_{z\in \hat{Z}'_\ep} \ep^d \int_{D_0}\left| \fint_{\ep(Q+z)}w(x,y)\,\de x \right|^p\,\de{y}.
	\end{align*}
	From Jensen's inequality we then deduce
	\begin{align}\label{eq:wep-bd}
		\int_{\Omega} \va{\tilde u_\eps(x)}^p \de x \leq
		\sum_{z\in \hat{Z}'_\ep} \ep^d \int_{D_0}\fint_{\ep(Q+z)}|w(x,y)|^p\,\de{x}\,\de{y}
		=\int_{D_0}\int_{\hat{\Omega}'_\ep}|w(x,y)|^p\,\de{x}\,\de{y}.
	\end{align}

	As for the convergence of $\set{\tilde u_\eps}$,
	in view of Lemma \ref{lemma:2-scale},
	\eqref{eq:wep-bd} yields the existence of a map $\tilde{w}\in L^p(\Omega;L^p_{\rm per}(\R^d;\R^N))$ such that 
	$\tilde u_\ep \weakts \tilde{w}$ weakly two-scale in $L^p$.
	To identify the limit $\tilde{w}$, we consider $\phi\in C(\bar{\Omega};C_\per(\Rd;\RN))$ and,
	by similar computations to the ones above, we find
	\begin{align*}
		\int_{\Omega} \tilde u_\eps(x)\cdot \phi\left(x,\frac{x}{\eps}\right) \de x 
		&  = \int_{\Omega_{0,\eps}} w_\eps\left(x,\frac{x}{\ep}\right)\cdot \phi\left(x,\frac{x}{\eps}\right) \de x \\
		& = \sum_{z\in Z_\ep}\int_{\ep(D_0+z)} w_\ep\left(x,\frac{x}{\ep}\right)\cdot \phi\left(x,\frac{x}{\eps}\right) \de x \\
		& = \ep^d\sum_{z\in Z_\ep}\int_{D_0} w_\ep\big(\ep (z+y),y\big)\cdot \phi\big(\ep (z+y),y\big) \de{y} \\
		& = \sum_{z\in \hat{Z}'_\ep} \int_{D_0}\int_{\ep(Q+z)} w(x,y)\cdot \phi\big(\ep (z+y),y\big) \de{x}\,\de{y} \\
		& = \int_{\hat \Omega'_\eps} \int_{D_0} w(x,y) \cdot\tilde \phi_\eps(x,y)\,\de y\,\de x,
	\end{align*}
	where $\tilde \phi_\eps (x,y)\coloneqq \phi (\eps(y+z),y)$
	if $x\in \eps(Q+z)$ with $z\in\hat Z'_\eps$.
	By the dominated convergence theorem, we infer
	\begin{align*}
		\lim_{\ep\to 0} \int_{\Omega} \tilde u_\eps(x)\cdot \phi\left(x,\frac{x}{\eps}\right) \de x
		= \int_{\Omega'}\int_{D_0} w(x,y) \cdot \phi(x,y)\, \de y\, \de x,
	\end{align*}
	that is, $\tilde u_\ep\weakts w'$ weakly two-scale in $L^p$
	(recall that $w(x,y)=0$ if $y\in D_1$).
	In turn, the weak two-scale convergence implies
	$$\norm{w}_{L^p(\Omega';L^p_{\rm per}(\R^d;\R^N))}\leq \liminf_{\ep\to 0} \norm{\tilde u_\ep}_{L^p(\Omega;\R^N)},$$
	which, combined with \eqref{eq:wep-bd}, ensures that
	$$\lim_{\ep\to 0}\norm{\tilde u_\ep}_{L^p(\Omega;\R^N)}=\norm{w}_{L^p(\Omega';L^p_{\rm per}(\R^d;\R^N))}.$$
	In view of the definition of strong two-scale convergence, cf. Definition \ref{stm:twoscaleconv},
	we conclude that 
	$\tilde u_\eps \strongts w'$ strongly two-scale in $L^p$.
	
	Finally, we show that $\set{\tilde u_\eps}$ is $\A$-free on $\Omega$.
	To this aim, we fix $\psi \in W^{1,p'}_0(\Omega;\RM)$ and,
	by repeating the steps already seen above, we obtain
	\begin{align*}
	\int_{\Omega} \tilde u_\eps(x)\cdot \A^\ast \psi(x) \de x 
	=\sum_{z\in \hat{Z}'_\ep} \int_{\ep(Q+z)} \int_{D_0} w(x,y)\cdot \A^\ast \psi\big(\ep (z+y)\big) \de{y}\,\de{x}.
	\end{align*}
	If $\eta \in C^\infty_c(Q;[0,1])$ is a cut-off function which equals $1$ on $D_0$,
	thanks to the fact that $w(x,\,\cdot\,)$ vanishes on $D_1$ and is $\A$-free on $\Td$ for almost every $x\in \hat{\Omega}'_\ep$, 
	we conclude 
	\begin{align*}
	\int_{\Omega} \tilde u_\eps(x)\cdot \A^\ast \psi(x) \de x 
	& =\sum_{z\in \hat{Z}'_\ep} \int_{\ep(Q+z)} \int_{D_0} w(x,y)\cdot \A^\ast \psi\big(\ep (z+y)\big) \de{y}\,\de{x} \\
	& =\sum_{z\in \hat{Z}'_\ep} \int_{\ep(Q+z)} \int_{D_0} w(x,y)\cdot \A^\ast \Big[\eta(y) \psi\big(\ep (z+y)\big) \Big] \de{y}\,\de{x} \\
	& = \sum_{z\in \hat{Z}'_\ep} \int_{\ep(Q+z)} \int_{Q} w(x,y)\cdot \A^\ast \Big[\eta(y) \psi\big(\ep (z+y)\big) \Big] \de{y}\,\de{x} \\
	& = 0.
	\end{align*}
	 
	\noindent {\sc Step 2: limsup inequality when $w$ is regular}.
	We have so far shown that,
	if $w\in U_0(\Omega') \cap L^p_\per(\Rd;C(\bar{\Omega};\RN))$ and $\set{\tilde{u}_\eps}$ is the sequence of Step 1, then
	\[
		\eps u_\eps=\tilde{u}_\ep  \strongts w'
		\quad\text{strongly two-scale in } L^p,
	\]
	where $w'(x,y) \coloneqq \chi_{\Omega'}(x)w(x,y)$.
	Moreover, $u_\eps=$ on $\Omega_{0,\eps}$ and $\A u_\eps = 0$ on $\Omega$ for all $\eps>0$.
	Here, we prove the inequality
	\begin{equation}\label{eq:limsup-w}
		\limsup_{\ep\to 0} \F_{0,\ep}(u_\eps)
		=\limsup_{\ep\to 0} \int_{\Omega_{0,\ep}} f_{0,\eps} \big( \tilde u_\eps(x)  \big) \de x
			\leq \int_{\Omega'} \int_{D_0} f_0 \big( w(x,y) \big)\,\de y \de x.
	\end{equation}
	
	We preliminarily notice that, by construction,
	$\tilde u_\eps$ vanishes outside the set $\hat{\Omega}'_{\ep}$.
	Therefore, 
	we have the identities 
	\begin{align*}
	\F_{0,\ep}(u_\eps) 
	& = \int_{\hat \Omega'_{\ep}} f_{0,\eps} \big( \tilde u_\eps(x) \big) \de x\\
	& = \eps^d \sum_{z \in \hat{Z}'_\eps} \int_Q
	f_{0,\ep}\Big ( \tilde u_\eps \big( \eps(y+z) \big) \Big)\de{y}.
	\end{align*}
	Since $z \in \Zd$, we have that
	$\tilde u_\eps ( \eps(y+z) ) = \mathsf{S}_\eps \tilde u_\eps (\eps z, y)$,
	$\mathsf{S}_\eps$ being the unfolding operator.
	Therefore, by exploiting the properties of $\mathsf{S}_\ep$,
	the energy of $\tilde u_\eps$ is rewritten as follows:
	\begin{align*}
	\F_{0,\ep}(u_\eps) 
	& = \eps^d \sum_{z \in \hat{Z}'_\eps} \int_Q
	f_{0,\ep}\big ( \mathsf{S}_\eps \tilde u_\eps (\eps z, y) \big) \de{y} \\
	& = \sum_{z\in\hat{Z}'_\eps} \int_{\eps (Q+z)}
	\int_{D_0} {f}_{0,\eps} \left( \mathsf{S}_\eps \tilde u_\eps \left( \eps  \left\lfloor \frac{x}{\eps}\right\rfloor, y \right) \right)  \de y\,\de x \\
	& = \int_{\hat{\Omega}'_\eps}
	\int_{D_0} {f}_{0,\eps} \big( \mathsf{S}_\eps \tilde u_\eps ( x, y ) \big)  \de y\,\de x \\
	& = \int_{\Omega'}
	\int_{D_0} {f}_{0,\eps} \big( \mathsf{S}_\eps \tilde u_\eps ( x, y ) \big)  \de y\,\de x,
	\end{align*}
	where the last identity is due to the fact that $\mathsf{S}_\eps \tilde u_\eps$ vanishes for $x\notin \hat{\Omega}'_\eps$.
	
	In the light of the previous equalities, we have
	\begin{align*}
	\F_{0,\ep}(u_\eps)
	&= \int_{\Omega'} \int_{D_0} {f}_{0,\eps} \big( \mathsf{S}_\eps \tilde u_\eps ( x, y ) \big)  \de y\,\de x
			- \int_{\Omega'} \int_{D_0} {f}_{0,\eps} \big( w(x,y) \big)  \de y\,\de x \\
	&+ \int_{\Omega'} \int_{D_0} {f}_{0,\eps} \big( w(x,y) \big)  \de y\,\de x
				- \int_{\Omega'} \int_{D_0} f_0 \big( w(x,y) \big)\de y \de x \\
	& + \int_{\Omega'} \int_{D_0} f_0 \big( w(x,y) \big)\de y \de x.
	\end{align*}
	Keeping in mind Lemma \ref{lemma:2-scale},
	we observe that $\mathsf{S}_\eps \tilde u_\eps \to w$ strongly in $L^p(\Omega'\times D_0;\R^N)$, 
	and consequently, owing to {\bf H4} and {\bf H5},
	the limits as $\eps \to 0$ of the first and of the second term on the right-hand side equal $0$.
	We thereby conclude that \eqref{eq:limsup-w} holds.

\noindent {\sc Step 3: recovering the general statement.} 
Let $w\in U_0(\Omega')$. We first extend it to the whole space setting $w=0$ for $x\in \Rd \setminus \Omega$,
and then we mollify it with respect to the variable $x$.
This procedure ensures that $\tilde w$, the regularization of $w$,
belongs to $U_0(\Omega') \cap L^p_\per(\Rd;C(\bar{\Omega};\RN))$
and it is close to $w$ in $L^p(\Omega;L^p_\per(\Rd;\RN))$.
Now, Steps 1--2 yield a recovery sequence for $\tilde w$,
and, by means of a diagonal argument and of the continuity of the right-hand side of \eqref{eq:limsup-w} in $L^p$,
a recovery sequence for $w$ satisfying all the desired requirements can be constructed.
\end{proof}	

{\RRR
\begin{remark}[On recovery sequences]
Observe that
the construction we propose does not yield, in general, a recovery sequence
with respect to the weak $L^p$-convergence,
that is, a sequence $\set{u_\eps}\subset L^p(\Omega;\RN)$ fulfilling Corollary \ref{cor:limsup-soft} with (2) replaced by

{\it (2')} $u_\eps \weak u$ in $L^p(\Omega;\RN)$.

\noindent Indeed, to this aim, letting $w'$ be as in Proposition \ref{prop:limsup-soft},
we would need to exhibit a sequence $\set{\tilde u_\eps}$ such that
\[
    \frac{1}{\eps} \tilde u_\eps\strongts w'
	\quad\text{strongly two-scale in } L^p,
\]
and we would then again set $u_\eps \coloneqq \tilde{u}_\eps / \eps $.
\end{remark}
}

On the whole, the results above form a proof of the $\Gamma$-convergence of the energies $\set{\F_{0.\ep}}$.

\begin{proof}[Proof of Proposition \ref{stm:Glim-soft}]
The fact that $\F_0$ is a lower bound for the asymptotic behavior of the sequence $\set{\F_{0,\ep}}$ is established in Proposition \ref{prop:liminf-gen}. The optimality of such lower bound is proven in Corollary \ref{cor:limsup-soft}.
\end{proof}


\section{Asymptotics for the stiff component}
\label{sec:stiff}
In this section, we characterize the limiting behavior of the family $\set{\F_{1,\ep}}$ in \eqref{eq:F1eps},
which accounts for the energy of the `stiff' component of the system,
and we prove that $\set{\F_{1,\ep}}$ $\Gamma$-converges to the functional $\F_1$ in \eqref{eq:F1}
We rely on the contribution by {\sc I. Fonseca \& S. Kr\"omer}
about homogenization of multiple integrals under differential constraints:

\begin{theorem}[Theorem 1.1 in \cite{fonseca.kromer}]
\label{stm:stiff-hom}
Let $g\colon \Rd \times \RN \to [0,+\infty)$ be a Carath\'eodory function
which is $Q$-periodic with respect to its first argument.
Assume that
there exist $c\geq 0$, $\Lambda > 0$, and $p>1$ satisfying
\[
    -c \le g(y,\xi) \leq \Lambda(1+\va{\xi}^p)
    \quad\text{for a.e. $y\in \Rd$ and all $\xi\in\RN$}.
\]
Then, for all $u\in L^p(\Omega;\RN)$ such that $\A u = 0$,
the functional
\[
    \G_\eps ( u ) \coloneqq \int_{\Omega} g \left( \frac{x}{\eps}, u(x) \right) \de x
\]
$\Gamma$-converges with respect to the weak $L^p$-convergence to
\[
    \G_\hom (u) \coloneqq \int_{\Omega} g_\hom \big( u(x) \big) \de x,  
\]
where, for all $\xi \in \RN$,
\begin{multline}\label{eq:hom-density0}
    g_\hom (\xi) \coloneqq \liminf_{k\to +\infty}
        \inf \Bigg\{
            \int_Q g \big( ky, \xi + v(y) \big)\,\de y
            : v\in L^p_\per(\Rd;\RN) \\
            \text{ with } \int_Q v(z) \de z = 0 \text{ and } \A v=0 \text{ in } W^{-1,p}(\Td;\RM)
            \Bigg\}.
\end{multline}
\end{theorem}

As a corollary of the theorem above,
we derive the asymptotics of $\set{\F_{1,\ep}}$.

\begin{proof}[Proof of Proposition \ref{prop:stiff-part}]
    For $\eps>0$, we introduce the set
    \[
        \tilde \Omega_{1,\eps} \coloneqq \Omega \cap \bigcup_{z \in \Zd} \eps ( D_1 + z),
    \]
    and we observe that $\tilde \Omega_{1,\eps} \subset \Omega_{1,\eps}$ (recall \eqref{eq:om1}).
    We let
    \begin{equation*}
    \tilde \F_{1,\ep}(u) \coloneqq
    \begin{cases}
    \displaystyle{ \int_{\tilde \Omega_{1,\eps}} \tilde f_1\left(\frac{x}{\eps}, u \right) \de x}   &\text{if } u\in U_1,\\
    +\infty         &\text{otherwise in }L^p(\Omega;\R^N),
    \end{cases}
    \end{equation*}    
    with $\tilde f_1\colon \Rd \times \RN \to [ -a\lambda , +\infty )$ defined as
    \[
        \tilde f_1(y,\xi) \coloneqq
            f_1 ( y, \xi ) \sum_{z \in \Zd} \chi_{D_1 + z} (y),
    \]
    where $a$ and $\lambda$ are the constants in {\bf H3}, and $\chi_{D_1+z}$ is the characteristic function of the set $D_1+z$. 
    When $u$ is $\A$-free we can hence write
    \begin{equation}\label{eq:F1tilde}
        \F_{1,\eps}(u) =
            \tilde \F_{1,\eps} (u)
            + \int_{\Omega_{1,\eps} \setminus \tilde \Omega_{1,\eps}} f_1 \left( \frac{x}{\eps}, u(x) \right) \de x.
    \end{equation}
    We observe that Theorem \ref{stm:stiff-hom} applies to $\tilde{\F}_{1,\eps}$,
    yielding
    \begin{equation*}
        \F_1 ( u ) = \Gamma-\lim_{\eps \to 0} \tilde{\F}_{1,\eps} ( u )
    \end{equation*}
    in the weak $L^p$-topology.
    As for the difference between $\F_{1,\eps}$ and $\tilde \F_{1,\eps}$,
    by {\bf H3} one has
    \[
        -a \lambda \Ld ( \Omega_{1,\eps} \setminus \tilde \Omega_{1,\eps} )
        \leq
            \int_{\Omega_{1,\eps} \setminus \tilde \Omega_{1,\eps}} f_1 \left( \frac{x}{\eps}, u(x) \right) \de x
        \leq \Lambda\left(
            \Ld ( \Omega_{1,\eps} \setminus \tilde \Omega_{1,\eps} )
            + \int_{\Omega_{1,\eps} \setminus \tilde \Omega_{1,\eps}} \va{ u(x) }^p \de x\right).
    \]
    Notice that the term
    $\Ld ( \Omega_{1,\eps} \setminus \tilde \Omega_{1,\eps} )$ is  at most of order $\eps$.
    Thus, the lower bound for $\set{\F_{1,\eps}}$ follows from \eqref{eq:F1tilde} and Theorem \ref{stm:stiff-hom}:
    \begin{equation*}
    	\F_1 ( u ) \leq \liminf_{\eps \to 0} \F_{1,\eps} ( u_\eps )
    \end{equation*}
    for all $u\in L^p(\Omega;\RN)$ such that $\A u = 0$
    and all $\A$-free sequences $\set{u_\eps} \subset L^p(\Omega;\RN)$
    that converge to $u$ weakly.
    For what concerns the upper bound,
    let us fix $u\in L^p(\Omega;\RN)$ in the kernel of $\A$
    and consider a recovery sequence $\set{u_\eps}$
    for $\set{\tilde \F_{1,\eps}}$ given by Theorem \ref{stm:stiff-hom}.
    Thanks to Lemma \ref{stm:Adecomp},
    we get a $p$-equiintegrable family $\set{v_\eps} \subset L^p(\Omega;\RN)$
    such that $\set{ u_\eps - v_\eps}$ converges to $0$
    strongly in $L^q(\Omega;\RN)$ for all $q\in[1,p)$.
    Therefore,
    \[
        \lim_{\eps \to 0} \int_{\Omega_{1,\eps} \setminus \tilde \Omega_{1,\eps}} \va{ v_\eps(x) }^p \de x = 0.
    \]
    Let now $\set{\ep_n}$ be such that
    $$\limsup_{\eps \to 0} \tilde{\F}_{1,\eps} ( v_\eps )=\lim_{n\to +\infty} \tilde{\F}_{1,\eps_n} ( v_{\eps_n} ).$$
    By Lemma \ref{stm:substitution},
    we obtain the desired upper limit inequality:
    \begin{multline*}
    	\limsup_{\eps \to 0} \F_{1,\eps} ( v_\eps ) = \limsup_{\eps \to 0} \tilde{\F}_{1,\eps} ( v_\eps )
    	=\lim_{n\to +\infty} \tilde{\F}_{1,\eps_n} ( v_{\eps_n} ) \\
    	\leq \lim_{n\to +\infty} \tilde{\F}_{1,\eps_n} ( u_{\eps_n} )
    	\leq \limsup_{\eps \to 0} \tilde{\F}_{1,\eps} ( u_\eps )
    	\leq \F_1 ( u ).
    \end{multline*}
\end{proof}


\section{Admissible differential constraints}
\label{sec:adm}

In what follows we deepen our analysis of constant rank operators.
We investigate the existence of Sobolev potentials for $\A$-free maps and of extension operators on perforated domains. In particular, we prove Theorems \ref{stm:raita-bis} and \ref{thm:exist-ext}.

In contrast to the previous sections,
here we take into account also differential operators of order greater than $1$:
we treat the broader class
of linear, $k$-th order, homogeneous differential operators with constant coefficients.
To fix the notation, given $k,N,M\in\N\setminus\set{0}$,
we consider linear maps $A^{(i)}\colon \RN\to\RM$
for all $d$-dimensional multi-indices $i$ with $\va{i}=k$ and
we focus on the operators
whose action on a function $u\colon \RN\to\RM$
is given by 
\begin{equation*}
    \A u \coloneqq \sum_{\va{i}=k} A^{(i)}\partial_i u.
\end{equation*}
Recall that if $i\coloneqq (i_1,\dots,i_d)\in \N^d$ is a multi-index, then
\[
    \partial_i u(x) \coloneqq \frac{ \partial^{\va{i}} u}{\partial_{i_1} x_1 \cdots \partial_{i_d} x_d}(x).
\]

We keep in place the constant rank condition,
which currently reads:
there exists $r\in\N$
such that the rank of $\bA[\omega]$ equals $r$ for all $\omega \in \Rdmz$,
where
\begin{equation*}
    \bA[\omega] \coloneqq \sum_{\va{i}=k} \omega^i A^{(i)}
\end{equation*}
is the symbol of $\A$
and $\omega^i = \Pi_{j=1}^d \omega_j^{i_j}$
if $i=(i_1,\dots,i_d)$ is a multi-index.
Note that when $k=1$ we recover the setting of Subsection \ref{subs:diff}. As before, we say that a map $u$ is $\A$-free in $\Omega$ if $\A u=0$ in $W^{-k,p}(\Omega;\R^M)$.


\subsection{Existence of potentials}
This subsection is devoted to the proof of Theorem \ref{stm:raita-bis}.
Loosely speaking, our goal is to prove that,
given an operator $\A$ of constant rank,
there exists a second differential operator $\B$
such that any $\A$-free map $u\in L^p$
satisfies $u=\B w$ for a suitable Sobolev function $w$.
When such a relation holds, we will say for brevity
that $w$ is a {\it potential} for $u$ and that $\B$ is a {\it potential} for $\A$.
We are able to prove existence of potentials
for those functions that are $\A$-free 
on a certain class of subsets of $\Rd$, which we introduce in the next definition. 

\begin{definition}\label{stm:ext-dom}
    Let $\A$ be a homogeneous operator of order $k$. An open set $\Omega \subset \Rd$ is an {\it $\A$-extension domain}
    if there exist $\mathsf{E}_\A \colon L^p(\Omega;\RN) \to L^p(\Rd;\RN)$
	and $c> 0$ such that the following holds:
	for all $u\in L^p(\Omega;\RN)$
	such that $\A u = 0$ in $W^{-k,p}(\Omega;\R^M)$, we have
	\begin{enumerate}
    \item ${\mathsf E}_{\A} u = u$ a.e. in $\Omega$,
    \item $\norm{ \mathsf E_{\A} u}_{L^p(\Rd;\RN)} \leq c \norm{u}_{L^p(\Omega;\RN)}$, and
	\item $\A (\mathsf E_{\A} u) = 0$ on $\Rd$.
	\end{enumerate}
\end{definition}

Analogous definitions concerning extensions of sequences that are (asimptotically) $\A$-free
are found in \cite[Definition 1.4]{fonseca.kruzik} and \cite[Lemma 2.8]{fonseca.kromer} (i.e. Lemma \ref{stm:Afree-ext} above), while the reader is referred to Section \ref{sec:extension} below
for a more detailed discussion on extension problems. We also refer to \cite{Heida} for a related study in the stochastic setting.
It is important to notice that
the property of being an $\A$-extension domain carries some topological consequences.

\begin{remark}[On the topology of $\A$-extension domains]\label{stm:simplconn}
The requirements of Definition \ref{stm:ext-dom} may act as implicit restrictions on the topology of the set $\Omega$, according to the specific choice of $\A$. 
Let us consider, for instance, the case in which $d=3$, $\A$=curl, and $\Omega=\R^3\setminus \R e_3$, $\set{e_1,e_2,e_3}$ being the canonical basis in $\R^3$.
We assume by contradiction that $\Omega$ is an $\A$-extension domain.
Then, if for every $x=(x_1,x_2,x_3)\in \Omega$ we let
\[u(x_1,x_2,x_3)\coloneqq \left(-\frac{x_2}{x_1^2+x_2^2}, \frac{x_1}{x_1^2+x_2^2}, 0\right),\]
we find that ${\rm curl}\,u=0$ almost everywhere in $\Omega$,
whence ${\rm curl}\, (\mathsf{E}_{\A} u)=0$ in $\R^3$,  if $\mathsf{E}_{\A}$ is the extension operator in Definition \ref{stm:ext-dom} for $\A=\mathrm{curl}$. This fact, in turn, yields a potential $w$ such that $\mathsf{E}_{\A} u=\nabla w$ in $\R^3$ and, from the requirement that $\mathsf{E}_{\A} u=u$ on $\Omega$, we conclude that $u=\nabla w$ on $\Omega$. This leads to a contradiction, for the field $u$ is not conservative.  
\end{remark}

For the proof of Theorem \ref{stm:raita-bis}
we rely heavily on some recent contributions
by {\sc B. Rai\c{t}\u a} and coauthors,
who established existence of potentials for $\A$-free smooth maps and
Korn-type inequalities for constant rank operators. 
Further related bibliographical comments are provided
when we discuss Propositions \ref{stm:raita} and \ref{stm:guerra-raita} below.
Before dealing with the proof, we pinpoint here that
Theorem \ref{stm:raita-bis} constitutes a sufficient condition for Assumption \ref{stm:null-av} to hold:


\begin{corollary}
\label{cor:u0}
Assume that
$\A$ is a first order operator satisfying the assumptions of Theorem \ref{stm:raita-bis}
and that the unit cube $Q$ is an $\A$-extension domain.
Then, for any open set $\Omega'\subset \Omega$, the space $U_0(\Omega')$ in \eqref{eq:U0} is characterized as follows:
\begin{multline}\label{eq:U0bis}
    U_0(\Omega') = \Bigg\{ u \in L^p(\Omega;L^p_\per(\Rd;\RN)) : u=0 \text{ if } y\in D_1, \\ \A_y u = 0 \text{ in } W^{-1,p}(\Td;\R^M) \text{ for a.e. } x\in\Omega', \text{ and } \int_Q u(x,y)\, \de y = 0 \text{ for a.e. }x \in\Omega' \Bigg\}.    
\end{multline}
\end{corollary}
\begin{proof}
For the desired equality to hold,
we just need to prove that
\begin{multline*}
    U_0(\Omega') \subset \Bigg\{ u \in L^p(\Omega;L^p_\per(\Rd;\RN)) : u=0 \text{ if } y\in D_1, \\ \A_y u = 0 \text{ in } W^{-1,p}(\Td;\R^M) \text{ for a.e. } x\in\Omega', \text{ and } \int_Q u(x,y)\, \de y = 0 \text{ for a.e. }x \in\Omega' \Bigg\}.    
\end{multline*}

Let us fix $u\in U_0(\Omega')$.
Theorem \ref{stm:raita-bis} grants that
for a.e. $x\in \Omega'$ there exists $w_x\in W^{\ell,p}_\per(\Rd;\RM)$
such that $u(x,\cdot)=\B_y w_x$ a.e. in $Q$.
To the aim of defining a measurable selection of the maps $\set{w_x}_{x\in \Omega'}$,
we exploit Proposition \ref{stm:castaing-valadier}.

We first observe that
for every $u\in L^p(\Omega;L^p_\per(\Rd;\RN))$ and $x\in \Omega'$
the set
\[
    W(x) \coloneqq \Set{ w \in W^{\ell,p}_\per(\Rd;\RM) : \B_y w(y) = u(x,y) \text{ for a.e. } y\in \R^d}
\]
is either empty or a closed (and hence complete) subspace of $W^{\ell,p}_\per(\Rd;\RM)$.
Besides, being $x \mapsto u(x,\,\cdot\,)\in L^p_\per(\Rd;\RN)$ measurable, then for every open set $O\subset W^{\ell,p}_\per(\Rd;\RM)$,
the set
\[
    \Set{ x\in \Omega' : W(x) \cap O \neq \emptyset}
\]
is measurable as well.
It follows that the multifunction $x\mapsto W(x)$ admits a measurable selection,
which we denote by $w(x,\,\cdot\,)$.

Summing up,
we recovered a measurable function $w$
from $\Omega'$ to $W^{\ell,p}_\per(\Rd;\RM)$
with the property that $u=\B_y w$ for almost every $x\in \Omega'$ and $y\in \R^d$.
We conclude that
\[
    \int_Q u(x,y)\,\de y 
    = \sum_{\va{i}=\ell} B^{(i)} \int_Q \partial_{y^i} w(x,y)\,\de y = 0,
\]
where the latter inequality follows from the periodicity of $w$ in its second variable.
\end{proof}

We devote the remainder of the section
to the proof of Theorem \ref{stm:raita-bis}
and to the pertaining tools.
We first highlight that
{\sc B. Rai\c{t}\u a} \cite{raita} has recently pointed out that
the kernel of the symbol of \emph{any} constant rank operator
coincides with the image of the symbol of a suitable operator $\B$, and that
the latter is actually a potential for $\A$
when tools from Fourier analysis are applicable.
Precisely, we have:

\begin{proposition}[Theorem 1 and Lemma 2 in \cite{raita}]\label{stm:raita}
        Let $\A$ be a linear homogeneous differential operator on $\Rd$
        with constant coefficients.
        Then, $\A$ is of constant rank
        if and only if
        there exists a linear homogeneous differential operator $\B$ on $\Rd$
        with constant coefficients such that
        \begin{equation}\label{eq:kerA=imB}
            \ker \bA[\omega] = \im \bB[\omega]
            \quad\text{for all }\omega\in \Rdmz,
        \end{equation}
        where $\bA$ and $\bB$ are
        the symbols of $\A$ and $\B$, respectively.
        
        Moreover, 
        suppose that
        $\A$ and $\B$ are, respectively, operators from $\RN$ to $\RM$
        and from $\RM$ to $\RN$.
        If \eqref{eq:kerA=imB} holds,
        then, for all $u\in \call{S}(\Rd;\RN)$
        such that $\A u = 0$,
        there exists $w \in \call{S}(\Rd;\RM)$
        such that $u = \B w$.
\end{proposition}

Observe that if the operator $\B$ satisfies \eqref{eq:kerA=imB},
then it is in turn of constant rank;
precisely, $\mathrm{rank}\,\bB[\omega] = N - r$ if $\mathrm{rank}\,\bA[\omega] = r$ for all $\omega\in\Rdmz$.
Proposition \ref{stm:raita} builds a unified framework 
that encompasses some well-studied special cases,
such as the one of the curl and of the divergence.
Results in the same spirit had been previously obtained by {\sc J. Van Schaftingen} \cite[Proposition 4.2]{vanschaftingen} for elliptic operators,
which correspond to the subclass of operators whose symbol is injective.

A second relevant finding related to constant rank operators was obtained
by {\sc A. Guerra \& B. Rai\c{t}\u a} \cite{guerra.raita},
who showed that this class is exactly the one
in which a sort of Korn inequality holds.
To state their result,
we need to introduce the \emph{projection} $\sfPi_\A$
associated with the operator $\A$.
For $u\in \mathscr{S}(\Rd;\RN)$,
it is defined as
\begin{equation}\label{eq:Pi}
	\sfPi_\A u(x) \coloneqq \big(
	\mathcal{F}^{-1} ( \mathbb{P}_\A \mathcal{F}u)
	\big)(x),
\end{equation}
where $\mathcal{F}$ and  $\mathcal{F}^{-1}$ are the Fourier transform and its inverse (see \eqref{eq:fourier}), and
the map
\begin{equation*}
	\mathbb{P}_\A\colon \Rdmz \to \mathrm{Lin}(\RN;\RN)
\end{equation*} 
associates to each $\omega\in\Rdmz$ the orthogonal projection operator onto $\ker \bA[\omega]$.
Here and in the rest of the paper,
$\mathrm{Lin}(V;W)$ is the set of linear operators
from the vector space $V$ to the vector space $W$.
We have:
\begin{proposition}[Korn-type inequality for constant rank operators \cite{guerra.raita}]
    \label{stm:guerra-raita}
    Let $p\in(1,\infty)$
    and let $\A$ be a linear, $k$-th order, homogeneous differential operator
    with constant coefficients.
    Then, $\A$ is of constant rank if and only if
    there exists $c\coloneqq c(d,p)$ such that
    \begin{equation*}
        \norm{\nabla^k (\phi - \sfPi_\A\phi)}_{L^p(\Rd;\R^{N\times d^k})}
        \leq c \norm{\A \phi}_{L^p(\Rd;\RM)}
        \quad\text{for all } \phi\in\Cc^\infty(\Rd;\RN).
    \end{equation*}
\end{proposition}

For future use, we remark that
the conclusion remains valid when $\phi\in \mathscr{S}(\Rd;\RN)$,
as the Fourier analysis
on which the proof is based
is still viable in this case,
or even when $\phi\in W^{k,p}(\Rd;\RN)$, by approximation.
The sufficiency of the constant rank condition had been already observed in the literature,
see for instance the bibliography in \cite{guerra.raita} or the paper by {\sc D. Gustafson} \cite{gustafson},
where a proof for 
first order operators is given.

To the purpose of constructing Sobolev potentials for $\A$-free fields,
we first need to extend the projection operator $\sfPi_\A$ in \eqref{eq:Pi} to nonsmooth maps.
We invoke the following result by {\sc T. Kato} \cite{kato},
which we present in a form due to {\sc F. Murat}:

\begin{lemma}[Lemma 3.7 in \cite{murat}]
    \label{stm:murat}
	Let $\A$ be a linear, $k$-th order, homogeneous differential operator
    with constant coefficients and constant rank.
    Then, the operator  $\sfPi_\A$ in \eqref{eq:Pi} can be extended
    to a bounded linear operator from $L^p(\Rd;\RN)$ to $L^p(\Rd;\RN)$.
\end{lemma}

In the proof of the lemma
the constant rank assumption is fundamental.
Indeed, it ensures that 
the map $\omega \mapsto \mathbb{P}_\A[\omega]$ introduced above
is analytic on $\Rdmz$ (see \cite{murat});
this, in combination with the $0$-homogeneity of $\mathbb{P}_\A$,
allows the use of {\sc S.G.~Mikhlin}'s multipliers theorem
(cf. Corollary \ref{stm:mikhlin}).

\begin{remark}[Self-adjointness of the projection operator]\label{stm:selfadj}
	As a consequence of Parseval's formula and
	of the self-adjointness of $\mathbb{P}_\A$,
	by density there holds
	\begin{equation}\label{eq:selfadj}
		\int_{\Rd} \sfPi_\A u \cdot v\,\de{x} = \int_{\Rd} u \cdot \sfPi_\A v\,\de{x}
		\quad\text{for all } u,v\in L^p(\Rd;\RN).
	\end{equation}
\end{remark}

We next clarify why $\sfPi_\A$ is entitled to be named projection:
Lemma \ref{stm:proj} below shows that
for an $L^p$-function $u$ one has that
$\A(\sfPi_\A u) = 0$, and that $\sfPi_\A u = u$ if $\A u = 0$.
We premise the instrumental notion of \emph{Moore-Penrose generalized inverse} $\Lambda^\dagger$
of a linear map $\Lambda$.

\begin{lemma}[Properties of the generalized inverse \cite{guerra.raita, campbell.meyer}]
    \label{stm:geninv}
    Given the map $\Lambda\in \mathrm{Lin}(\RN;\RM)$,
    we let $\Lambda^\dagger\in \mathrm{Lin}(\RM;\RN)$
    be defined as
    \begin{equation*}
        \Lambda^\dagger \coloneqq
            \left( \Lambda\lfloor_{(\ker\Lambda)^\perp} \right)^{-1}
            \circ \mathbb{P}_{\im \Lambda},
    \end{equation*}
    where $(\ker\Lambda)^\perp=\im \Lambda^\ast$ is the orthogonal complement of the kernel of $\Lambda$ and
    $\mathbb{P}_{\im \Lambda}$ is the orthogonal projection on the image of $\Lambda$.
    Then, the following hold:
    \begin{enumerate}
        \item $\Lambda^\dagger$ is the unique element
        in $\mathrm{Lin}(\RM;\RN)$ such that
        $\Lambda^\dagger \circ \Lambda = \mathbb{P}_{\im\Lambda^\ast}$ and
        $\Lambda \circ \Lambda^\dagger = \mathbb{P}_{\im \Lambda}$.
        \item Let $\bA\colon O \to \mathrm{Lin}(\RN;\RM)$ be a smooth map
        on the open set $O\subset \Rd$.
        If $\mathrm{rank}\, \bA[\omega]$ is constant for all $\omega\in O$,
        then the map $O\ni\omega \mapsto (\bA[\omega])^\dagger$ is locally bounded and smooth.
        \item If $\A$ is a linear, $k$-th order, homogeneous differential operator
        with constant coefficients,
        then the map $\omega \mapsto (\bA[\omega])^\dagger$ is $(-k)$-homogeneous.
    \end{enumerate}
\end{lemma}

We can now characterize the properties of the projection $\sfPi_\A$.
Note that they are comparable
to the ones presented by {\sc I. Fonseca \& S. M\"uller} \cite[Lemma 2.14]{fonseca.muller}
for the periodic setting. We also refer to \cite[Subsection 2.8]{kreisbeck.rindler} for an alternative projection operator on the unit torus for which no null-average conditions are imposed.

\begin{lemma}\label{stm:proj}
	Let $\A$ be a linear, $k$-th order, homogeneous differential operator
    with constant coefficients and constant rank.
	For every $u\in L^p(\Rd;\RN)$, there holds:
	\begin{enumerate}
		\item
		For all $\psi\in C_c^1(\Rd;\RM)$, we have
		$\sfPi_\A(\A^\ast\psi) = 0$ and,
		as a consequence,
		$\A(\sfPi_\A u) = 0$.		
		\item\label{stm:Piu=u}  
		If $h=0$ and $u\in L^p(\Rd;\RN)$, or
		if $h=1,\dots,k-1$ and $u\in W^{h,p}(\Rd;\RN)$,
		there exists a constant $c\coloneqq c(p,h)>0$ such that
		\begin{equation}\label{eq:W-kp}
		\norm{\nabla^h(u-\sfPi_\A u)}_{L^p(\Rd;\R^{N\times d^h})} \leq c\norm{\A u}_{W^{-(k-h),p}(\Rd;\RN)}.
		\end{equation}
		In particular, when $\A u = 0$ in $\Rd$,
		then $\sfPi_\A u = u$ a.e.
	\end{enumerate}
\end{lemma}
\begin{proof}
	When $u\in \mathscr{S}(\Rd;\RN)$,
	the facts
	that $\A(\sfPi_\A u) = 0$ and $\sfPi_\A u = u$ if $\A u = 0$
	are simple consequences of  \eqref{eq:Pi}.
	To extend these properties to the case of a nonsmooth $u$,
	we regard
	$\psi\in C_c^1(\Omega;\RM)$ as a Schwartz function on the whole space and
	we find
	\[
		\sfPi_\A (\A^\ast\psi) = \mathcal{F}^{-1} \big( \mathbb{P}_\A (\bA^\ast \mathcal{F} \psi ) \big) = 0,
	\]
	because the image of $\bA^\ast[\omega]$ is orthogonal to $\ker \bA[\omega]$ for all $\omega \in \Rd$.
	It follows by \eqref{eq:selfadj} that	
	\[
		\int_{\Rd} \sfPi_\A u \cdot \A^\ast \psi\,\de{x} = \int_{\Rd} u \cdot \sfPi_\A \A^\ast \psi\,\de{x} = 0
	\]
	and statement (1) holds.
	
	We now turn to point (2). We argue for $h=0$ as the case $h>0$ is analogous.
	We consider at first $u\in \mathscr{S}(\Rd;\RN)$ and
	we compute the Fourier transform of $u-\sfPi_\A u$.
    Setting $\hat u = \mathcal{F} u$, we find
    \[
	    \hat u[\omega] - \mathbb{P}_\A[\omega]\hat u[\omega]
	    = \mathbb{P}_{\im \bA[\omega]^\ast} \hat u[\omega]
	    = \bA[\omega]^\dagger \big(\bA[\omega] \hat u[\omega]\big)
	    \quad\text{for all } \omega\in\Rd,
	\]
    where in the latter inequality we exploited statement (1) in Lemma \ref{stm:geninv}.
	
    Since $\mathbb{A}$ is $k$-homogeneous and $\mathbb{A}^\dagger$ is $(-k)$-homogeneous, 
    we infer by Corollary \ref{stm:mikhlin} that $\mathbb{A}^\dagger \mathbb{A}$ is an $L^p$-Fourier multiplier.
    Therefore, the operator $\mathsf T$
    defined for $u \in \mathscr{S}(\R^d;\R^N)$ as $\mathsf{T}u \coloneqq \mathcal{F}^{-1}(\hat{u}-\mathbb{P}_{\A}\hat{u})$ can be extended to a bounded linear operator
    from $L^{p}(\Rd;\RN)$ to $L^p(\Rd;\RN)$,
    which we still denote by $\mathsf{T}$.
    
    For every $u\in \mathscr{S}(\R^d;\R^N)$ we write
    \begin{align*}
        \mathcal{F} (\mathsf{T}u)[\omega] &=
            \mathbb{A}^\dagger [\omega] \big(\mathcal{F}(\A u)[\omega] \big) \\
        & = (1+\va{\omega}^2)^{k/2}\mathbb{A}^\dagger [\omega] \big( (1+\va{\omega}^2)^{-k/2} \mathcal{F}(\A u)[\omega] \big),
    \end{align*}
    whence, by the $(-k)$-homogeneity of $\A^\dagger$,
    {Mikhlin}'s multipliers theorem yields
    \begin{equation*}
    \|\mathsf T u\|_{L^p(\Omega;\R^N)}
        \leq c \|(I-\Delta)^{-k/2}(\A u)\|_{L^p(\Omega;\R^N)}.
    \end{equation*}
    From the definition of $\mathsf T$ and the characterization of $W^{-k,p}$ recalled in Subsection \ref{sec:fourier},
    \eqref{eq:W-kp} follows for $u\in \mathscr{S}(\R^d;\R^N)$.
    The general assertion is then obtained by density. 
\end{proof}

Eventually, we are able to establish
the existence of potentials in a nonsmooth setting and prove the first main result of this section.
We will use the following variant of Poincar\'e-Wirtinger's inequality,
which can be derived as a corollary of Rellich-Kondrachov's Theorem.

\begin{lemma}
    Let $\Omega\subset\Rd$ be a bounded, connected,  open set with Lipschitz boundary.
    There exists a constant $c>0$ such that for every $u\in W^{\ell,p}(\Omega;\RN)$
    \begin{equation}\label{eq:higherord-poincare}
      \norm{u - \sfPi_{\nabla^\ell} u}_{L^p(\Omega;\RN)} \leq c \norm{\nabla^\ell u}_{L^p(\Omega;\R^{N\times d^\ell})},
    \end{equation}
    where $\sfPi_{\nabla^\ell}$ is the projection on the kernel of the operator $\nabla^\ell$.
\end{lemma}

\begin{proof}[Proof of Theorem \ref{stm:raita-bis}]
		Let us fix $u\in L^p(\Omega;\RN)$ satisfying $\A u = 0$ in $W^{-k,p}(\Omega;\RM)$.
		Since we postulate that
		$\Omega$ is an $\A$-extension domain,
		there exists an operator $\mathsf{E}_\A \colon L^p(\Omega;\RN) \to L^p(\Rd;\RN)$
		as in Definition \ref{stm:ext-dom}.
		In particular,
		if we let $\tilde u \coloneqq \mathsf{E}_\A u$,
		then $\tilde u$ is $\A$-free on $\Rd$.
		
		As a first step,
		we approximate $\tilde u$ by maps in the image of $\B$.
		By the definition of the projection $\sfPi_\A$,
		there exist a sequence $\set{u_k}\subset \call{S}(\Rd;\RN)$
		such that
		\[
			u_k \to \tilde u
			\quad\text{and}\quad
			\tilde u_k \coloneqq \sfPi_\A u_k \to \sfPi_\A \tilde u
			\quad\text{strongly in } L^p(\Rd;\RN).
		\]
		By construction,
		the functions $\tilde u_k$ belong to $\call{S}(\Rd;\RN)$ and,
		in view of Proposition \ref{stm:raita},
		we recover a sequence $\set{w_k}\subset \call{S}(\Rd;\RN)$
		such that $\tilde u_k = \B w_k$.
		Therefore,
		recalling that $\sfPi_\A \tilde u =  \tilde u$ because $\tilde u$ is $\A$-free,
		we deduce
		\begin{equation}\label{eq:Bwk-to-u}
			\B w_k \to \tilde u
			\quad \text{strongly in } L^p(\Rd;\RN).
		\end{equation}

		Next, we proceed by applying Proposition \ref{stm:guerra-raita}:
		since $\B$ has constant rank,
		\begin{equation*}
		\left\Vert \nabla^\ell \big( w_k-\sfPi_\B w_k \big) \right\Vert_{L^p(\Rd;\RM)}
		\leq c \norm{\tilde u_k}_{L^p(\Rd;\RN)},
		\end{equation*}
		where $\ell$ is the order of $\B$.
		Note that the right-hand side is uniformly bounded in $k$
		because $\set{\tilde u_k}$ is convergent in $L^p(\Omega;\R^N)$ by \eqref{eq:Bwk-to-u}.
		As a consequence, inequality \eqref{eq:higherord-poincare} yields
		a function $w\in W^{\ell,p}(\Omega;\RM)$ such that (up to subsequences)
		\[
			\tilde{w}_k \coloneqq w_k-\sfPi_\B w_k - \sfPi_{\nabla^\ell}\big( w_k-\sfPi_\B w_k \big) \to w
		\quad\text{strongly in } L^p(\Omega;\RM).
		\]
		Since the functions $w_k$ are smooth,
		the equality $ \B w_k = \B \tilde w_k $ holds pointwise and
		we deduce from \eqref{eq:Bwk-to-u} that,
		for all $\phi\in\Cc^\infty(\Omega;\RN)$,
		\begin{equation*}
		\begin{split}
		\int_{\Omega} u(y) \cdot \phi(y)\, \de y 
		= \lim_{k\to +\infty} \int_{\Omega} \tilde w_k \cdot \B^\ast \phi(y)\, \de y
		= \int_{\Omega} w(y) \cdot \B^\ast \phi(y)\, \de y,
		\end{split}
		\end{equation*}
		where $\B^\ast$ is the adjoint of $\B$.
		The conclusion is then achieved.
\end{proof}

\begin{remark}\label{stm:est-potential}
    Let $u\in L^p(\Omega;\RN)$ be an $\A$-free map and
    let $w\in W^{\ell,p}(\Omega;\RM)$ be a potential for $u$ in the sense of Theorem \ref{stm:raita-bis}.
    Then, there exist constants $c_0,c_1>0$ such that
    \begin{equation}\label{eq:u-vs-w}
        \frac{1}{c_0}\norm{u}_{L^p(\Omega;\RN)}
        \leq \norm{w}_{W^{\ell,p}(\Omega;\RN)}
        \leq c_1 \norm{u}_{L^p(\Omega;\RN)}.
    \end{equation}
    
    The first inequality easily follows from the identity
    $u=\B w$.
    Conversely,
    keeping in force the notations in the proof of Theorem \ref{stm:raita-bis},
    we have
    \begin{equation*}
		\left\Vert \tilde w_k \right\Vert_{W^{\ell,p}(\Omega;\RM)}
		\leq c_1 \norm{\tilde u_k}_{L^p(\Rd;\RN)},
	\end{equation*}
	for a suitable $c_1> 0$,
	whence, by taking the limit $k\to+\infty$
	and recalling requirement (2) in Definition \ref{stm:ext-dom},
	the second estimate in \eqref{eq:u-vs-w} is achieved.
\end{remark}

\subsection{$\A$-free extensions}\label{sec:extension}

The $\Gamma$-convergence analysis that we developed in Sections \ref{sec:comp+split} -- \ref{sec:stiff}
is grounded on the splitting argument contained in Lemma \ref{stm:split},
which in turn requires Assumption \ref{stm:exist-ext}.
For the reader's convenience,
we recall the latter here:

\setcounter{assumption}{1}
\begin{assumption}[$\A$-free extension]
	There exists an $\eps$-independent constant $c> 0$ and a sequence of operators $\set{\mathsf{E}_\A^\ep}$, with  $\mathsf{E}_\A^\ep \colon L^p(\Omega;\RN) \to L^p(\Omega;\RN)$
	 such that the following holds:
	for all $u\in L^p(\Omega;\RN)$
	such that $\A u = 0$ in $W^{-1,p}(\Omega;\R^N)$
	\begin{enumerate}
    \item ${\mathsf E}_{\A}^\ep u = u$ a.e. in $\Omega_{1,\ep}$,
    \item $\norm{ \mathsf E_{\A}^\ep u}_{L^p(\Omega;\RN)} \leq c \norm{u}_{L^p(\Omega_{1,\eps};\RN)}$,
	\item $\A (\mathsf E_{\A}^\ep u) = 0$ on $\Omega$, and
	\item if $\set{u_\ep}\subset L^p(\Omega;\R^N)$ is $p$-equiintegrable, then $\set{\mathsf{E}_{\A}^\ep u_\ep}\subset L^p(\Omega;\R^N)$ is also $p$-equiintegrable.
	\end{enumerate}
\end{assumption}

In contrast to the classical cases of the curl and of higher-order gradients
(see the examples at the end of this section),
we did not manage to ensure that
for a general operator $\A$ the whole Assumption \ref{stm:exist-ext} holds
(cf. Theorem \ref{thm:exist-ext}).
Note that this assumption is designed to tackle the periodically perforated structure appearing in our problem.
The first step to prove properties (1)--(3) is thus to start
from the simpler scenario
in which the parameter $\eps$ is neglected.
In this respect, we establish the following:

\begin{theorem}[Existence of $\A$-free extensions]
    \label{stm:Aext}
    Let $D,O\subset \Rd$ be open sets such that $D$ is connected, $O$ is bounded, and
	$\partial D\cap \bar O$ is a Lipschitz boundary.
	Let also $\A$ be a linear, $k$-th order, homogeneous differential operator
    with constant coefficients and constant rank, and
    let $\B$ be a linear, $\ell$-th order,
    homogeneous differential operator 
    with constant coefficients such that \eqref{eq:kerA=imB} holds.
    We further assume that
    \begin{itemize}
        \item for all $\A$-free $u\in L^p(D;\RN)$ there exists $w\in W^{\ell,p}(D;\RM)$ satisfying $u=\B w$;
        \item there exist a projection operator on the subspace of $\B$-free maps $\sfPi_\B\colon W^{\ell,p}(D;\RM) \to W^{\ell,p}(D;\RM)$ and a constant $c>0$ such that
        \begin{equation}\label{eq:genKorn}
            \norm{\nabla^\ell(w - \sfPi_\B w)}_{L^p(D;\R^{N\times d^\ell})}
            \leq c \norm{\B w}_{L^p(D;\RM)}
            \quad\text{for all } w\in W^{\ell,p}(D;\RN).    
            \end{equation}
    \end{itemize}
    Then, there exist a map
	\[
	{\mathsf E}_{\A}\colon L^p(D;\RN) \to L^p(O;\RN)
	\]
	and a constant $c\coloneqq c(d,p,\A,D,O)$ such that,
    for all $u\in L^p(D;\RN)$ with $\A u=0$ in $W^{-k,p}(D\cap O;\R^N)$,
	\begin{enumerate}
		\item ${\mathsf E}_{\A} u = u$ a.e. in $D\cap O$,
		\item $\norm{ \mathsf E_{\A} u}_{L^p(O;\RN)} \leq c \norm{u}_{L^p(D;\RN)}$, and
		\item $\A (\mathsf E_{\A} u) = 0$ on $O$.
	\end{enumerate} 
\end{theorem}

Here, the projection $\sfPi_\B$ on the kernel of $\B$ has to be understood as an analogue of the one in \eqref{eq:Pi} and Lemma \ref{stm:murat}.
The main difference is that $\sfPi_\B$ acts on functions defined on a domain, and not on the whole space.

	\begin{remark}(On Korn-type inequalities)
		\RRR 
		After the first version of this manuscript was completed,
		{\sc A. Arroyo-Rabasa} proved in the preprint \cite{arroyo2} that
		inequalities of the form \eqref{eq:genKorn-O}
		(that is, \eqref{eq:genKorn} for every open bounded $D$)
		hold whenever $\B$ meets a suitable maximal rank requirement,
		which entails, in particular, that $\B$ has constant rank.
		Therefore, if $\A$ admits a potential $\B$ in such class,
		the assumptions in Theorem \ref{stm:Aext} are satisfied.
		An example of maximal rank operator is the divergence, see Example \ref{ex:div} below.
	\end{remark}

We will comment on the relationships 
between the previous theorem and the theory developed above
at the end of this section, see Remark \ref{stm:conclusions}.
For the moment being,
let us just highlight that
the hypothesis concerning existence of potentials for $\A$-free fields on $D$
enables us to recast the problem
in terms of extension of Sobolev maps.
For the latter, an adaptation of well-established arguments yields the following:

\begin{lemma}[cf. Lemma 2.6 in \cite{ACDP}]
	\label{stm:sobolevext}
	Let $D,O\subset \Rd$ be open sets.
	If $D$ is connected, $O$ is bounded, and
	$\partial D\cap \bar O$ is a Lipschitz boundary,
	there exist a bounded linear map
	\[
		{\mathsf E}\colon W^{\ell,p}(D;\RN) \to W^{\ell,p}(O;\RN)
	\]
	and a constant $c\coloneqq c(d,p,D,O)$ such that
	\begin{enumerate}
		\item ${\mathsf E} u = u$ a.e. in $D\cap O$,
		\item $\norm{ \mathsf E u}_{L^p(O;\R^N)} \leq c \norm{u}_{L^p(D;\R^N)}$, and
		\item $\norm{ \nabla^\ell(\mathsf E u)}_{L^p(O;\R^{N\times d^\ell})} \leq c \norm{\nabla^\ell u}_{L^p(D;\R^{N\times d^\ell})}$.
	\end{enumerate}
\end{lemma}
\begin{proof}
    The proof follows the same lines of \cite[Lemma 2.6]{ACDP}.
    Note that, in order to recover item (3) when $\ell>1$, Poincar\'e's inequality has to be replaced with \eqref{eq:higherord-poincare}.
    \end{proof}
We are now in a position to prove the first main result of this section.
\begin{proof}[Proof of Theorem \ref{stm:Aext}]
	Let us fix $u\in L^p(D;\RN)$ such that $\A u = 0$.
	The current assumptions grant that
	there is a Sobolev potential $w\in W^{\ell,p}(D;\RM)$ for $u$,
	i.e. $u=\B w$,
	and that it satisfies
	\begin{equation}\label{eq:gradw}
        \norm{\nabla^\ell(w - \sfPi_\B w)}_{L^p(D;\R^{M\times d^\ell})}
        \leq c \norm{u}_{L^p(D;\RN)}.
    \end{equation}
	If $\mathsf E$ is the extension  operator in Lemma \ref{stm:sobolevext}, we set
	\begin{equation*}
		{\mathsf E_\A} u \coloneqq \B \big({\mathsf E} (w - \sfPi_\B w) \big).
	\end{equation*}
	By construction, then,
	${\mathsf E}_{\A} u = u$ almost everywhere in $D\cap O$. Additionally, $\A (\mathsf E_{\A} u) = 0$ on $O$
	by the definition of $\A$-free maps and \eqref{eq:kerA=imB}.
	
	To conclude, we need to show that
	$\norm{{\mathsf E_\A} u}_{L^p(O;\RN)} \leq c \norm{u}_{L^p(D;\RN)}$.
	By the definition of ${\mathsf E_\A} u$, we have
	\begin{equation*}
		\norm{{\mathsf E_\A} u} _{L^p(O;\RN)}
		= \norm{\B \big({\mathsf E} (w - \sfPi_\B w) \big)} _{L^p(O;\RN)}
		\leq c \norm{\nabla^\ell \big({\mathsf E} (w - \sfPi_\B w)\big)}_{L^p(O;\R^{M\times d^\ell})},
	\end{equation*}
	$c$ being a constant depending on $\B$
	(and hence on $\A$).
	Thanks to Lemma \ref{stm:sobolevext}, 
	we obtain a bound in terms of the potential of $u$:
	\begin{equation*}
		\norm{{\mathsf E_\A} u} _{L^p(O;\RN)}
		\leq c \norm{\nabla^\ell\big(w - \sfPi_\B w\big)}_{L^p(D;\R^{M\times d^\ell})}
	\end{equation*}
	for some $c\coloneqq c(d,p,\A,D,O)$.
	The conclusion is achieved by combining the above inequality with \eqref{eq:gradw}.
\end{proof}

Arguing as in the proof of Theorem \ref{stm:Aext},
we ground the study about extensions from perforated domains on the corresponding result for Sobolev functions.
When the perforations are detached from the boundary,
the following holds:

\begin{proposition}[cf. Lemma 8 in \cite{CC} and references therein]
\label{thm:extension-classic}
\RRR
Let $\Omega_{1,\eps}$ as in \eqref{eq:om1}.
Then, there exist a constant $c\coloneqq c(d,p,D)$
independent of $\ep$ and $\Omega$,
as well as a sequence of operators $\set{{\mathsf E}^\ep}$,
with ${\mathsf E}^\ep\colon W^{\ell,p}(\Omega_{1,\eps};\R^N)\to W^{\ell,p}(\Omega;\R^N)$,
such that
\begin{enumerate}
		\item ${\mathsf E}^\ep u = u$ a.e. in $\Omega_{1,\eps}$,
		\item $\norm{ \mathsf E^\ep u}_{L^p(\Omega;\RN)} \leq c \norm{u}_{L^p(\Omega_{1,\eps};\RN)}$; 
		\item $\norm{ \nabla^\ell(\mathsf E^\ep u)}_{L^p(\Omega;\R^{N\times d^\ell})} \leq c \norm{\nabla^\ell u}_{L^p(\Omega_{1,\eps};\R^{N\times d^\ell})}$, and
		\item if $\set{u_\ep} \subset L^p(\Omega;\RN)$ is bounded and $\set{\nabla^\ell u_\ep}$ is $p$-equiintegrable, then $\set{\nabla^\ell (\mathsf{E}^\ep u_\ep)}$ is $p$-equiintegrable as well.
\end{enumerate}
\end{proposition}

The proposition may be proved
by adapting the strategy in the seminal work by {\sc E. Acerbi \& Al.} \cite{ACDP}.
Their result addresses only the case $\ell=1$, and the analogue of (4) is not mentioned;
we omit nonetheless the proof in the case $\ell>1$,
which is a natural adaptation of \cite[Proof of Theorem 2.1]{ACDP}.
As for point (4), it is a mere consequence of the construction of $\mathsf E^\ep$:
it suffices to note that reflections, dilations, and patching by partitions of unity preserve $p$-equiintegrability.

{\RRR Thanks to the previous result,
we obtain that
Assumption \ref{stm:exist-ext} is fulfilled
by an operator $\A$
as soon as it admits a potential for which a Korn-type inequality holds.}
Note that it is fundamental that we start from fields $u$ which are $\A$-free in the whole set $\Omega$:
in principle, the existence of a potential would be false
if we worked with maps that are $\A$-free just on the perforated set $\Omega_{1,\ep}$, cf. Remark \ref{stm:simplconn}.

\begin{proof}[Proof of Theorem \ref{thm:exist-ext}]
\RRR
Let $u\in L^p(\Omega;\RN)$ be an $\A$-free map and
let $w\in W^{\ell,p}(\Omega;\RM)$ be its potential.
In the same spirit of Theorem \ref{stm:Aext}, we set
    \[
        \mathsf E^\eps_\A u \coloneqq \B \big({\mathsf E}^\eps (w - \sfPi_\B w) \big),
    \]
where $\mathsf E^\eps$ is the extension operator in Proposition \ref{thm:extension-classic}.
As in the proof of Theorem \ref{stm:Aext},
it easy to check that (1) and (3) in Assumption \ref{stm:exist-ext} hold.
We now turn to item (2), that is,
    \begin{equation*}
        \norm{ \mathsf E_{\A}^\ep u}_{L^p(\Omega;\RN)} \leq c \norm{u}_{L^p(\Omega_{1,\eps};\RN)}
        \qquad\text{ for all $\A$-free }u\in L^p(\Omega;\RN).
    \end{equation*}

By the definition of $\Omega_{0,\eps}$ in \eqref{eq:om0},
there exists an open set $\Omega'\subset \Omega$ with Lipschitz boundary such that
$\Omega_{0,\eps}\subset \Omega'$ and that $\delta\coloneqq \dist(\partial \Omega',\partial \Omega) >0$.
In particular, recalling \eqref{eq:def-new-set},
$\Omega'\subset \hat \Omega_\eps$ if $\sqrt{d}\eps < \delta$, and
$\set{\hat \Omega_\eps,\Omega\setminus \bar \Omega'}$ is an open cover of $\Omega$.
We observe that
    \[
        \mathsf{E}^\eps w = w
        \quad\text{a.e. in } \Omega\setminus \bar \Omega' \subset \Omega_{1,\eps},
    \]
whence, by the definition on $\mathsf{E}^\eps_\A$,
    \begin{equation}\label{eq:bound-ext-1}
		\norm{{\mathsf E^\eps_\A} u} _{L^p(\Omega\setminus \Omega';\RN)}
		= \norm{u}_{L^p(\Omega\setminus \Omega';\RN)}
		\leq \norm{u}_{L^p(\Omega_{1,\eps};\RN)}
    \end{equation}
For what concerns the contribution on $\hat \Omega_\eps$, if we let
    \begin{equation*}
        \hat{\Omega}_{1,\ep}\coloneqq \bigcup_{z\in \hat Z_\eps} \eps(D_1+z),
        \quad\text{with}\quad
        \hat Z_\ep \coloneqq \set{ z \in \Z^d : \eps (Q + z ) \subset \Omega},
    \end{equation*}
we have
\begin{align*}
		\norm{{\mathsf E^\eps_\A} u} _{L^p(\hat \Omega_\eps;\RN)}
		\leq c \norm{\nabla^\ell \big({\mathsf E^\eps} (w - \sfPi_\B w)\big)}_{L^p(\hat \Omega_\eps;\R^{M\times d^\ell})}
		\leq c \norm{\nabla^\ell \big( w- \sfPi_\B w\big)}_{L^p(\hat \Omega_{1,\eps};\R^{M\times d^\ell})}.
\end{align*}
The second inequality is a consequence of the construction of $\mathsf{E}^\eps$,
which is obtained by patching together the extension operators from each `stiff' unit $\eps(D_1+z)$ to the cell $\eps(Q+z)$.
Thanks to the definition of $\hat\Omega_{1,\eps}$
we can further bound the last quantity from above
by invoking Korn's inequality
    \[
    \norm{\nabla^\ell \big( w - \sfPi_\B w\big)}_{L^p(D_1;\R^{M\times d^\ell})}
    \leq c \norm{\B w}_{L^p(D_1;\RN)},
    \]
which holds by assumption.
We obtain
    \begin{equation*}
        \norm{{\mathsf E^\eps_\A} u} _{L^p(\hat \Omega_\eps;\RN)}
		\leq c\norm{u}_{L^p(\hat \Omega_{1,\eps};\RN)}
		\leq c\norm{u}_{L^p(\Omega_{1,\eps};\RN)}.
    \end{equation*}
On the whole, the last estimate and \eqref{eq:bound-ext-1} yield
    \begin{equation*}
    \norm{ \mathsf E_{\A}^\ep u}_{L^p(\Omega;\RN)}
    \leq \norm{ \mathsf E_{\A}^\ep u}_{L^p(\hat\Omega_\eps;\RN)} + \norm{ \mathsf E_{\A}^\ep u}_{L^p(\Omega\setminus \Omega';\RN)}
    \leq c \norm{u}_{L^p(\Omega_{1,\eps};\RN)}.    
    \end{equation*}

Eventually, we turn to (4).
For any $\eps>0$, let $w_\eps$ be the potential of $u_\eps$.
We fix $\eta>0$ arbitrarily and 
we let $E\subset \Omega$ be a Lebesgue measurable set.
As a first step, we prove that
there exists $m>0$ such that
$\mathscr L^d(E)<m$ implies
    \begin{equation*}
        \norm{\nabla^\ell(w_\eps - \sfPi_\B w_\eps)}^p_{L^p(E;\R^{M\times d^\ell})} < \eta.
    \end{equation*}

Since $\set{u_\eps}$ is $p$-equiintegrable,
there exists $\tilde m>0$ with the property that
$\norm{u_\eps}_{L^p(F;\RN)} < \eta$ whenever $\mathscr L^d(F)<2\tilde m$.
Thanks to the outer regularity of the Lebesgue measure,
we can select a finite union of open rectangles $U\supset E$
such that $\mathscr L^d(U)<\mathscr L^d(E)+\tilde m$.
Thus, if we set $m\coloneqq \tilde m$ and we assume that $\mathscr L^d(E)<m$,
thanks to \eqref{eq:genKorn-O} we deduce
    \begin{align*}
        \norm{\nabla^\ell(w_\eps - \sfPi_\B w_\eps)}^p_{L^p(E;\R^{M\times d^\ell})}
        \leq
        \norm{\nabla^\ell(w_\eps - \sfPi_\B w_\eps)}^p_{L^p(U;\R^{M\times d^\ell})}
        \leq
        \norm{u_\eps}^p_{L^p(U;\RN)}
        < \eta,
    \end{align*}
where the last inequality is a consequence of the $p$-equiintegrability of $\set{ u_\eps}$.

In conclusion, we infer the $p$-equiintegrability of $\set{\mathsf{E}_\A^\eps u_\eps}$ from the one of $\set{\nabla^\ell(w_\eps - \sfPi_\B w_\eps)}$.
According to the definition of ${\mathsf E_\A^\eps} u$, we have
	\begin{equation*}
		\norm{{\mathsf E_\A^\eps} u_\eps} _{L^p(E;\RN)}
		= \norm{\B \big({\mathsf E^\eps} (w_\eps - \sfPi_\B w_\eps) \big)} _{L^p(E;\RN)}
		\leq c \norm{\nabla^\ell \big({\mathsf E^\eps} (w_\eps - \sfPi_\B w_\eps)\big)}_{L^p(E;\R^{M\times d^\ell})},
	\end{equation*}
whence, owing to point (4) in Proposition \ref{thm:extension-classic},
    \[
        \norm{\mathsf{E}_\A^\eps u_\eps}_{L^p(E;\RN)} < \eta
    \]
if $\mathscr L^d(E)$ is sufficiently small, as desired.
\end{proof}

In light of our analysis,
existence of potentials and of extension operators turn out
to be almost equivalent;
a key role in this respect is played by \eqref{eq:genKorn}.
We elaborate on this point in the next remark.

\begin{remark}[Relations between existence of potentials and of extension operators]
    \label{stm:conclusions}
    Let us compare the main results of the current section.
    
    We let the operator $\A$ be as above,
    notably we assume that it has constant rank.
    Thanks to Theorem \ref{stm:raita-bis},
    we know that
    if $D\subset \Rd$ is a bounded, connected, open set with Lipschitz boundary
    which is also an $\A$-extension domain in the sense of Definition \ref{stm:ext-dom},
    then $\A$-free maps on $D$ admit potentials.
    In short,
    for the class of operators under consideration and for sufficiently `nice' open sets $D$,
    it holds
    \begin{equation}\label{eq:extension=>potential}
        D \text{ is an $\A$-extension domain} \implies \text{existence of potentials for $\A$-free fields on } D.
    \end{equation}
    
    Conversely,
    if $D\subset \Rd$ is an open bounded set with Lipschitz boundary
    such that
    all $\A$-free fields on $D$ admit Sobolev potentials through the operator $\B$
    and if for the latter \eqref{eq:genKorn} holds,
    then $D$ is an $\A$-extension domain.
    Indeed, by a slight adaptation of the proof,
    a variant of Theorem \ref{stm:Aext} for the case $O=\Rd$
    can be established.
    Schematically, again for sufficiently `nice' open sets $D$,
    we have
    \begin{equation*}
        \left.
        \begin{aligned}
        \text{existence of potentials for $\A$-free fields on } D \\
        \text{Korn-type inequality for $\B$ on } D
        \end{aligned}
        \right\}
        \implies D \text{ is an $\A$-extension domain}.
    \end{equation*}
    
    All in all, we see that,
    given a constant rank operator $\A$ and
    a bounded, connected, open set with Lipschitz boundary $D$,
    the existence of Sobolev potentials for $\A$-free maps on $D$
    and the existence of an $\A$-free extension operator from $D$ to $\Rd$
    are nearly equivalent.
    More specifically,
    they would actually be equivalent
    as soon as we knew that
    the generalised Korn inequality \eqref{eq:genKorn} holds
    when $D$ is a `nice' open set and $\B$ is a constant rank operator.
    In conclusion, we believe that
    investigations about the validity of \eqref{eq:genKorn} constitute a very interesting line of research.
    \end{remark}

We conclude with a parade of examples.

\begin{example}[Curl]\label{ex:curl}
    For the choice $\A=\mathrm{curl}$,
    we have (classically) $d=N=M=3$ and
    \[
        \mathrm{curl}\, u = \sum_{i=1}^{3} A^{(i)}\pder{u}{x_i},
    \]
    with
    \[
        A^{(1)} \coloneqq
        \begin{pmatrix}
            0 & 0 & 0 \\
            0 & 0 & -1 \\
            0 & 1 & 0 
        \end{pmatrix},
        \quad
        A^{(2)} \coloneqq
        \begin{pmatrix}
            0 & 0 & 1 \\
            0 & 0 & 0 \\
            -1 & 0 & 0 
        \end{pmatrix},
        \quad
        A^{(3)} \coloneqq
        \begin{pmatrix}
            0 & -1 & 0 \\
            1 & 0 & 0 \\
            0 & 0 & 0 
        \end{pmatrix}.
    \]
    The symbol of $\mathrm{curl}$ is then
    \[
        \bA[\omega] = \begin{pmatrix}
            0 & -\omega_3 & \omega_2 \\
            \omega_3 & 0 & -\omega_1 \\
            -\omega_2 & \omega_1 & 0 
        \end{pmatrix}
    \]
    and $\mathrm{rank}\,\mathrm \bA[\omega] = 2$ for all $\omega\in \Rdmz$.
    
    It is easy to check that Assumption \ref{stm:null-av} holds.
    Besides, when $\Omega$ is simply connected,
    in view of Proposition \ref{thm:extension-classic}
    all the conditions in Assumption \ref{stm:exist-ext} are fulfilled too.
\end{example}

\begin{example}[Operators associated with higher-order gradients]
    Let $\Omega$ be simply connected.
    For any $k\in\N\setminus\set{0}$,
    a constant rank differential operator $\A$ can be constructed
    such that $\A u = 0$ if and only if $u=\nabla^k w$ for a suitable $w$ \cite{fonseca.muller}.
    Then, similarly to the previous example,
    for such operator $\A$ Assumptions \ref{stm:null-av} and \ref{stm:exist-ext}
    are consequences, respectively, of a simple check and of Proposition \ref{thm:extension-classic}.
    \end{example}

\begin{example}[The $\mathrm{curl\,curl}$ operator]
Let $d=M=N=3$ and,
for all $i,j=1,2,3$,
let $E^{(i,j)}$ be the matrix whose entries all $0$,
but for the one in position $(i,j)$, which equals $1$.
We consider the operator $\A={\rm curl\,curl}$, defined as
\[
    \mathrm{curl\,curl}\, u \coloneqq \sum_{i,j=1}^{3} A^{(i,j)}\frac{\partial^2 u}{\partial x_j \partial x_i},
\]
where
$A^{(i,i)} = - E^{(i+1,i+1)} - E^{(i+2,i+2)}$ (the indices are computed modulo $3$) and $A^{(i,j)} = E^{(i,j)}$.
The symbol of this operator is
\[
        \bA[\omega] = \begin{pmatrix}
            -\omega^2_2 -\omega^2_3 & \omega_1 \omega_2 & \omega_1 \omega_3 \\
            \omega_1 \omega_2 & -\omega^2_1 -\omega^2_3 & \omega_2 \omega_3 \\
            \omega_1 \omega_3 & \omega_3 \omega_3 & -\omega^2_1 -\omega^2_2
        \end{pmatrix}
    \]
    and $\mathrm{rank}\,\mathrm \bA[\omega] = 2$ for all $\omega\in \Rdmz$.

When $\Omega\subset \R^3$ is a bounded and simply connected domain with Lipschitz boundary, we have that
$\A u=0$ if and only if $u=\B w$ for a suitable potential $w$, where $\B$ is the symmetric gradient $\B w\coloneqq (\nabla w+(\nabla w)^\mathsf{t})/2$.
As a corollary of a recent result by {\sc F. Cagnetti \& al.} \cite[Theorem 1.1]{chambolle.perugini}, every such $\Omega$ is a $\mathrm{ curl\,curl}$-extension domain,
whence Assumption \ref{stm:null-av} is satisfied.
{\RRR Classical Korn's inequality grants also that Assumption \ref{stm:exist-ext} holds.}
\end{example}
    
\begin{example}[Divergence]\label{ex:div}
    Let us choose $\A=\div$,
    $\div$ being the standard divergence operator on $\Rd$.
    Then, $N=d$, $M=1$, and
    \[
        \div u = \sum_{i=1}^d e_i^{\mathsf t}\cdot\pder{u}{x_i},
    \]
    where, for $i=1,\dots,d$,
    $e_i$ is the $i$-th
    element of the canonical basis of $\Rd$ and
    $e_i^{\mathsf t}$ is its transpose.
    The symbol of $\div$ is
    \[
        \bA[\omega] = \sum_{i=1}^d \omega_i e_i^{\mathsf t},
    \]
    thus $\mathrm{rank}\,\mathrm \bA[\omega] = 1$ for all $\omega\in \Rdmz$.
    For what concerns Assumption \ref{stm:null-av} and the existence of extension operators,
    we resort to a result by {\sc T. Kato \& al.},
    which we present in a simplified setting:
    
    \begin{proposition}[Extensions and potentials for divergence-free vector fields \cite{kato.mitrea}]
    Let $\Omega\subset \Rd$ be a bounded and simply connected set with Lipschitz boundary.
    Then, the following holds:
    \begin{enumerate}
    \item $\Omega$ is a $\div$-extension domain in the sense of Definition \ref{stm:ext-dom}.
    \item For every $u\in L^p(\Omega;\Rd)$ such that $\div u = 0$ in $\Omega$,
        there exists $w\in W^{1,p}(\Omega;\mathrm{Antisym}(d\times d))$ satisfying
        \[
        u = \mathrm{Div}\, w \coloneqq \sum_{i,j} \pder{w_{i,j}}{x_j} e_i
        \quad\text{in } \Omega,
        \]
        where $\mathrm{Antisym}(d\times d)$ is the space of $d \times d$ antisymmetric matrices.
        Besides, $w$ can be selected in such a way that
        the map $u\mapsto w$ from $L^p(\Omega;\Rd)$ to $W^{1,p}(\Omega;\Rd)$ is linear and bounded.
    \end{enumerate}
    \end{proposition}
    
    It is interesting to notice that
    the authors of \cite{kato.mitrea} derive item (2) from (1), that is,
    they prove an implication of the form \eqref{eq:extension=>potential}.
    
    {\RRR It is proved in \cite[Subsection 4.1]{arroyo2} that $\mathrm{Div}$ satisfies \eqref{eq:genKorn-O}.
    Hence, Theorem \ref{stm:Aext} holds for $\A=\mathrm{div}$ and
    Assumption \ref{stm:exist-ext} is fulfilled too.
    }
\end{example}

\section*{Acknowledgements}
We are thankful to the anonymous referee for their comments,
that helped us improve the quality of the manuscript.
We acknowledge support
from the Austrian Science Fund (FWF) projects F65, V 662, Y1292,
from the FWF-GA\v{C}R project I 4052/19-29646L, and  
from the OeAD-WTZ project CZ04/2019 (M\v{S}MT\v{C}R 8J19AT013).
Data sharing is not applicable to this article,
as no datasets were generated or analysed during the current study.


\end{document}